\documentclass{siamonline190516}
\usepackage{a4,latexsym,exscale,theorem,epsfig,amsmath}
\usepackage{amssymb,psfrag,epsf,amsmath,verbatim,bbm,float}
\usepackage{listings}
\usepackage{color}
\usepackage{mathbbol}
\usepackage{enumerate}
\newtheorem{remark}[theorem]{Remark}
\newtheorem{example}[theorem]{Example}

\usepackage{algpseudocode}
\usepackage{geometry}
\usepackage{accents}
\usepackage{todonotes}
\usepackage{caption}
\DeclareMathOperator*{\esssup}{ess\sup}
\DeclareMathOperator*{\essinf}{ess\inf}

\floatstyle{plain}
\newfloat{diagram}{hbt}{lop}[section]
\floatname{diagram}{Diagram}
\makeatletter
\let\c@table\c@table
\let\c@diagram\c@table
\makeatother
\begin{document}
\makeatletter
\let\c@algorithm\c@theorem
\makeatother
\newcommand {\eps} {\varepsilon}
\newcommand {\Z} {\mathbbm{Z}}
\newcommand {\R} {\mathbbm{R}}
\newcommand {\N} {\mathbbm{N}}
\newcommand {\Q} {\mathbbm{Q}}
\newcommand {\C} {\mathbbm{C}}
\newcommand {\I} {\mathbbm{I}}
\newcommand {\PP} {\mathbbm{P}}
\newcommand {\ang} {\measuredangle}
\newcommand {\e} {{\rm{e}}}
\newcommand {\rank} {{\rm{rank}}}
\newcommand {\Span} {{\mathrm{span}}}
\newcommand {\card} {{\rm{card}}}
\newcommand {\ED} {\mathrm{ED}}
\newcommand {\cA} {\mathcal{A}}
\newcommand {\cO} {\mathcal{O}}
\newcommand {\cF} {\mathcal{F}}
\newcommand {\cN} {\mathcal{N}}
\newcommand {\cV} {\mathcal{V}}
\newcommand {\cG} {\mathcal{G}}
\newcommand {\cB} {\mathcal{B}}
\newcommand {\cD} {\mathcal{D}}
\newcommand {\cP} {\mathcal{P}}
\newcommand {\cQ} {\mathcal{Q}}
\newcommand {\cW} {\mathcal{W}}
\newcommand {\cT} {\mathcal{T}}
\newcommand {\cI} {\mathcal{I}}
\newcommand {\Sn}[1] {\mathcal{S}^{#1}}
\newcommand {\range} {\mathcal{R}}
\newcommand {\kernel} {\mathcal{N}}
\newcommand{\one}{\mathbb{1}}
\renewcommand{\thefootnote}{\fnsymbol{footnote}}
\newcommand{\rle}{\rotatebox[origin=c]{-90}{$\le$}}
\newcommand{\rl}{\rotatebox[origin=c]{-90}{$<$}}
\newcommand{\rg}{\rotatebox[origin=c]{-90}{$=$}}

\newcommand*\flag[1]{{\textbf{*** #1 ***}}}


\title{\bf Angular values of nonautonomous and random linear dynamical systems:
     Part I -- Fundamentals}
\author{Wolf-J\"urgen Beyn\footnotemark[1]\qquad
  Gary Froyland \footnotemark[2] \qquad
  Thorsten H\"uls\footnotemark[1]
}
\footnotetext[1]{Department of Mathematics, Bielefeld University,
  33501 Bielefeld, Germany \\
\texttt{beyn@math.uni-bielefeld.de}, \texttt{huels@math.uni-bielefeld.de}}
\footnotetext[2]{School of Mathematics and Statistics, University of
New South Wales, Sydney NSW 2052, Australia\\ \texttt{g.froyland@unsw.edu.au}}

\maketitle


\begin{abstract}
  We introduce the notion of angular values
  for deterministic linear difference equations and random linear cocycles.
  We measure the principal angles between subspaces of
  fixed dimension as they evolve under nonautonomous
   or random linear dynamics.
  The focus is on long-term averages of these principal angles, which we
  call angular values:  we demonstrate
   relationships between different types of angular values
  and prove their existence for random
  dynamical systems.
  For one-dimensional subspaces in two-dimensional systems our
  angular values agree with the classical theory of rotation numbers for
  orientation-preserving circle homeomorphisms if the matrix has
  positive determinant and does not rotate vectors by more than $\frac{\pi}{2}$.
  Because our notion of angular values ignores orientation by looking at subspaces rather than
  vectors, our results apply to dynamical systems of any dimension and to
  subspaces of arbitrary dimension.  The second part of the paper delves deeper into the theory of the autonomous case.
   We explore the relation to (generalized) eigenspaces, provide
  some explicit formulas for angular values, and set up a general numerical
  algorithm for computing angular values via Schur decompositions.
\end{abstract}



\section{Introduction}
\label{sec0}
\renewcommand{\thefootnote}{\arabic{footnote}}
\setcounter{footnote}{0}
In this paper we propose and analyze suitable notions of angular values
 for linear nonautonomous discrete-time dynamical systems.
The systems are of the form
\begin{equation}\label{diffeq}
  u_{n+1} = A_n u_n,\quad u_0 \in \R^d, \quad  n \in \N_0
\end{equation}
with  $A_n \in \mathrm{GL}(\R^{d})$, i.e.\ with real invertible $d\times d$ matrices
$A_n,n \in \N_0$.
  Our goal is to study  the average rotation of
  $s$-dimensional  subspaces $V_0 \subseteq \R^d$ for $s=1,\ldots,d$
when iterated as in  \eqref{diffeq}, i.e.\ we consider
the sequence of subspaces generated by
\begin{equation} \label{diffV}
  V_{n+1}= A_n V_n, \quad n \in \N_0,
\end{equation}
so that $V_{n+1}=V_{n+1}(V_0)$ depends on $V_0$ via $V_{n+1}=A_{n}A_{n-1}\cdots A_1A_0V_0$.
Since the matrices $A_n$ are invertible the  subspaces $V_n$ have
the same dimension $s$ for all $n \in \N_0$. Their rotation is
measured by the well-established notion of principal angles between
subspaces which originates with C. Jordan in 1876.
By $\ang(V,W)$ we denote the maximum principal angle of two subspaces $V,W$ and
we recall that $0 \le \ang(V,W)\le \frac{\pi}{2}$ holds. Some basics of the
theory of principal angles and of their numerical computation may be found in
\cite{Meyer2000}, \cite[Ch.6.4]{GvL2013}. Generalizations to complex
vector spaces and the triangle inequality appear in the papers
\cite{Ga2006}, \cite{js96}, \cite{ZK2013}.
In  Section \ref{sec1} we derive some specific results, tailored to our
needs, such as estimates of principal angles in terms of norms and an angle bound
for linear maps.
Using principal angles between successive spaces $V_{j-1}$ and $V_{j}$ generated by
\eqref{diffV} we form the $n$-step average
\begin{equation} \label{average}
  \frac{1}{n}a_{1,n}(V_0), \quad \text{where} \quad a_{1,n}(V_0)=
  \sum_{j=1}^n \ang(V_{j-1},V_{j}), \quad n \ge 1
\end{equation}
and two types of limiting values
\begin{equation} \label{angularvalues}
  \begin{aligned}
    \bar{\theta}_s &= \limsup_{n \to \infty} \sup_{V_0\in  \mathcal{G}(s,d)}
    \frac{1}{n}a_{1,n}(V_0),
    \qquad \hat{\theta}_s =  \sup_{V_0\in  \mathcal{G}(s,d)} \limsup_{n \to \infty}
    \frac{1}{n}a_{1,n}(V_0),
  \end{aligned}
\end{equation}
 where $\mathcal{G}(s,d)$ denotes the Grassmann manifold of
 $s$-dimensional subspaces of $\R^d$.
 We call
 $\bar{\theta}_s$  the $s$-inner and $\hat{\theta}_s$ the $s$-outer
 angular value of the system \eqref{diffeq}.
  In sections \ref{sec2}-\ref{sec3} we
   will discuss systems for which the $\limsup$s in \eqref{angularvalues}
   are actually limits.
 More variations of these notions will
be defined in Section \ref{sec2.1}, and
some key examples will be presented in  Section \ref{sec3.2}  which
show that all types of angular values differ in general.

As a physical motivation of angular values
  consider some object, such as a small massless rod or a sheet,
  carried materially by a time-varying fluid flow,
   and assume that data about its position and orientation
  are available at discrete time instances. The task then is
   to measure the  maximum average rotation
  of the object. In mathematical terms we think of a continuous time
  dynamical system determining its trajectory,
  and we assume that the system \eqref{diffeq}  describes its linearization
  about the trajectory  when sampled at discrete times. Then the first and second outer
  angular values measure the maximum average angle of rotation
  exerted by the flow on a line
  ($s=1$) or on a plane ($s=2$).
  Rotations of subspaces $s\ge 3$ may be relevant in higher-dimensional phase spaces.
In view of such  applications
it is natural to extend the quantities \eqref{angularvalues}
to continuous-time systems. A short discussion of such an extension is
given in the outlook of this article.

Perhaps the simplest example is a $2\times 2$ orthogonal matrix, where $d=2$,
$s=1$ and
\begin{equation} \label{rotation}
  A_n\equiv A=T_{\varphi}=\begin{pmatrix}
  \cos\varphi & -\sin\varphi \\
  \sin\varphi & \cos\varphi \end{pmatrix}, \quad 0 \le \varphi \le \frac{\pi}{2}.
\end{equation}
All summands in \eqref{average} are $\varphi$ and $a_{1,n}(V_0)=n
\varphi$ for all one-dimensional $V_0\subset \R^2$. Hence we find
  $\bar{\theta}_1= \hat{\theta}_1= \varphi$ in this case.

 A first motivating example  is the following randomized version of \eqref{rotation}.
 Let $(\Omega,\PP)$ be a probability space, $\tau:\Omega\to\Omega$ be
 an ergodic transformation preserving $\PP$ and $\varphi:\Omega \to
 [0,\frac{\pi}{2}]$ be
 a random variable.
Setting $A(\omega)=T_{\varphi(\omega)}$
and $A_n=A(\tau^n\omega_0)$ for some $\omega_0\in\Omega$ we see that
$a_{1,n}(V_0)=\sum_{j=0}^{n-1}\varphi(\tau^j\omega_0)$ for every
$V_0$.
By Birkhoff's ergodic theorem, for $\PP$-almost every $\omega_0$, one
has $\lim_{n\to\infty}\frac{1}{n} a_{1,n}(V_0)=\int \varphi(\omega)\
\mathrm{d}\PP(\omega)$.
The above general formula holds for driving systems $\tau$
modeling any ergodic stationary deterministic or stochastic process.
 In Section \ref{sec3} we generalize the various notions  of angular values
 to the general setting of random dynamical systems (cf.\ \cite{A1998}).
 We establish their existence via ergodic theorems and
 prove inequalities between the various types; see Theorem \ref{thmRDS}.

A second motivating example abandons orthogonality
and changes \eqref{rotation} by a skewing factor $0<\rho \le 1$  to
\begin{equation*} \label{skew2dim}
 A_n\equiv  A(\rho,\varphi)= \begin{pmatrix} \cos(\varphi) & - \rho^{-1} \sin(\varphi)\\
    \rho \sin(\varphi) & \cos(\varphi) \end{pmatrix},  \quad
   0\le \varphi \le \frac \pi 2.
\end{equation*}

This matrix turns out to be a kind of normal form with regard to measuring
angles between a one-dimensional subspace and its image (see
Proposition \ref{lem2.3}). The angular values $\hat{\theta}_1$ and
$\bar{\theta}_1$ agree in this case,
but they differ from $\varphi$ in general and depend critically on the value
of $\rho$ (see Proposition \ref{lem2.3} and Theorem \ref{prop5.1}).

There is a weak analogy of first angular values to Lyapunov exponents
  which measure the maximum average exponential growth of a linear
  nonautonomous system \eqref{diffeq}; see e.g.\ \cite[Ch.3.2]{A1998},
  \cite{barreira2017}, \cite[Suppl.2]{KH95}).
  For the latter purpose it is enough to compare the norm
  of the last iterate with the first one and average the logarithm.
  However, in the angular direction one expects only linear growth
  which requires one to calculate an arithmetic average over every single time step.

For certain systems, the above definition of angular values is related to existing
concepts of measuring rotations in dynamical systems, which we now discuss.
We first mention the classical theory of rotation numbers for
orientation-preserving homeomorphisms of the circle, cf.\
\cite{dMvS1993}, \cite[Ch.11]{KH95}, \cite{N1971}. If the system
  \eqref{diffeq} is two-dimensional and autonomous (i.e.\ $A_n\equiv A
  \in \mathrm{GL}(\R^2)$), then it generates a homeomorphism of the
  unit circle, which is orientation-preserving for $\det(A)>0$. If, in
  addition, no vector rotates by an angle
  greater than $\frac{\pi}{2}$, then the rotation number agrees (up to a factor
  of $2 \pi$) with the
  first angular value; see Section \ref{sec5a}, Remark \ref{rotnumber}
  and Proposition \ref{lem2.3} for more details.
However, such a comparison is no longer possible  for a reflection or
for matrices which generate rotations of vectors with angles larger than $\frac{\pi}{2}$.
By contrast to rotation numbers, our definition \eqref{angularvalues}
avoids assuming or
specifying any orientation, even when one observes the motion of
one-dimensional subspaces (rather than vectors) in a two-dimensional space.
Including orientation typically leads to complications in discrete-time systems.
For example, for rotations that are close to reflections one needs
extra analytic information from the system (such as $\det(A_n)$),
which we consider as inaccessible to observation.
When rotations of vectors larger than $\frac{\pi}{2}$ occur, our definition
takes the smaller of both possible angles; Figure \ref{hen0} illustrates
this for a sequence of subspaces.
Note that angles between successive subspaces are indicated by black
arcs with time progressing outward.

\begin{figure}[hbt]
 \begin{center}
   \includegraphics[width=0.35\textwidth]{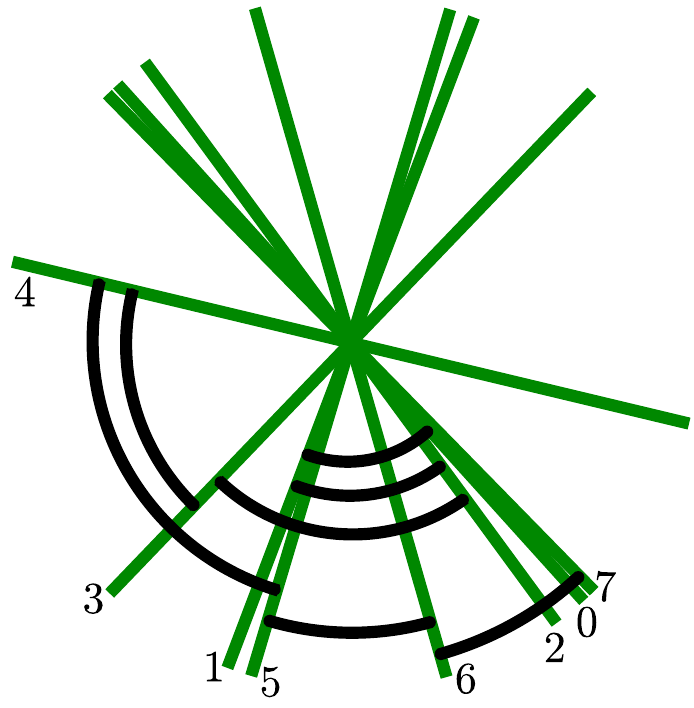}
 \end{center}
 \caption{\label{hen0}Angles between successive subspaces in the H\'{e}non system (\ref{var2}).}
\end{figure}

The theory of rotation numbers for homeomorphisms of the circle
  has been generalized to so-called rotation sets of toral automorphisms
  in \cite{MZ1989}, and a numerical approach appears in \cite{PPZ2017}.
  However,
  there seems to be  no connection to the definition \eqref{angularvalues} in higher
  dimensions.

  Another far-reaching extension of rotation numbers to nonautonomous
  continuous time systems of arbitrary dimension
  has been proposed and investigated in \cite{AS1989}, \cite[Ch.6.5]{A1998}.
   The average rotation of vectors is measured within all two-dimensional subspaces
   (more generally within tangent planes of a manifold) mapped by the system.
   Orientation is taken into
   account where counterclockwise refers to positive values.
   In essence one studies the flow induced  by the
   given system on the Grassmannian $\mathcal{G}(2,d)$. The concept
   generalizes  to nonlinear random dynamical systems and even leads to
   a multiplicative ergodic theorem, see \cite[Th.6.5.14]{A1998}.
   However, the conceptual difference to angular values remains the same as for the classical
   rotation numbers.

       Yet another concept of rotation numbers has been developed for continuous
   time linear Hamiltonian systems of arbitrary dimension; see \cite{j87},
   \cite{jn94}, with a route from the theory to numerical results provided
   in \cite{ef13}. The notion is based on a suitable generalization of
   the $\mathrm{arg}$-function from a scalar complex system to the even dimensional real
   case. Then the rotation number appears as the limit of the time
   average of the $\mathrm{arg}$-function when applied to a symplectic
   fundamental matrix. The setting is similar to the  random dynamical systems
   mentioned above. The resulting rotation number has interesting relations to the
   dichotomy spectrum of a parametrically perturbed Hamiltonian system;
   see \cite[Theorem 4-6]{ef13}. This notion differs from
   the first angular values of this paper  since time is continuous and orientation is
   taken into account by the choice of the $\mathrm{arg}$-function.

    Let us also mention the notion of antieigenvalues and antieigenvectors
   developed in \cite{Gu2012}. They are determined by the maximum angle
   $\ang(v,Av)$ by which a given matrix $A$ can turn a vector $v \in \R^d$.
   This corresponds to maximizing the first summand in \eqref{average},
   but ergodic averages seem not to have been considered in this theory.

   In the following we
    summarize some further results of this paper. In Section
   \ref{sec2.1} we collect elementary properties
   of angular values, such as inequalities among them and invariance under
   special kinematic similarities, see \cite{GKK96} for this notion.
Section \ref{sec5} presents an in-depth study of the autonomous case $A_n \equiv A$.
   The main theoretical result is Theorem \ref{cor2.4} which
   reduces the computation of angular values to the case of a block-diagonal matrix.
   Theorem \ref{cor2.4}
     builds on a spectral decomposition (Blocking Lemma \ref{lemblock}), on
     a special treatment
   of multiple real eigenvalues (Proposition \ref{prop2.5}), and on a detailed
   analysis of the two-dimensional case (Proposition \ref{lem2.3}).
   In the two-dimensional case we show that all types of first angular
     values coincide  and provide a rather explicit
     formula (Proposition \ref{lem2.3}, Theorem \ref{prop5.1}). While real
     eigenvalues of the matrix lead to a vanishing angular value,
     complex conjugate ones lead to interesting
     resonances depending on a skewness parameter; see Figures \ref{reso},
     \ref{reso3d}. In the latter case we use ergodic theory to derive an integral expression
     for the first angular value when rotation occurs with irrational
     multiples of $\pi$, and we reduce the computation to maximizing a
     finite sum in the rational
     case.
      In Section \ref{sec6} we present a numerical
   algorithm for the autonomous case  based on eigenvalue computations
   and one-dimensional optimization which avoids failure caused
   by  simple forward iteration.
   We apply the algorithm to study various systems up to
     dimension $10^4$, and we
    confirm numerically the rather subtle behavior in the
    two-dimensional complex conjugate case.


\section{Angles of subspaces}
\label{sec1}
In this section we collect some useful results about principal angles between
subspaces.
In the following, let
$\|v \| = \sqrt{v^{\top}v}$ denote the Euclidean norm for $v \in \R^d$
and let $\range(A)$, $\kernel(A)$ and $\sigma(A)$ denote the range, the kernel
and the spectrum of a matrix $A$.
Recall the definition of principal angles and principal
 vectors of two subspaces $V,W$ of $\R^d$ of equal dimension from
\cite[Ch.6.4.3]{GvL2013}.
\begin{definition} \label{def1}
Let $V,W$ be subspaces of $\R^d$ of dimension $s$.
Then the principal angles $0\le \phi_1 \le \ldots \le \phi_{s}\le
\frac{\pi}{2}$ and associated principal vectors $v_j\in V$, $w_j \in W$
 are defined recursively for $j=1,\ldots,s$ by
\begin{equation} \label{eq1.0}
  \cos(\phi_j)= \max_{\substack{v\in V,
                         \|v\|=1 \\ v^{\top}v_{\ell}  =0,
      \ell=1,\ldots,j-1}}\
\max_{\substack{w\in W, \|w\|=1 \\ w^{\top}
                    w_{\ell} =0,\ell=1,\ldots,j-1}} v^{\top}w
=v_j^{\top} w_j.
\end{equation}
\end{definition}
The right-hand side of \eqref{eq1.0} lies in $[0,1]$,
so that $\phi_j \in [0, \frac{\pi}{2}]$ is uniquely defined by \eqref{eq1.0}.
While principal angles are unique,  principal vectors are not, in
general. Let us note that principal angles and principal vectors are
also defined for subspaces of different dimension
  (see \cite[Ch.6.4.3]{GvL2013}), but this feature will not be used due
  to our assumption of invertibility.
We further write $\phi_j=\phi_j(V,W)$ to indicate the dependence on the
subspaces, and for the largest angle we introduce the notation
\begin{equation*} \label{eq1.a}
\phi_{s}(V,W)= \ang(V,W).
\end{equation*}
If the subspaces $V$ and $W$ are one-dimensional we may write
\begin{align*}
  \ang(v,w) = \ang(\mathrm{span}(v),\mathrm{span}(w)), \quad v,w \in \R^d,
  v,w \neq 0.
\end{align*}
Let us also note that the usage of the angle between subspaces varies
in the literature. For example, in \cite[(3.3.13)]{Ad1995}, \cite[p.216]{A1998} this notion is
  used for $\sin(\phi_1)$ where $\phi_1$ is the smallest angle. Then
  \eqref{eq1.0} turns into a min-min characterization, and the angle
  becomes zero if both subspaces share a common direction.

Principal values and vectors
can be computed from a singular value decomposition (SVD) as follows.
\begin{proposition}(\cite[Algorithm 6.4.3]{GvL2013}) \label{prop1}
Let $P, Q \in \R^{d,s}$ be two matrices with orthonormal columns and consider the SVD
\begin{equation} \label{eq1.b}
P^{\top}Q = Y \Sigma Z^{\top}, \quad Y,Z,
\Sigma=\mathrm{diag}(\sigma_1,\ldots,\sigma_{s}) \in \R^{s,s}, \; Y^{\top}Y=I_s=Z^{\top}Z.
\end{equation}
Then the principal angles $\phi_j$ of $V=\range(P)$ and
$W= \range(Q)$ satisfy
\begin{equation*} \label{eq1.c}
\sigma_j = \cos(\phi_j), j=1,\ldots, s,
\end{equation*}
and principal vectors are given by
\begin{equation} \label{eq1.d}
PY = \begin{pmatrix}  v_1 & \cdots & v_{s} \end{pmatrix}, \quad
QZ = \begin{pmatrix}  w_1 & \cdots & w_{s} \end{pmatrix}.
 \end{equation}

\end{proposition}
Since the singular values of $P^{\top}Q$ and $Q^{\top}P$ agree,
  principal angles are symmetric with respect to $V$ and $W$. In
  particular, the maximum angle satisfies
\begin{equation*} \label{eq1.z}
  \ang(V,W) = \ang(W,V).
\end{equation*}
In Definition \ref{def1} the angles between two subspaces of equal
dimension are defined recursively.
For the computation of the $j$-th principal angle, the
  max-max characterization \eqref{eq1.0} requires knowledge of the
  principal vectors from index $1$ to $j-1$. In the following
  proposition we state a complementary min-max characterization.
  It begins with $\phi_s$ and computes $\phi_j$ via the known
  principal vectors for indices $s$ to $j+1$.
The result is motivated by the  Hausdorff semi-distance between unit balls
and proves to
be  better suited for the key estimates below.
The proof will be given in the Supplementary materials \ref{app1}.
  \begin{proposition}\label{Lemma2}
  Let $V, W\subseteq \R^d$ be two $s$-dimensional subspaces.
  Then the principal angles and principal vectors satisfy for $j=s,\ldots,1$
  \begin{equation} \label{A2}
    \cos(\phi_j)=
    \min_{\substack{v\in V, \|v\|=1
            \\v^{\top}v_{\ell} =0,\ell=j+1,\ldots,s }}\
\max_{\substack{w\in W, \|w\|=1 \\w^{\top} w_{\ell} =0,\ell=j+1,\ldots,s}} v^{\top}w
=v_j^{\top} w_j.
  \end{equation}
  In particular, the following relation holds
  \begin{equation}\label{A1}
    \ang(V,W) = \phi_s(V,W)=
    \max_{\substack{v\in V \\ v\neq 0} }\min_{\substack{w\in W \\ w \neq 0}} \ang(v,w)
  = \arccos\big(\min_{\substack{v\in V\\\|v\|=1}} \max_{\substack{
      w\in W\\\|w\|=1}} v^{\top} w\big).
  \end{equation}
\end{proposition}
\begin{remark}
  A related variational characterization appears in
    \cite[Theorem 3]{RW2003}
  \begin{equation*}
    \cos(\phi_j) = \min_{\substack{U\subseteq V \\\dim U = j-1}}
    \max_{\substack{x\in U^{\perp}\cap V, \|x\|=1 \\ y \in W,
        \|y\|=1}}
  |\langle x, y \rangle |.
  \end{equation*}
  If $j=s$ then $\dim U = s-1$ and $x \in U^{\perp}\cap V$ runs
  through $V$ with $\|x\|=1$. Therefore, the formula implies
  \eqref{A2} in the case $j=s$, but
 for  $j <s$ the formulas differ.
\end{remark}

Next we  recall some well-known properties of the Grassmannian,
\begin{equation*} \label{eq1.4}
  \mathcal{G}(s,d) = \{ V \subseteq \R^d \; \text{is a subspace of dimension}
  \; s \},
\end{equation*}
 which may be found in \cite[Ch.6.4.3]{GvL2013}, \cite{js96},
  for example.

\begin{proposition} \label{prop2g}
  The Grassmannian $\mathcal{G}(s,d)$ is a compact smooth manifold of
  dimension $s(d-s)$ and a metric space with respect to
  \begin{equation*}
    d(V,W) = \| P_V - P_W \|,
  \end{equation*}
  where $P_V,P_W$ are the orthogonal projections onto $V$ and $W$, respectively,
  and the formula
\begin{align*}
  d(V,W) = \sin (\ang(V,W)), \quad V,W \in \mathcal{G}(s,d)
\end{align*}
holds. Furthermore, $\ang(V,W)$ defines an equivalent metric on
$\mathcal{G}(s,d)$ satisfying $$\frac{2}{\pi} \ang(V,W) \le d(V,W)
  \le \ang(V,W).$$
\end{proposition}
Some useful geometric estimates for angles of vectors and subspaces are the following:

\begin{lemma} (Angle estimates) \label{lem0}
  \begin{enumerate}[(i)]
  \item For any two vectors $v,w \in \R^d$ with $\|v\| < \|w\|$ the following holds
    \begin{equation} \label{eq1.6}
      \begin{aligned}
        \tan^2\ang(v+w,w)& \le \frac{\|v\|^2}{\|w\|^2 - \|v\|^2}, \\
        \cos^2\ang(v+w,w)& \ge  \frac{\|w\|^2-\|v\|^2} { \|w\|^2}.
      \end{aligned}
    \end{equation}
  \item
    Let $V \in \mathcal{G}(s,d)$ and $P \in \R^{d,d}$ be such that
    for some $0 \le q <1$
    \begin{equation} \label{eq1.7}
      \| (I-P)v\| \le q \| Pv \| \quad \forall \  v \in V.
    \end{equation}
    Then $\dim(V)=\dim(PV)$ and the following estimate holds
    \begin{equation} \label{eq1.8}
      \ang(V,PV) \le \frac{q}{(1-q^2)^{1/2}}.
    \end{equation}
  \end{enumerate}
\end{lemma}
\begin{proof}
  The first inequality in \eqref{eq1.6} follows from the second via the relation
  $\tan^2\alpha = \frac{1}{\cos^2 \alpha}-1$. The second
  inequality
  in \eqref{eq1.6} can be rewritten as
  \begin{align*} \big( (v+w)^{\top} w \big)^2- \|v+w\|^2(\|w\|^2 - \|v\|^2) \ge 0 .
  \end{align*}
  A short computation shows that the left-hand side agrees with
  $\big(v^{\top}w + \|v\|^2 \big)^2$ which proves our assertion.
  The estimate \eqref{eq1.7} shows that $Pv=0,v\in V$ implies $v=0$, hence
  $\dim(V)=\dim(PV)$.
  Inequality \eqref{eq1.8} follows from \eqref{eq1.6} and the characterization \eqref{A1}
  \begin{align*}
    \ang(V,PV) & = \max_{\substack{v \in V\\ v\neq 0}} \min_{\substack{w \in PV\\
      w \neq 0}} \ang(v,w) \le
      \max_{\substack{v \in V\\ v \neq 0}}\ang(v,Pv)
      \le \max_{\substack{v \in V\\ v \neq 0}}|\tan \ang(v,Pv)|\\
      &\le \max_{v\in V, v\neq 0}
      \frac{\|(I-P)v\|}{\|Pv\|} \Big(1 -\frac{\|(I-P)v\|^2}{\|Pv\|^2}
      \Big)^{-1/2} \le \frac{q}{(1-q^2)^{1/2}}.
  \end{align*}
\
\end{proof}

\begin{remark}
  The proof shows that the inequalities in \eqref{eq1.6} are strict
  for $v \neq 0$.
\end{remark}
Lemma \ref{lem0} will be important for proving the Blocking Lemma \ref{lemblock}  in the
autonomous case. The next auxiliary result provides an angle-bound for an invertible
matrix; it will be used in Proposition \ref{prop2.5} for treating real eigenvalues in the autonomous case.

\begin{lemma} \label{lem2.6}
  Let $S \in \mathrm{GL}(\R^d)$ and $\kappa=\|S^{-1}\| \|S\|$ be its
  condition number. Then the following estimate holds
  \begin{equation} \label{anglebound}
    \ang(SV,SW) \le \pi \kappa(1+\kappa) \ang(V,W) \quad
      \forall \  V,W \in \mathcal{G}(s,d),\quad 1\le
        s\le d.
  \end{equation}
\end{lemma}
\begin{proof} Let us first prove \eqref{anglebound} for $s=1$. Then we can assume
  $V=\Span(v)$, $W=\Span(w)$ with $\|v\|=\|w\|=1$ and $v^{\top}w \ge
  0$. From Proposition \ref{prop2g}
  we have
  \begin{equation} \label{angnorm}
  \begin{aligned}
    \frac{1}{\pi}\ang(v,w) & \le \frac{1}{2}d(V,W) =\frac{1}{2} \|v v^{\top}- ww^{\top}\|
    =\frac{1}{2}\|(v-w)v^{\top}+ w(v-w)^{\top}\|\\
    & \le \|v-w\|=(2(1-\cos(\ang(v,w))))^{1/2}= 2 \sin(\tfrac{1}{2}\ang(v,w))
  \le \ \ang(v,w).
  \end{aligned}
  \end{equation}
  We apply the first inequality in \eqref{angnorm} to the image spaces and obtain
  \begin{align*}
    \ang(Sv,Sw)& = \ang(\|Sv\|^{-1}Sv,\|Sw\|^{-1}Sw)
    \le \pi \|S\left( \|Sv\|^{-1}v - \|Sw\|^{-1}w \right) \| \\
    & \le \pi \|S\|\left( | \|Sv\|^{-1}-\|Sw\|^{-1}| + \|Sw\|^{-1}\|v-w\|\right)\\
    & \le \pi \|S\| \|Sw\|^{-1} \left( \|Sv\|^{-1} \|S(w- v)\| + \|v-w\|\right).
  \end{align*}
  Now  $\|Sw\|^{-1},\|Sv\|^{-1} \le \|S^{-1}\|$ and the last inequality
  from \eqref{angnorm} lead to
  \begin{align*}
  \ang(Sv,Sw) & \le \pi \kappa (1 + \kappa) \|v-w\| \le \pi \kappa (1 + \kappa)\ang(v,w).
  \end{align*}
    For the general case $s \ge 1$ we use \eqref{anglebound} for all
  vectors $v \in V, v\neq0$, $w \in W, w\neq0$ and then apply the
  $\max$-$\min$ characterization \eqref{A1} from Proposition \ref{Lemma2}.
  \end{proof}

\section{Basic theory of angular values}
\label{sec2}
 For an invertible nonautonomous linear
 system \eqref{diffeq} we define the solution operator $\Phi_A$ by
\begin{equation*} \label{solop}
\Phi_A(n,m) =
  \begin{cases}
   A_{n-1}\cdot\ldots\cdot A_m,& \text{ for } n > m,\\
   I, & \text{ for } n = m, \\
    A_n^{-1} \cdot \ldots \cdot A_{m-1}^{-1},&
  \text{ for } n<m.
  \end{cases}
\end{equation*}
Usually we suppress the dependence on the matrix sequence $A_n, n\in \N_0$ and
simply write  $\Phi=\Phi_A$. However, in Section \ref{sec3}
we consider matrix families  generated by a linear random dynamical system
for which the dependence on the family is essential.

\subsection{Definitions and elementary properties}
\label{sec2.1}
In the following we consider various ways of defining the average
angular rotation that the system \eqref{diffeq} exerts on subspaces of
a fixed dimension. For this we use the notion of  angles of subspaces
from Section \ref{sec1}.

We reconsider a rigid rotation \eqref{rotation} as a simple motivating example,
but now we allow $0 \le \varphi \le \pi$.
For $v\in\R^2, v\neq 0$ and $j\in \N$ one obtains with Proposition \ref{prop1} that
$$
\ang(v, T_\varphi v) = \ang(T^{j-1}_\varphi v, T^j_\varphi v) = \arccos(|\cos(\varphi)|)=
\min(\varphi,\pi-\varphi).
$$
Hence we
obtain for $n\in\N$ the arithmetic mean
$$
\sup_{v\in\R^2} \frac 1n \sum_{j=1}^n\ang(T^{j-1}_\varphi v,
T^j_\varphi  v)
= \sup_{V\in\cG(1,2)} \frac{1}{n} \sum_{j=1}^{n} \ang(\Phi(j-1,0)V,\Phi(j,0)V)=
\min(\varphi,\pi-\varphi)
$$
and the same value for
both types of limits $\sup_{V\in\cG(1,2)} \lim_{n\to \infty}$
  and $ \lim_{n\to \infty}\sup_{V\in\cG(1,2)}$.

For general systems however,
it turns out that the limit does not necessarily commute with the
supremum, and sometimes the limit does not even exist. Therefore, we
introduce several different types of angular values.

\begin{definition} \label{defangularvalues}
  Let the invertible nonautonomous system \eqref{diffeq} be given.
  For every \linebreak $s\in \{1,\ldots,d\}$ define the quantities
   \begin{equation} \label{defsums}
    a_{k+1,k+n}(V) = \sum_{j=k+1}^{k+n} \ang(\Phi(j-1,0)V,\Phi(j,0)V) \quad
   n \in \N,\; k\in \N_0, \; V\in \mathcal{G}(s,d).
    \end{equation}
  \begin{enumerate}[i)]
  \item
    The upper resp.\ lower $s$-th {\bf inner angular value} is defined by
    \begin{equation}\label{dinner}
\begin{aligned}
  \bar \theta_s =  \limsup_{n\to\infty} \frac{1}{n} \sup_{ V \in \mathcal{G}(s,d)}
   a_{1,n}(V),\quad
  \underaccent{\bar}\theta_s = \liminf_{n\to\infty} \frac{1}{n}
  \sup_{ V \in \mathcal{G}(s,d)}
   a_{1,n}(V).
\end{aligned}
    \end{equation}
  \item
    The upper resp.\ lower $s$-th {\bf outer angular value} is defined by
    \begin{equation}\label{douter}
\begin{aligned}
  \hat \theta_{s} = \sup_{V \in \mathcal{G}(s,d)}
  \limsup_{n\to\infty} \frac{1}{n}  a_{1,n}(V), \quad
    \underaccent{\hat}\theta_s =\sup_{V \in \mathcal{G}(s,d)}
  \liminf_{n\to\infty} \frac{1}{n}  a_{1,n}(V).
\end{aligned}
    \end{equation}
    \item
    The upper resp.\ lower $s$-th {\bf uniform inner angular value} is defined by
\begin{equation}\label{duniinner}
\begin{aligned}
  \bar \theta_{[s]} = &
  \limsup_{n\to\infty} \frac{1}{n} \sup_{V \in \mathcal{G}(s,d)}
  \sup_{k\in\N_0}a_{k+1,k+n}(V),\\
  \underaccent{\bar}\theta_{[s]} = &
  \liminf_{n\to\infty} \frac{1}{n}\sup_{V \in \mathcal{G}(s,d)} \inf_{k\in\N_0} a_{k+1,k+n}(V).
\end{aligned}
\end{equation}
  \item
    The upper resp.\ lower $s$-th {\bf uniform outer angular value} is defined by
\begin{equation}\label{duniouter}
\begin{aligned}
  \hat \theta_{[s]} = & \sup_{V \in \mathcal{G}(s,d)}
  \lim_{n\to\infty} \frac{1}{n} \sup_{k\in\N_0}a_{k+1,k+n}(V), \\
  \underaccent{\hat}\theta_{[s]} = & \sup_{V \in \mathcal{G}(s,d)}
  \lim_{n\to\infty} \frac{1}{n} \inf_{k\in\N_0} a_{k+1,k+n}(V).
\end{aligned}
\end{equation}
\end{enumerate}

  \end{definition}

\begin{remark} \label{remtrivial}
  In the case $s=d$, all angular values are zero since the
    invertible system keeps the space $V=\R^d$ fixed and since
  $\ang(\R^d,\R^d) = 0$.
\end{remark}

Our guiding principle in forming these quantities is to
  seek the subspace $V$ which maximizes an angular value. The notions
  of 'upper' and 'lower' are motivated by the possible gap between $\limsup$
  and $\liminf$ while 'outer' and 'inner' result from the noncommuting
  $\lim$ and $\sup$. The corresponding uniform angular
  values (and their 'lower' and 'upper' variants)  become relevant
  when passing from autonomous to  nonautonomous systems;
  see Sections \ref{sec3.2} and \ref{sec5}.

As shorthand, we use
   an up/down bar for  upper/lower inner angular values
   and an up/down hat for upper/lower outer angular values, while their uniform
   equivalents are
  indicated by the bracketed index $[s]$.

Clearly, the $\limsup$ and $\liminf$ in \eqref{dinner}, \eqref{douter},
\eqref{duniinner} are finite due to the boundedness of the angles.
In Section
\ref{sec3} we prove that the $\limsup$ and $\liminf$ in \eqref{dinner}
actually become limits in the setting of random dynamical systems.
Let us further mention that the supremum for both quantities
in \eqref{dinner} can be replaced by a maximum since $a_{1,n}(V)$ depends
continuously on $V$ in the compact space $\mathcal{G}(s,d)$.

In the following lemma we show that the limits in \eqref{duniouter}
always exist and that the $\limsup$ in the definition
\eqref{duniinner} of $\bar \theta_{[s]}$ is in fact a limit. Further,
we  collect some easy relations between the various angular values.
\begin{lemma} \label{lem2.2}
  The limits in the definition \eqref{duniouter} of the uniform outer
  angular values exist in $[0,\frac{\pi}{2}]$ and the $\limsup$ in
  the definition of $\bar \theta_{[s]}$ is a limit.
  Moreover, the relations of Diagram \ref{eq2.4} hold for all $s=1,\ldots,d$.
\begin{diagram}
\begin{equation*}
    \begin{matrix}
    \underaccent{\hat} \theta_{[s]} & \le & \underaccent{\hat}
    \theta_s & \le & \hat \theta_s & \le & \hat \theta_{[s]}\\
    \rle& & \rle && \rle && \rle\\
    \underaccent{\bar} \theta_{[s]} & \le & \underaccent{\bar}
    \theta_s & \le & \bar \theta_s & \le & \bar \theta_{[s]}\\
    \end{matrix}
  \end{equation*}
\caption{Comparison of angular values.\label{eq2.4}}
\end{diagram}

  For the smallest and the largest value in this diagram we have the estimate
  \begin{align} \label{eqvalest}
    \sup_{V \in \mathcal{G}(s,d)}\inf_{k \in \N_0}\ang(\Phi(k,0)V, A_k
  \Phi(k,0)V) \le \underaccent{\hat}{\theta}_{[s]} \le \bar{\theta}_{[s]}
  \le \sup_{V \in \mathcal{G}(s,d)}\sup_{ k \in \N_0}\ang(V,A_kV).
\end{align}
\end{lemma}
\begin{proof}
  For every $V\in \mathcal{G}(s,d)$, the sequence $a_n(V)= \sup_{k \in
    \N_0}a_{k+1,k+n}(V)$ lies in $[0,\frac{ n \pi}{2}]$ and is
  subadditive
  \begin{align*}
    a_{n+m}(V)= & \sup_{k\in \N_0}(a_{k+1,k+n}(V) + a_{k+n+1,k+n+m}(V))\\
    \le & \sup_{k\in \N_0}a_{k+1,k+n}(V)+ \sup_{\kappa \ge n}a_{\kappa+1,\kappa+m}(V)
    \le a_n(V) + a_m(V).
  \end{align*}
  By Fekete's subadditive lemma \cite[Lemma 4.2.7]{FH19} this ensures
  \begin{align*}
    \lim_{n \to \infty} \frac{1}{n}a_n(V) = \inf_{n\in
    \N}\frac{1}{n}a_n(V) \in [0,\frac{\pi}{2}].
  \end{align*}
  In a similar way, the sequence $a_n = \sup_{V\in \mathcal{G}(s,d)}a_n(V)$ turns
  out to be subadditive, which shows that $\limsup=\lim$ for the first quantity
  in \eqref{duniinner}. Further,
  the sequence $\alpha_n(V)= \inf_{k \in \N_0} a_{k+1,k+n}(V)$ turns out to be
  superadditive,
  i.e.\ $\alpha_{n+m}(V) \ge \alpha_n(V) + \alpha_m(V)$ for $n,m \in \N$, and thus
    \begin{align*}
      \lim_{n \to \infty} \frac{1}{n}\alpha_n(V) = \sup_{n\in \N}\frac{1}{n}\alpha_n(V) \in
      [0,\frac{\pi}{2}].
    \end{align*}
Next we prove the inequalities $\underaccent \hat \theta_s \le \hat
\theta_s \le \bar \theta_s \le \bar \theta_{[s]}$,
\begin{align*}
\underaccent \hat \theta_s &= \sup_{V\in\cG(s,d)} \liminf_{n\to \infty}\frac{1}{n}
a_{1,n}(V) \le \sup_{V\in\cG(s,d)} \limsup_{n\to \infty}\frac{1}{n}
a_{1,n}(V) = \hat \theta_s\\
&\le \limsup_{n\to \infty}\frac{1}{n} \sup_{V\in\cG(s,d)} a_{1,n}(V) = \bar
\theta_s \le  \limsup_{n\to \infty}\frac{1}{n} \sup_{V\in\cG(s,d)}
 \sup_{k\in\N_0} a_{k+1,k+n}(V) = \bar \theta_{[s]}.
\end{align*}
The remaining assertions in Diagram \ref{eq2.4} follow in a similar way.
Finally, note that Fekete's lemma leads to the representations
\begin{align*}
   \sup_{V \in \mathcal{G}(s,d)}\sup_{n \in \N} \frac{1}{n}\inf_{k \in \N_0}a_{k+1,k+n}(V)
 = \underaccent{\hat}{\theta}_{[s]} \le \bar{\theta}_{[s]}
 = \inf_{n\in \N}\frac{1}{n} \sup_{V \in \mathcal{G}(s,d)}\sup_{ k \in \N_0}a_{k+1,k+n}(V).
 \end{align*}
The inequalities \eqref{eqvalest} then follow  by
setting $n=1$ in  $\sup_n$ and $\inf_n$.
\end{proof}

We extend the
motivating example \eqref{rotation} and analyze in detail the outer angular values
of the $3$-dimensional system defined by

\begin{equation}\label{mot3d}
A_n = A =
\begin{pmatrix}
\cos(\varphi) & -\sin(\varphi) & 0\\\sin(\varphi) & \cos(\varphi) & 0
\\ 0 & 0 & 2
\end{pmatrix},\quad n \in \N_0,\
0 < \varphi \le \frac \pi 2.
\end{equation}
Denote by $e_j$ the $j$-th unit vector in $\R^3$. For $v\in
\Span(e_1,e_2)$, we get, cf.\ \eqref{rotation}, that $\ang(A^{i-1}v,
A^i v) = \varphi$ for all $i\in\N$. For $v\in\Span(e_3)$ one has $\ang(A^{i-1}v,
A^i v) = 0$, $i \in \N$.
Next, we take a vector with components in
both relevant subspaces. This vector is pushed under iteration with
$A$ towards the most unstable direction $e_3$. Thus, we expect that the
angle between two subsequent iterates converges to $0$. The following
estimate proves that this convergence is indeed geometric.
Consider $v =
\left(\protect\begin{smallmatrix}z\\1\protect\end{smallmatrix}\right)$
with $0\neq z \in \R^2$.
From the triangle inequality and the estimate \eqref{eq1.6} in Lemma
\ref{lem0} we find a constant $C>0$ such that for all $i\in\N$
\begin{align*}
  \ang(A^{i-1}v, A^i v) &= \ang\left(
  \begin{pmatrix}T_\varphi^{i-1} z\\[1mm]2^{i-1}\end{pmatrix},
  \begin{pmatrix}T_\varphi^{i} z\\[1mm]2^{i}\end{pmatrix}
  \right)
  =
  \ang\left(
  \begin{pmatrix}2^{1-i}T_\varphi^{i-1} z\\1\end{pmatrix},
  \begin{pmatrix}2^{-i}T_\varphi^{i} z\\1\end{pmatrix}
  \right)\\
  &\le \ang\left(
  \begin{pmatrix}2^{1-i}T_\varphi^{i-1} z\\1\end{pmatrix},
  \begin{pmatrix} 0\\1\end{pmatrix}
  \right)
 + \ang\left(
  \begin{pmatrix}0\\1\end{pmatrix},
  \begin{pmatrix}2^{-i}T_\varphi^{i} z\\1\end{pmatrix}
    \right)\\
    &\le \tan \ang\left(
  \begin{pmatrix}2^{1-i}T_\varphi^{i-1} z\\1\end{pmatrix},
  \begin{pmatrix} 0\\1\end{pmatrix}
  \right)
 + \tan \ang\left(
 \begin{pmatrix}2^{-i}T_\varphi^{i} z\\1\end{pmatrix},
   \begin{pmatrix}0\\1\end{pmatrix}
    \right)\\
  &  \le C \cdot 2^{-i}.
\end{align*}

Thus
$$
\frac 1n \sum_{i=1}^n \ang(A^{i-1} v,A^i v) \le \frac 1n
\sum_{i=1}^\infty C\cdot 2^{-i} = \frac{2C}n \to 0 \text{ as } n\to
\infty.
$$
As a consequence, all first outer angular values from Definition
\ref{defangularvalues} coincide and have the value $\varphi$, see
Figure \ref{ai} and Theorem \ref{cor2.4} for the inner angular values.

For analyzing the second outer angular values, we first note that for all
$V\in\cG(2,3)$ there exists a $u \in \Span(e_1,e_2)$ such that $V =
\Span(u,v)$ with $v\in\R^3$. Without loss of generality, we assume
that $u=e_1$. We observe for $v \in \Span(e_1,e_2)$ that $a_{1,n}(V)
= 0$ and for $v\in \Span(e_3)$, we obtain $a_{1,n}(V)= \varphi$.
Next, we consider the mixed case  $v=\begin{pmatrix} z_1 & z_2 &
1\end{pmatrix}^{\top}$, with $0\neq z \in \R^2$.
Let $W=\Span (e_1,e_3)$ then we get for $i\in\N$
$$
\ang(A^{i-1}V,A^iV) \le \ang(A^{i-1}V,A^{i-1}W) + \ang(A^{i-1}W,A^iW) + \ang(A^iW,A^iV).
$$
The second term is equal to $\varphi$ for all $i\in\N$.
We
conclude that all second outer angular values coincide with $\varphi$
by showing that the first and third term converge to zero with a
geometric rate.
Note that for $i\in\N_0$ we have
\begin{align*}
A^i V &= \Span(A^ie_1, A^i e_3 +
        A^i \begin{pmatrix}z_1&z_2&0\end{pmatrix}^\top)\\& =
           \Span(A^ie_1,e_3
        +2^{-i} A^i \begin{pmatrix}z_1&z_2&0\end{pmatrix}^\top),\\
A^i W &= \Span(A^i e_1,e_3).
\end{align*}
With $P_i = I+Q_i$, $Q_i = 2^{-i}A^i \begin{pmatrix}0 & 0 & z_1\\ 0 &
  0 & z_2 \\ 0 & 0 & 0\end{pmatrix}$
it follows that $A^i V = P_iA^i W$.
Furthermore, we find an $i$-independent constant $C>0$ such
that $\|(I-P_i)v\|=\|Q_i v\| \le 2^{-i} C \|P_i v\|$ for all $v\in
\R^3$. Thus, Lemma \ref{lem0}, (ii) applies for sufficiently large
$i\in\N$ and provides the estimate
$$
\ang(A^i V, A^iW) \le 2^{-i} \frac C{(1-2^{-i}C)^{\frac 12}}
$$
which completes the proof.

\begin{figure}[hbt]
    \begin{center}
      \includegraphics[width=0.7\textwidth]{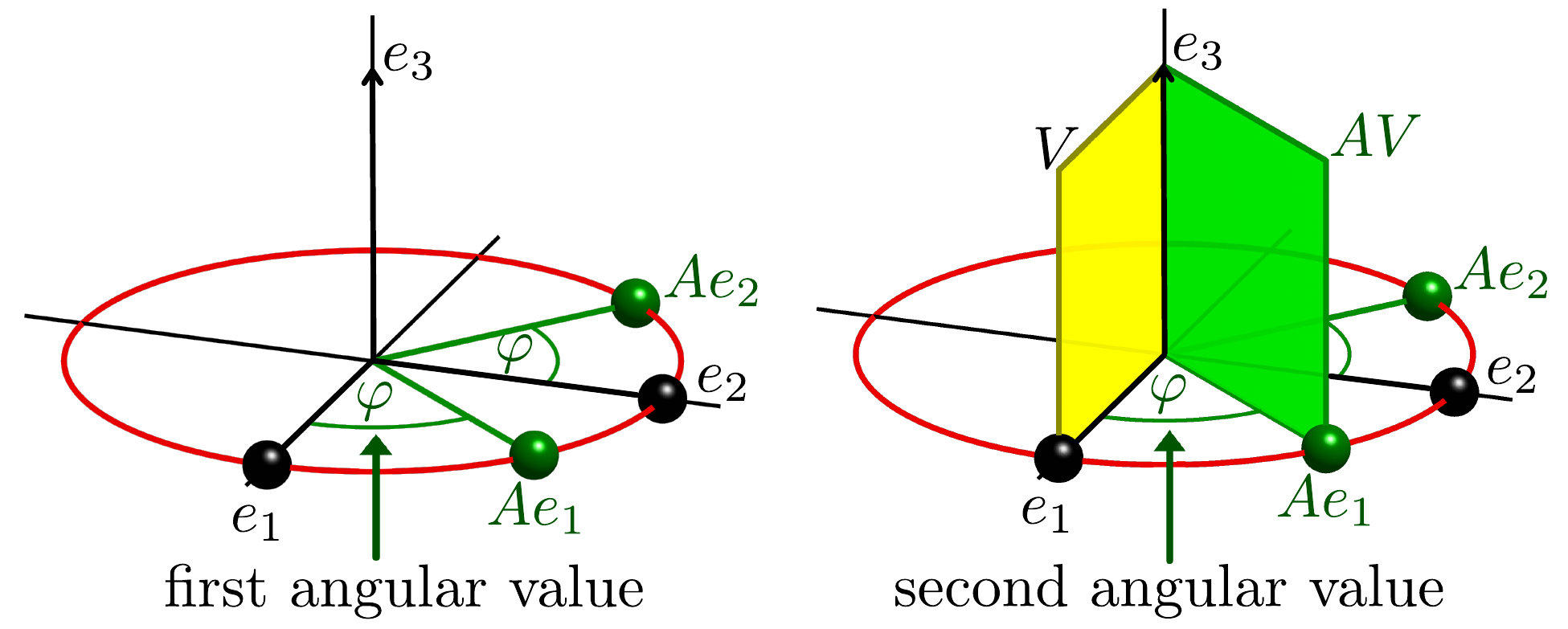}
    \end{center}
\caption{\label{ai} First and second angular values for the motivating
three-dimensional system \eqref{mot3d}.}
\end{figure}
In general, equality does not hold in Diagram \ref{eq2.4}. This phenomenon
is illustrated in Section \ref{sec3.2} by Examples \ref{example1}
and \ref{example2}.
However, angular values do agree  when the angles of iterates or their
averages occurring in
Definition \ref{defangularvalues} have some uniformity properties.
For this purpose let us introduce for $n \in \N$ the functions
\begin{equation} \label{Fbequi}
  \begin{aligned}
    b_n: \mathcal{G}(s,d)\to \R, & \quad b_n(V)  =
    \ang(\Phi(n-1,0)V,\Phi(n,0)V),
     \end{aligned}
\end{equation}
and recall $a_{1,n}(V)= \sum_{j=1}^n b_j(V)$ from Definition
\ref{defangularvalues}. Let us also recall the notion of uniform
almost periodicity
for a sequence of functions.
\begin{definition}\label{unifap}
Given a set $\cV$ and a Banach space $(\cW,\|\cdot\|)$. A sequence of mappings
$b_n: \cV\to \cW$, $n \in \N$ is called uniformly almost
periodic if
\begin{align*}
&\forall \eps >0\ \exists P\in\N:\forall V \in \cV\ \forall \ell\in\N\ \exists
p\in\{\ell,\dots,\ell+P\}: \\
&\forall n \in \N: \|b_n(V)-b_{n+p}(V)\|\le \eps.
\end{align*}
\end{definition}
\begin{remark} \label{remuniform}
  Our definition is slightly weaker than the standard notion
  (\cite[Ch.4.1]{P89}) which requires for each $\eps>0$ the
  existence of a relatively dense set $\cP\subset \N$ such that
\begin{equation*}
\forall n\in\N\ \forall p \in\cP \ \forall V \in \cV:
\|b_n(V)-b_{n+p}(V)\|\le \eps.
\end{equation*}
This is more restrictive,
since the choice of $p\in\{\ell,\dots,\ell+P\}\cap \cP$ is  uniform in $V$.
\end{remark}

The following Proposition \ref{propequi}
will be used repeatedly when determining angular values
for the two-dimensional case; see Proposition \ref{lem2.3}.
First, we state a crucial observation, which is proven in the
Supplementary materials \ref{app2}.
\begin{lemma}\label{estap}
Let $b_n: \cV\to \cW$, $n \in \N$ be a sequence of uniformly almost periodic
and uniformly bounded functions.
Then for all $\eps >0$ there exists $N\in\N$ such that for all $n\ge m\ge N$,
$ k\in\N$, $V\in\cV$
\begin{equation*}
 \Big\|\frac 1n \sum_{j=1}^n b_j(V) - \frac 1m \sum_{j=1}^m
b_{j+k}(V)\Big\| \le \eps.
\end{equation*}
\end{lemma}

\begin{proposition} \label{propequi} The following statements hold
  for all $s \in \{1,\ldots,d\}$.
  \begin{itemize}
  \item[(a)] If the functions
    $\frac{1}{n}a_{1,n}\colon\mathcal{G}(s,d)\to \R $
    converge uniformly to the constant function  $\varphi \in [0,\frac{\pi}{2}]$
    as $n  \to \infty$,
  then all nonuniform angular values coincide, i.e.\
  $\bar{\theta}_s=\hat{\theta}_s=
  \underaccent{\bar}\theta_s=\underaccent{\hat}\theta_s=\varphi$.
\item[(b)] If the functions $b_n$, $n \in \N$ from \eqref{Fbequi} are uniformly almost
  periodic, then all angular values coincide,
  $$
\underaccent{\hat}\theta_{[s]} = \underaccent{\hat}\theta_{s} = \hat
\theta_s = \hat\theta_{[s]} = \underaccent{\bar}\theta_{[s]} =
\underaccent{\bar}\theta_{s} = \bar \theta_s = \bar\theta_{[s]}.
$$
  \end{itemize}
\end{proposition}

\begin{proof}

  The claim in (a) is clear since $\limsup$ and $\liminf$
  in  \eqref{douter} are limits and the supremum is continuous w.r.t.\
  uniform convergence.

  Lemma \ref{lem2.2} shows that it suffices for (b) to prove
  $\bar \theta_{[s]}\le \underaccent{\hat}\theta_{[s]} $.
  By the definition \eqref{duniinner} we find for every $\eps>0$
a number $N_1\in\N$ and for all $n\ge N_1$
elements  $V_n\in\cG(s,d)$, $k_n \in\N$ such that
$$
\Big| \bar \theta_{[s]} - \frac 1n \sum_{j=1}^n b_{j+k_n}(V_n)\Big|\le \frac \eps 2.
$$
From  Lemma \ref{estap} we obtain $N=N(\varepsilon)\in\N$, $N\ge N_1$ such that for
all $n\ge N$, $h\in \N_0$
\begin{align*}
\Big|\frac 1n \sum_{j=1}^n b_j(V_N) - \frac 1N \sum_{j=1}^N
  b_{j+k_N}(V_N)\Big| +
\Big|\frac 1n \sum_{j=1}^n b_j(V_N) - \frac 1n \sum_{j=1}^n
  b_{j+h}(V_N)\Big| \le \frac \eps 2.
\end{align*}
Combining these results yields for all $h\in \N_0$, $n\ge N$:
\begin{align*}
  \bar\theta_{[s]}& \le \frac 1n \sum_{j=1}^n b_{j+h}(V_N)
  + \Big|\frac 1n \sum_{j=1}^n b_{j}(V_N) - \frac 1n \sum_{j=1}^n b_{j+h}(V_N)
\Big|\\
&+ \Big|\frac 1N \sum_{j=1}^N b_{j+k_N}(V_N) - \frac 1n \sum_{j=1}^n
b_{j}(V_N)\Big|+ \Big|\bar\theta_{[s]}-\frac 1N \sum_{j=1}^N b_{j+k_N}(V_N)\Big|
  \\
&\le \frac \eps 2 + \frac \eps 2  + \frac 1n \sum_{j=1}^n b_{j+h}(V_N).
\end{align*}

Taking the infimum over $h$ and the limit $n \to \infty$ (see Lemma
\ref{lem2.2} for its existence) shows that there exists
$V(\eps):=V_{N(\eps)}\in \cG(s,d)$ satisfying
\begin{equation*}
\bar \theta_{[s]} -  \eps \le  \lim_{n\to \infty} \inf_{h\in\N_0} \frac 1n
\sum_{j=1}^n b_{j+h}(V(\eps)).
\end{equation*}
Thus we deduce
\begin{equation*}
\bar \theta_{[s]} - \eps \le  \sup_{V\in\cG(s,d)}  \lim_{n\to \infty}\inf_{h\in\N_0}
\frac 1n  \sum_{j=1}^n b_{j+h}(V) = \underaccent{\hat}\theta_{[s]} \le
\bar \theta_{[s]}.
\end{equation*}
\
\end{proof}

Next, we apply a kinematic similarity, induced by a
transformation $\tilde{u}_n=Q_n u_n$ with $Q_n \in \mathrm{GL}(\R^d)$
 to \eqref{diffeq}, i.e.\ we consider
\begin{equation} \label{difftransform}
  \tilde{u}_{n+1} = \tilde{A}_n \tilde{u}_n, \quad \tilde{A}_n =
  Q_{n+1}A_n Q_n^{-1},
\end{equation}
and ask when angular values remain unchanged.

\begin{proposition}(Invariance of angular values) \label{prop2}
  \begin{enumerate}[(i)]
  \item
   Assume $Q_n=r_nQ$, $n \in \N$ with $r_n\in \R$, $r \neq 0$ and $Q \in \R^{d,d}$
   orthogonal.
  Then the angular values of \eqref{diffeq} and \eqref{difftransform}
  agree.
\item Assume constant transformation matrices $Q_n=Q, n \in \N_0$
  with $Q$ invertible. If any of the values $\theta_s \in \{
  \underaccent{\hat}\theta_{[s]}, \underaccent{\hat}\theta_{s}, \hat
\theta_s, \hat\theta_{[s]}, \underaccent{\bar}\theta_{[s]},
\underaccent{\bar}\theta_{s}, \bar \theta_s, \bar\theta_{[s]}\}$
  in Definition
  \ref{defangularvalues} vanishes for the system \eqref{diffeq}
  then the same angular value vanishes for the transformed system \eqref{difftransform}.
  \end{enumerate}
\end{proposition}
\begin{proof}
First note that the solution operators $\Phi_A$, $\Phi_{\tilde{A}}$  of
\eqref{diffeq}, \eqref{difftransform}
  are related by
  \begin{equation} \label{relatePhi}
   \Phi_{\tilde{A}}(n,m)Q_m = Q_n \Phi_A(n,m), \quad n,m \in\N_0.
  \end{equation}
  The result of (i) follows from \eqref{relatePhi} and the invariance of angles
  under scalings and orthogonal transformations (cf.\ Proposition \ref{prop1})
  \begin{align*}
    \ang(\Phi_{\tilde{A}}(j-1,0)QV,\Phi_{\tilde{A}}(j,0)QV)& =
    \ang(\frac{r_{j-1}}{r_0}Q\Phi_A(j-1,0)V,\frac{r_{j}}{r_0} Q\Phi_A(j,0)V) \\
    &=\ang(\Phi_A(j-1,0)V, \Phi_A(j,0)V).
  \end{align*}
   For case (ii) the relation \eqref{relatePhi} reads  $\Phi_{\tilde{A}}(j,0)Q =  Q \Phi_A(j,0)$
  and the assertion follows from the angle-boundedness \eqref{anglebound}
  of the matrix $S=Q$.
\end{proof}
 One can strengthen Proposition \ref{prop2} as follows.
For assertion (i) it is sufficient if
$Q_n=r_nP_n$ where $P_n$ converges to some orthogonal matrix, and for
assertion (ii) it is sufficient if $Q_n$ converges to some
invertible matrix.
However, we do not expect substantially more general
transformations to leave all angular values invariant. For example, if a
  kinematic similarity preserves the single terms
  $\ang(u_n,A_n u_n) =\ang(Q_n u_n,Q_{n+1} A_n u_n)$, for all $n \in
  \N_0$ and if a condition is desired which does not  depend on the
  particular choice of $A_n$,
one is led to the
property $\ang(Q_nu,Q_{n+1}v)=\ang(u,v)$ for all $u,v \in \R^d$,
$n\in\N_0$.
The latter condition implies that all matrices  $Q_n$, $n\in\N_0$ are
 multiples of a common orthogonal matrix.

Finally, we discuss an invariance property of maximizers which occur
with the outer angular values.
Starting with the difference equation \eqref{diffeq}, we define
for $\eta,n \in \N_0$ the matrices $A_n(\eta) := A_{n+\eta}$. Denote by
$\Phi_{\eta}^{+}$ the solution operator of the shifted difference equation
\begin{equation}\label{shifted}
u_{n+1} = A_{n+\eta} u_n,\quad n\in\N_0
\end{equation}
and observe that for all $n,m,\eta \in \N_0$
$$
\Phi_{\eta}^+(n,m) = \Phi(n+\eta,m+\eta).
$$
Let $\hat \theta_s(\eta)$ be the $s$-th upper outer
angular value for
\eqref{shifted}. The corresponding maximizers that occur
with the outer values are given by
$$
\hat \cV_s(\eta) = \Big\{V\in\cG(s,d): \hat\theta_s(\eta) = \limsup_{n\to
  \infty} \frac 1n \sum_{j=1}^n
\ang(\Phi_{\eta}^+(j-1,0)V,\Phi_{\eta}^+(j,0)V)\Big\}.
$$
Note that this set may be empty.
We obtain the following invariance.

\begin{proposition}\label{invar1}
Let $A_n \in \R^{d,d}$, $n\in\N_0$ be invertible matrices.
Then the following relation holds for all $\eta \in \N$,
  \begin{equation}\label{invari}
  A_\eta \hat \cV_s(\eta) = \hat \cV_s(\eta+1).
  \end{equation}
\end{proposition}

\begin{proof}
Fix $\eta \in \N$ and let $V\in \cG(s,d)$.
Then we get
\begin{align*}
 & \frac 1n \sum_{j=1}^n \ang(\Phi_{\eta+1}^+(j-1,0)A_\eta
   V,\Phi_{\eta+1}^+(j,0)A_\eta V)\\
 & = \frac 1n  \sum_{j=1}^n \ang(\Phi(j+\eta,\eta+1)A_\eta
   V,\Phi(j+\eta+1,\eta+1)A_\eta V)\\
 & = \frac 1n  \sum_{j=1}^n
   \ang(\Phi(j+\eta,\eta)V,\Phi(j+\eta+1,\eta)V)\\
 & = \frac {n+1}n \frac 1{n+1} \left( \sum_{j=1}^{n+1}
   \ang(\Phi_{\eta}^+(j-1,0)V,\Phi_{\eta}^+(j,0)V) - \ang(V,\Phi_{\eta}^+(1,0)V)\right).
\end{align*}
Taking $\limsup$ as $n\to \infty$ we have
\begin{align*}
&\limsup_{n\to\infty} \frac 1n \sum_{j=1}^n \ang(\Phi_{\eta+1}^+(j-1,0)A_\eta
   V,\Phi_{\eta+1}^+(j,0)A_\eta V)\\
&= \limsup_{n\to\infty} \frac 1n \sum_{j=1}^n \ang(\Phi_{\eta}^+(j-1,0)
   V,\Phi_{\eta}^+(j,0) V).
\end{align*}
In the case $\hat \cV_s(\eta) = \emptyset$ then $\hat \cV_s(\eta+1) =
\emptyset$ and \eqref{invari} is trivial. Otherwise, the invertibility
of $A_\eta$ yields
$$
V \in \hat \cV_s(\eta) \Leftrightarrow A_\eta V \in \hat \cV_s(\eta+1)
$$
which proves \eqref{invari}.
\end{proof}

A corresponding result also holds for lower outer angular values as
well as for uniform outer angular values.

\subsection{Some nonautonomous key examples}\label{sec3.2}
Upper, lower, uniform respectively non-uniform outer and inner angular values
do not coincide in general. The following
examples illustrate this fact.

First, we construct an example which possesses different upper, lower and
uniform angular values.
A related example in continuous time can be found in \cite[Example
2.2]{dv07}. There, the authors illustrate that the Lyapunov spectrum
may be a proper subset of the Sacker-Sell spectrum and that generally
both spectra do not consist of isolated points only.

\begin{example}\label{example1}
Fix $0 \le \varphi_0 < \varphi_1 \le \frac \pi 2$ and let $T_\varphi :=
\begin{pmatrix}
\cos\varphi & -\sin\varphi\\ \sin\varphi& \cos\varphi
\end{pmatrix}$.
For $n\in\N_0$, we define
$$
A_n =
\begin{cases}
T_{\varphi_0},& \text{ for } n = 0 \lor n \in \bigcup_{\ell=1}^\infty
[2^{2\ell-1},2^{2\ell} -1]\cap \N,\\
T_{\varphi_1},&\text{ otherwise.}
\end{cases}
$$
Table \ref{T1} illustrates this construction.

\begin{table}[hbt]
  \begin{center}
  \begin{tabular}{c|cccccccccccccccccccccc}
    $n$ & 0 & 1 & 2 & 3 & 4 & 5 & 6 & 7 & 8 & \dots & 15 & 16\\\hline
   $A_n$ & $T_{\varphi_0}$&$T_{\varphi_1}$& $T_{\varphi_0}$&
   $T_{\varphi_0}$ & $T_{\varphi_1}$ & $T_{\varphi_1}$ & $T_{\varphi_1}$ &
   $T_{\varphi_1}$ & $T_{\varphi_0}$ & $\dots$ & $T_{\varphi_0}$ &
    $T_{\varphi_1}$
  \end{tabular}
  \normalsize
\caption{Construction of $(A_n)_{n\in\N_0}$.\label{T1}}
\end{center}
\end{table}

Inner and outer angular values coincide for the nonautonomous
difference equation
$$
u_{n+1} = A_n u_n,\quad n \in \N_0,
$$
since all one-dimensional
subspaces rotate through the same angle.

Denote by $p_\ell$
the number of occurrences of $T_{\varphi_1}$ in $(A_n)_{0\le n\le \ell}$.
One observes for $n \in
\N$ that
$$
p_{2^{2n-1}-1} = \frac 13 (4^n-1) = p_{2^{2n}-1}
$$
and
$$
\lim_{n\to \infty} \frac 1{2^{2n-1}-1} p_{2^{2n-1}-1} = \frac 23,\quad
\frac 1{2^{2n}-1} p_{2^{2n}-1} = \frac 13.
$$
Thus, we obtain
$$
\underaccent{\bar}\theta_1 = \underaccent{\hat}\theta_1 = \frac 23
\varphi_0 + \frac 13 \varphi_1,\quad
\bar \theta_1 = \hat \theta_1 = \frac 13 \varphi_0 + \frac 23
\varphi_1.
$$
For each $n\in\N$, we find infinitely many indices $\nu \in \N$ such
that $A_{\nu+\ell} = T_{\varphi_0}$ (resp.\ $A_{\nu+\ell} = T_{\varphi_1}$) for all $\ell =
0,\dots,n-1$. As a consequence, the Diagram \ref{eq2.4} has the
explicit form in Diagram \ref{D1}.
\begin{diagram}
\begin{equation*}\label{diag1}
    \begin{matrix}
    \varphi_0 = \underaccent{\hat} \theta_{[1]} & < & \frac 23
    \varphi_0 + \frac 13 \varphi_1 = \underaccent{\hat}
    \theta_1 & < & \frac 13 \varphi_0 + \frac 23 \varphi_1 = \hat
    \theta_1 & < & \hat \theta_{[1]}= \varphi_1\\
    \rg& & \rg && \rg && \rg\\
    \varphi_0 = \underaccent{\bar} \theta_{[1]} & < & \frac 23
    \varphi_0 + \frac 13 \varphi_1 = \underaccent{\bar}
    \theta_1 & < & \frac 13 \varphi_0 + \frac 23 \varphi_1 = \bar
    \theta_1 & < & \bar \theta_{[1]} = \varphi_1
    \end{matrix}
\end{equation*}
\caption{Angular values of Example \ref{example1}. \label{D1}}
\end{diagram}
\end{example}

Although inner and outer angular values coincide for Example \ref{example1}, this
coincidence is in general not true.  We discuss the following
example.
\begin{example}\label{example2}
  Let
  $$
  C :=
  \begin{pmatrix}
  1 & 0\\ 0 & \frac 12
\end{pmatrix},\quad
R :=
\begin{pmatrix}
-1 & 0\\ 0 & 1
\end{pmatrix}.
$$
In the case of the reflection $R$, we observe for $v =
\begin{pmatrix}
\cos \phi\\\sin \phi
\end{pmatrix}$, $\phi \in [0,\frac \pi2]$ that
$$
\ang(v,Rv) =
\begin{cases}
  2\phi,& \text{ for } 0\le \phi \le \frac \pi 4,\\
  \pi - 2\phi,& \text{ for } \frac \pi 4< \phi \le \frac \pi 2
  \end{cases}
$$
and the maximal angle is achieved at $v\in
\mathrm{span}\left\{\left(\begin{smallmatrix}1\\1\end{smallmatrix}\right)\right\}$.

For $n\in\N_0$, we define
$$
A_n :=
\begin{cases}
R, & \text{ for } n \in \bigcup_{\ell = 1}^\infty [2\cdot 2^\ell -
4,3\cdot 2^\ell -5],\\
C,&\text{ otherwise.}
\end{cases}
$$
Table \ref{T2} illustrates this construction.

\begin{table}[hbt]
  \begin{center}
  \begin{tabular}{c|cccccccccccccccccccccc}
    $n$ & 0 & 1 & 2 & 3 & 4 & 5 & 6 & 7 & 8 & 9 & 10 & 11 & 12 & \dots
    & 19 & 20 \\\hline

   $A_n$ & $R$ & $R$ & $C$ & $C$ & $R$ & $R$ & $R$ & $R$ & $C$ & $C$ &
   $C$ & $C$ & $R$ & \dots & $R$ & $C$
  \end{tabular}
  \normalsize
\caption{Construction of $(A_n)_{n\in\N_0}$.\label{T2}}
\end{center}
\end{table}

We prove that inner and outer angular values of the nonautonomous
difference equation
$$
u_{n+1} = A_n u_n,\quad n\in\N_0\
$$
do not coincide. First we show that $\hat \theta_1 = 0$.
Let
$$
V_\phi := \mathrm{span}
\begin{pmatrix}
\cos \phi\\\sin \phi
\end{pmatrix},\quad
b_j(\phi) = \ang(\Phi(j-1,0) V_\phi,\Phi(j,0)V_\phi).
$$
For $\phi \in \{0,\frac \pi 2\}$ we get $b_j(\phi) = 0$ for all $j \in
\N$. In the case $\phi \in (0,\frac \pi 2)$ we observe that
$\Phi(j,0)V_\phi \to V_0$ as $j\to \infty$. Thus for each $\eps >0$
there exists an $N\in\N$ such that  $b_j(\phi) \le \eps$ for all $j\ge N$.
As a consequence we get for $n$ sufficiently large
\begin{align*}
  \frac 1n \sum_{j=1}^n b_j(\phi)
  & = \frac 1n \left( \sum_{j=1}^{N-1} b_j(\phi) + \sum_{j=N}^n
    b_j(\phi) \right)
   \le \frac 1n \left((N-1) \frac \pi 2 + (n+1 - N)\eps\right)
    \le 2\eps
\end{align*}
and this shows
$$
\hat \theta_1 = \sup_{V\in\cG(1,2)} \limsup_{n\to \infty} \frac 1n
\sum_{j=1}^n \ang(\Phi(j-1,0)V,\Phi(j,0)V) = 0.
$$
Similarly, all outer angular values are zero.

Next, we determine an estimate for the upper inner angular value. We
claim that $\bar \theta_1 \ge \frac \pi 6$.

For $\ell \in\N$ let $p_\ell := 3\cdot 2^\ell -5$. Note that the
matrix $C$ appears  $2^\ell-2$
times in $(A_n)_{0\le n\le p_\ell}$
and $R$ appears $2^\ell$ times in  $(A_n)_{p_{\ell-1}< n\le p_\ell}$.

Let $V(\ell) := \mathrm{span}\{ C^{-2^\ell
    +2} \left(\begin{smallmatrix}1\\1\end{smallmatrix}\right)\}$.
We obtain
\begin{align*}
  \bar \theta_1
  & = \limsup_{n\to \infty} \sup_{v\in \cG(1,2)} \frac 1n
    \sum_{j=1}^n \ang(\Phi(j-1,0)V,\Phi(j,0)V)\\
  & \ge \limsup_{\ell \to \infty} \frac 1{p_\ell+1}
    \sum_{j=1}^{p_\ell+1} \ang(\Phi(j-1,0)V(\ell),\Phi(j,0)V(\ell))\\
  & \ge \limsup_{\ell \to \infty} \frac 1{p_\ell+1} 2^\ell \frac \pi 2
    = \lim_{\ell \to \infty} \frac {2^\ell}{3\cdot 2^\ell -4}\cdot \frac \pi 2
    = \frac \pi 6.
\end{align*}

We obtain estimates for $\protect\underaccent{\bar}\theta_1$ by
  analyzing the subsequence $n=p_\ell-2^\ell+1$ for $\ell \in
  \N$. These indices detect the end of each block of $C$s. In
  particular, we observe that $\frac \pi{12} \le
  \protect\underaccent{\bar}\theta_1< \frac \pi 6$
and present results for all angular values in Diagram \ref{D2}.
\begin{diagram}
\begin{equation*}\label{diag2}
    \begin{matrix}
    0 = \underaccent{\hat} \theta_{[1]} & = & 0=\underaccent{\hat}
    \theta_1 & = & 0=\hat \theta_1 & = & \hat \theta_{[1]}=0\\
    \rg & & \rl && \rl && \rl\\
    0=\underaccent{\bar} \theta_{[1]} & < & \frac \pi{12} \le \underaccent{\bar}
    \theta_1 & < & \frac \pi 6 \le \bar \theta_1 & < & \bar
    \theta_{[1]} = \frac \pi 2
    \end{matrix}
\end{equation*}
\caption{Angular values of Example \ref{example2}. \label{D2}}
\end{diagram}
\end{example}

\section{Angular values of random linear cocycles}
\label{sec3}
Following \cite[Ch.3.3.1]{A1998} we consider a probability space
$(\Omega,\cF,\PP)$ and let $T:\Omega\circlearrowleft$ be a measurable,
$\PP$-preserving, ergodic transformation.
Let $A:\Omega\to \mathrm{GL}(\R^d)$ and set
$A^{(n)}_\omega=A(T^{n-1}\omega)\cdots A(T\omega)A(\omega)$.
Note that $A^{(n)}_\omega$ corresponds to a random solution
  operator $\Phi(n,0,\omega)$ in the setting of Section \ref{sec2}; cf.\
 \cite[(3.3.2)]{A1998}.
 In analogy to the right-hand side of \eqref{average}, for $n\ge 1$, define
 for $s\in\{1,\ldots,d\}$ and $V \in \mathcal{G}(s,d)$
\begin{equation*}
\label{randoma}
a_n(\omega,V)=\sum_{j=0}^{n-1}\ang(A_\omega^{(j)}V,A_\omega^{(j+1)}V).
\end{equation*}
Define a skew product $\tau:\Omega\times
\mathcal{G}(s,d)\circlearrowleft$ by
$\tau(\omega,V)=(T\omega,A(\omega)V)$ and
$f:\Omega\times\mathcal{G}(s,d)\to\R$ by
$f(\omega,V)=\ang(V,A(\omega)V)$.
One has the Birkhoff sum representation:
\begin{equation*}
\label{randomaskew}
a_n(\omega,V)=\sum_{j=0}^{n-1}f(\tau^j(\omega,V)).
\end{equation*}
The following result provides general conditions for the angular value
limits to be independent of the initial condition $\omega$ and
reference subspace $V$. {We use the notation
  $\vartheta_s$ for angular
values in this setting, to distinguish them from the angular values
$\theta_s$ in Section \ref{sec2}.
\begin{theorem} \label{thmbirkhoff}
Suppose $\tau$ preserves an ergodic probability measure $\mu$ on
$\Omega\times\cG(s,d)$, where $\mu$ has marginal $\PP$ on $\Omega$;
that is, $\mu(\cdot,\cG(s,d))=\PP$.\\
Then there is a $\vartheta_s\in [0,\frac \pi 2]$ satisfying
\begin{equation*}
\label{birkhoffav}
\vartheta_s=\lim_{n\to \infty}
\frac{1}{n}a_n(\omega,V)=\int_{\Omega\times\mathcal{G}(s,d)}
\ang(V,A(\omega)V)\ d\mu(\omega,V),
\end{equation*}
for $\mu$-almost every $(\omega,V)\in \Omega\times \mathcal{G}(s,d)$.
\end{theorem}
\begin{proof}
This follows immediately from Birkhoff's ergodic theorem and ergodicity of $\tau$.
\end{proof}

The existence of an ergodic invariant measure $\mu$ for $\tau$ is
connected with a certain irreducibility condition on the action on
subspaces, leading to an independence of the angular values with
respect to $V$.
We would like to treat general linear cocycles and so
in analogy to Section \ref{sec2.1},
we consider angular values of a random linear cocycle w.r.t.\
$\omega \in \Omega$ and ask for extreme values w.r.t.\
$V \in \mathcal{G}(s,d)$.
To keep the analogy with Definition \ref{defangularvalues} we use
   an up/down bar for  upper/lower inner angular values
   and an up/down hat for upper/lower outer angular values, while their uniform
   equivalents are
  denoted by the bracketed index $[s]$.

\begin{theorem} \label{thmRDS}
 Let  $T:\Omega\circlearrowleft$ be a measurable,
$\PP$-preserving, ergodic transformation and let $A: \Omega \to \mathrm{GL}(\R^d)$.
Then the following assertions hold.

\begin{enumerate}
\item There is a  number
  $\bar{\vartheta}_s$ such that for $\PP$-a.e.\ $\omega$,
  \begin{equation}
    \label{part1eq}
    \bar{\vartheta}_s=\lim_{n \to \infty}
  \max_{V\in\mathcal{G}(s,d)} \frac{a_n(\omega,V)}{n} = \inf_{n\in \N}\frac{1}{n}\int_\Omega
  \max_{V\in\mathcal{G}(s,d)} a_n(\omega,V)\
  d\PP(\omega).
  \end{equation}
  In particular, one has
  \begin{equation*}
   {\bar \vartheta}_s\le \int_\Omega
    \max_{V\in\mathcal{G}(s,d)} \ang(V,A(\omega)V)\ d\PP(\omega).
    \end{equation*}
\item  There is a number
  $\hat{\vartheta}_s$ such
  that for $\PP$-a.e.\ $\omega$,
  $$\hat{\vartheta}_s=\sup_{V\in\mathcal{G}(s,d)}
  \limsup_{n \to \infty} \frac{a_n(\omega,V)}{n}.$$
  Furthermore, if for $\PP$-a.e.\ $\omega$ the supremum over $V$ is
  achieved by at most $K<\infty$ subspaces $V_1(\omega),\ldots,
  V_K(\omega)$, then one may create $K$
  equivariant collections $\{V_k(\omega)\}_{1\le k\le K,\omega\in\Omega}$,
  satisfying $V_k(T\omega)=A(\omega)V_k(\omega)$, $k=1,\ldots,K$.
\item There is a number
  $\underaccent{\hat}{\vartheta}_s$ such
  that  for $\PP$-a.e.\ $\omega$,
  $$\underaccent{\hat}{\vartheta}_s=\sup_{V\in\mathcal{G}(s,d)}
  \liminf_{n \to \infty} \frac{a_n(\omega,V)}{n}.$$
   Furthermore, if for $\PP$-a.e.\ $\omega$ the supremum over $V$ is
   achieved by at most $K<\infty$ subspaces $V_1(\omega),\ldots,
   V_K(\omega)$, then one may create $K$
   equivariant collections $\{V_k(\omega)\}_{1\le k\le K,\omega\in\Omega}$,
   satisfying $V_k(T\omega)=A(\omega)V_k(\omega)$, $k=1,\ldots,K$.
\item There is a number $\bar{\vartheta}_{[s]}$ such that
$$\bar{\vartheta}_{[s]}=\lim_{n \to \infty}
\sup_{V\in\mathcal{G}(s,d)} \esssup_{\omega \in \Omega}
\frac{a_n(\omega,V)}{n}=\inf_{n \to \infty} \esssup_{\omega \in
  \Omega}\max_{V\in\mathcal{G}(s,d)} \frac{a_n(\omega,V)}{n}.$$
In particular, one has
\begin{equation*}
  \bar{\vartheta}_{[s]}
  \le \esssup_{\omega \in \Omega} \max_{V\in\mathcal{G}(s,d)}\ang(V,A(\omega)V).
  \end{equation*}
\item There is a  number
  $\underaccent{\hat}{\vartheta}_{[s]}$
  such that
    $$\underaccent{\hat}{\vartheta}_{[s]}=\sup_{V\in\mathcal{G}(s,d)}\lim_{n\to\infty}\essinf_{\omega
    \in \Omega} \frac{a_n(\omega,V)}{n}=
  \sup_{V\in\mathcal{G}(s,d)}\sup_{n\ge
    1}\; \essinf_{\omega \in \Omega} \frac{a_n(\omega,V)}{n}.$$
  In particular, one has
  $$
  \sup_{V\in\mathcal{G}(s,d)} \essinf_{\omega \in
    \Omega}\ang(V,A(\omega)V) \le \underaccent{\hat}{\vartheta}_{[s]}.$$
  \item There is a  number
  $\hat{\vartheta}_{[s]}$
  such that
    $$
    \hat{\vartheta}_{[s]}= \sup_{V\in\mathcal{G}(s,d)}\lim_{n \to \infty}
    \esssup_{\omega \in \Omega} \frac{a_n(\omega,V)}{n} =
    \sup_{V\in\mathcal{G}(s,d)} \inf_{n\ge 1}
  \esssup_{\omega \in \Omega} \frac{a_n(\omega,V)}{n}.$$
  In particular, one has
  $$  \hat{\vartheta}_{[s]}\le \sup_{V\in\mathcal{G}(s,d)}
  \esssup_{\omega \in \Omega} \ang(V,A(\omega)V).$$
\end{enumerate}
\end{theorem}
\begin{remark}
In Part 1 of Theorem \ref{thmRDS}, if $\Omega$ is a metric space,
$\omega\mapsto A(\omega)$ is continuous, and
$T:\Omega\circlearrowleft$ is uniquely ergodic, then using the fact
that $\omega\mapsto\max_V \ang(V,A(\omega)V)$ is continuous, we obtain
the following stronger conclusion.
By Theorem 1.5 \cite{stst00}, the limit in  \eqref{part1eq}
converges in a semi-uniform way (the limit exists for
\emph{all} $\omega\in\Omega$ and converges uniformly in $\omega$):
given $\epsilon>0$, there exists $n_0$ such that for all $n\ge n_0$,
$$
{\bar \vartheta}_s\le \max_{V\in\mathcal{G}(s,d)}
\frac{a_n(\omega,V)}{n}\le
{\bar \vartheta}_s+\epsilon\quad\mbox{ for all
  $\omega\in\Omega$.}
$$
A simple example of such a system is an irrational rotation on the
unit circle $\Omega=S^1$ and
$T\omega=\omega+\varphi$, where $\varphi\notin \Q$. The Lebesgue measure
on $S^1$ is the unique invariant probability measure. One may choose
\emph{any} continuous matrix-valued function $A$.
More generally, one may consider rationally independent
  translations on higher-dimensional tori.
\end{remark}

\begin{proof}[Proof of Theorem \ref{thmRDS}]
\
\begin{enumerate}
\item We note that the invertibility of the matrices $A(\omega)$
  implies $A(\omega)\mathcal{G}(s,d)=\mathcal{G}(s,d)$ for $\PP$ a.e.\
  $\omega$.
  As in the proof of Lemma \ref{lem2.2} one can easily show that
  \begin{align*}\max_{V\in\cG(s,d)} a_{n+m}(\omega,V)\le
    \max_{V\in\cG(s,d)} a_{n}(\omega,V)+\max_{V\in\cG(s,d)}
    a_{m}(T^n\omega,V)
  \end{align*}
  for every $n,m\ge 0$, and therefore
  $g_n(\omega):=\max_{V\in\cG(s,d)}  a_{n}(\omega,V)$ is a subadditive
  sequence of functions. Recall that
  $0\le a_n(\omega,V)\le \frac\pi 2$ for all $\omega, V$, implying $0 \le
  g_n\le \frac \pi 2$.
  The results are now immediate by the subadditive ergodic theorem
  applied to $g_n$, using ergodicity of $\PP$.
\item By direct computation, one verifies that
  \begin{align*}
    a_n(T\omega,A(\omega)V)=a_n(\omega,V)-\ang(V,A(\omega)V)+\ang(A^{(n)}_\omega
    V,A^{(n+1)}_\omega V).
    \end{align*}
  Thus, because values of angles are bounded, one has
  \begin{equation}
  \label{equi}
    \limsup_{n\to \infty}\frac 1n
    a_n(T\omega,A(\omega)V)=\limsup_{n\to\infty} \frac 1n a_n(\omega,V)=:g(\omega,V),
  \end{equation}
  where $g$ is measurable in $\omega$ and $V$.
      By the invertibility of $A(\omega)$ for a.e.\ $\omega$, we see that
      \begin{align*}
        h(\omega)&:=\sup_{V\in\cG(s,d)} g(\omega,V)
        =\sup_{V\in\cG(s,d)} \limsup_{n\to\infty}\frac 1n
        a_n(T\omega,A(\omega)V)\\
        &=\sup_{V\in\cG(s,d)}
        \limsup_{n\to\infty} \frac 1n a_n(T\omega,V)=h(T\omega).
      \end{align*}
          By ergodicity, the $T$-invariant function $h$ is constant a.e.\
  The expression \eqref{equi} demonstrates equivariance of a
  particular maximizing $V$ (if it exists).
  By hypothesis for $\PP$-a.e.\ $\omega$ there are at most $K$
  distinct subspaces $\tilde{V}$ satisfying
  $\sup_{V\in\mathcal{G}(s,d)}g(\omega,V)= g(\omega,\tilde{V})$.
  By ergodicity, the invertibility of the $A(\omega)$, and the
  equivariance property, the number of solutions must be independent
  of $\omega$ on a full $\mathbb{P}$-measure set; let us call this
  number $K$.
  The equivariance property allows us to ``match'' the $K$ pointwise
  solutions to create $K$ families of maximizing subspaces
  $V_1(\omega),\ldots, V_K(\omega)$ obeying equivariance.
\item This proof is analogous to Part 2.
\item Since $V\mapsto \frac 1n a_n(\omega,V)$ is continuous for each
 $n\in \N$ and $\PP$-a.e.\ $\omega$, and $\mathcal{G}(s,d)$ is compact,
  we may replace the $\max_{V\in\cG(s,d)}$ with $\sup_{V\in\cG(s,d)}$ in all statements
  of Part 4.
   One may now interchange the operations
  $\esssup_\omega$ and $\sup_{V\in\cG(s,d)}$.
 Similarly to the proof of Part 1, one shows that $\sup_{V\in\cG(s,d)}
 \esssup_\omega a_{n}(\omega,V)$ is a subadditive sequence. Then the
 results  follow immediately from
 Fekete's subadditive lemma.
  \item We note that for fixed $V$ we obtain a superadditive sequence of
  numbers $g_n(V):=\essinf_\omega a_n(\omega,V)$.
     By Fekete's superadditive lemma one has $\lim_{n\to\infty} g_n(V)$
     exists and equals $\sup_{n\in\N} g_n(V)$.  This proves all
     statements concerning
     $\underaccent{\hat}\vartheta_{[s]}$.
\item
     The results for $\hat{\vartheta}_{[s]}$ follow
     similarly, replacing superadditivity with subadditivity.
\end{enumerate}
\end{proof}
In the following we compare the various angular values
as in Diagram \ref{eq2.4}. First note that we have a limit in equation
\eqref{part1eq}. Therefore, it is unnecessary to distinguish between
upper and lower angular values $\bar{\vartheta}_s$,
$\underaccent{\bar}{\vartheta}_s$ as in Diagram \ref{eq2.4} for the
nonautonomous case.
To complete the following diagram,
we introduce the lower uniform
inner angular value
 \begin{equation*}
   \underaccent{\bar}{\vartheta}_{[s]}= \liminf_{n \to \infty}
   \sup_{V\in\cG(s,d)} \essinf_{\omega \in \Omega} \frac{a_n(\omega,V)}{n},
 \end{equation*}
 which does not appear in Theorem \ref{thmRDS}.

\begin{lemma} \label{diagramrandom}
  Let the assumptions of Theorem \ref{thmRDS} hold. Then the angular values
  defined above are related by Diagram \ref{D4}.
  \begin{diagram}
 \begin{equation*}\label{ineqRDS}
    \begin{matrix}
    \underaccent{\hat} \vartheta_{[s]} & \le & \underaccent{\hat}
    \vartheta_s & \le & \hat \vartheta_s & \le & \hat \vartheta_{[s]}\\
    \rle& & \rle && \rle && \rle\\
    \underaccent{\bar} \vartheta_{[s]} & \le & \bar\vartheta_s
    & = & \bar{\vartheta}_s & \le & \bar \vartheta_{[s]}\\
    \end{matrix}
     \end{equation*}
 \caption{Comparison of angular values for random dynamical systems.
 \label{D4}}
  \end{diagram}
\end{lemma}

\begin{proof}
 The proof of the inequalities in Diagram \ref{D4} is similar to the
 proof of Lemma \ref{lem2.2}.
\end{proof}
In the special case where $\Omega$ consists of a single point, we are
in the autonomous setting with a single matrix $A$.
We may apply the results of Theorem \ref{thmRDS} Parts 1--3 to obtain
the following corollary.
\begin{corollary} \label{corauto}
Let $a_n(V)=\sum_{j=0}^{n-1}\ang(A^jV,A^{j+1}V)$ with $A\in
\mathrm{GL}(\R^d)$ and $V \in \mathcal{G}(s,d)$. Then the following
holds:
\begin{enumerate}
\item The limit
$$
{\bar \vartheta}_s:=\lim_{n \to \infty} \max_{V\in\cG(s,d)}
  \frac{a_n(V)}{n}\mbox{ exists and equals }\inf_{n\in
    \N}\max_{V\in\cG(s,d)} \frac{a_n(V)}{n}.
$$
    In particular, one has
$$
\bar \vartheta_s\le  \max_{V\in\cG(s,d)} \ang(V,AV).
$$
\item There is a number $\hat{\vartheta}_s$ such
  that  $$\hat{\vartheta}_s=\sup_{V\in\mathcal{G}(s,d)}
  \limsup_{n \to \infty} \frac{a_n(V)}{n}.$$
  Furthermore, if the supremum over $V$ is achieved by a subspace $V$
  then the supremum is also achieved by $A^jV$ for all $j\in
  \Z$.
\item There is a number $\underaccent{\hat}{\vartheta}_s$ such
  that  $$\underaccent{\hat}{\vartheta}_s=\sup_{V\in\mathcal{G}(s,d)}
  \liminf_{n \to \infty} \frac{a_n(V)}{n}.$$
   Furthermore, if the supremum over $V$ is achieved by a subspace $V$
   then the supremum is also achieved by $A^jV$ for all $j\in
   \Z$.
\end{enumerate}
\end{corollary}

Finally, to contrast the random setting with the nonautonomous setting,
let us reexamine the nonautonomous Examples \ref{example1} and
\ref{example2}. There we found $\underaccent{\bar}\theta_1 < \bar \theta_1$
for the nonautonomous inner angular values in Diagrams \ref{D1} and \ref{D2}.
However, such a distinction is unnecessary in the random setting,
see Diagram \ref{D4}.
Therefore, the $0,1$ sequences underlying the choice
of matrices in Tables \ref{T1} and \ref{T2}
cannot occur for a set of full measure with an ergodic measure-preserving map.


\section{Angular values for the autonomous case}
\label{sec5}
A linear dynamical system, generated by  a single matrix $A \in
\mathrm{GL}(\R^d)$, fits into both -- the nonautonomous setting of
Section \ref{sec2} and the random setting of Section \ref{sec3}.
Therefore, the various angular values $\theta_s$ from Section
\ref{sec2} and $\vartheta_s$ from Section \ref{sec3} coincide.
Even for this case the computation of angular values turns out to be
nontrivial.
Since we will vary the matrix $A$, we write $\theta_s(A)$ to indicate
the dependence of the angular values on the matrix.

The following Corollary collects some equalities in Diagram
\ref{eq2.4} for autonomous systems.
\begin{corollary} \label{propauto}
  For an invertible autonomous system the following equalities hold
  for the values from Definition \ref{defangularvalues}
  \begin{equation} \label{coincideauto}
  \underaccent{\bar}{\theta}_{s}(A)= \bar{\theta}_{s}(A) =
  \bar{\theta}_{[s]}(A) =
\lim_{n \rightarrow \infty} \frac{1}{n} \sup_{V \in \mathcal{G}(s,d)}
                    \sum_{j=1}^n \ang(A^{j-1}V, A^j V).
      \end{equation}
  \end{corollary}
\begin{proof}
  The existence of the limit on the RHS of \eqref{coincideauto} is due
  to Part 1 of Corollary \ref{corauto}. Its existence can also be
  derived directly from Fekete's subadditive lemma.
  The fact that $\bar{\theta}_{s}(A)$ is equal to the limit on the RHS
  of \eqref{coincideauto} is by definition in Part 1 of
  \mbox{Corollary \ref{corauto}.}
  The equality $\underaccent{\bar}{\theta}_{s}(A)=
  \bar{\theta}_{s}(A)$ is trivial because we are discussing limits
  rather than limsup or liminf.
  The equality $\bar{\theta}_{s}(A) =
  \bar{\theta}_{[s]}(A)$ follows from Part $4$ of Theorem \ref{thmRDS},
    since $\Omega$ consists of a single point.
 It also follows from
    Definition \ref{defangularvalues} and the identity
    \begin{align*}
      a_{k+1,k+n}(V) = \sum_{j=k+1}^{k+n}\ang(A^{j-1}V,A^jV) =
      \sum_{\nu=1}^n \ang(A^{\nu-1}A^kV,A^{\nu}A^kV)= a_{1,n}(A^kV).
    \end{align*}
    \end{proof}

Next, we determine some explicit formulas for angular values in
the autonomous case.
Proposition \ref{prop2}(i) shows that we can assume $A$ to be in real Schur form
(cf.\ \cite[Theorem 2.3.4]{HJ2013}), i.e.\ $A$ is quasi-upper triangular
\begin{equation} \label{eq2.8}
  A = \begin{pmatrix}
    \Lambda_1 & A_{12} & \cdots & A_{1k} \\
    0 & \Lambda_2 & \cdots & A_{2k} \\
    \vdots & \ddots & \ddots & \vdots \\
    0 & \cdots & 0 & \Lambda_k
  \end{pmatrix}, \quad \Lambda_i \in \R^{d_i,d_i}, A_{ij}\in \R^{d_i,d_j},
\end{equation}
where either $d_i=1$  and $\Lambda_i=\lambda_i \in \R$ is a real eigenvalue or
$d_i=2$ and
\begin{equation} \label{eq2.9}
  \Lambda_i = \begin{pmatrix} \mathrm{Re}(\lambda_i) & -\frac{1}{\rho_i}
    \mathrm{Im}(\lambda_i) \\
    \rho_i \mathrm{Im}(\lambda_i) & \mathrm{Re}(\lambda_i)
  \end{pmatrix}, \quad 0<\rho_i \le 1, \quad
\end{equation}
for a complex eigenvalue $\lambda_i \in \C \setminus \R$.

\subsection{The two-dimensional case}
\label{sec5a}
Later on we use \eqref{eq2.8} to reduce the computation of
angular values to those of diagonal blocks. Therefore, we look at $2
\times 2$-matrices
first and compute $\bar{\theta}_1(A)$ in terms of the spectrum $\sigma(A)$.
This is already a nontrivial task.
Consider $A \in \R^{2,2}$ with complex conjugate eigenvalues
$\lambda,\bar{\lambda}$, $\mathrm{Im}(\lambda) >0$ and set
$\varphi=\mathrm{arg}(\lambda)$
where $\lambda=|\lambda| \mathrm{exp}(i \varphi), 0 < \varphi < \pi$.
By  orthogonal similarity transformations and a scaling with $|\lambda|^{-1}$
one can put $A$ into the normal form (see \eqref{eq2.9})
\begin{equation} \label{normalform}
  A(\rho,\varphi)= \begin{pmatrix} \cos(\varphi) & - \rho^{-1} \sin(\varphi)\\
    \rho \sin(\varphi) & \cos(\varphi) \end{pmatrix},  \quad
  0<\rho \le 1,\quad 0< \varphi<\pi.
\end{equation}
According to Proposition \ref{prop2}(i) these are the transformations which
leave all angular values invariant. Further, the matrix
$A(\rho,\varphi)$ leaves the ellipse $x^2 + \rho^{-2}y^2=1$ invariant, so
that $\rho \le 1$ can be achieved by a permutation.

Finally, we introduce the skewness of a matrix $A\in \mathrm{GL}(\R^d)$ by
\begin{equation*} \label{defskew}
  \mathrm{skew}(A) = \frac{1}{2 r(A)}\|A-A^{\top}\|, \quad r(A) = \max\{
  |\lambda|: \lambda \in \sigma(A) \}
  \end{equation*}
and note that this quantity is also invariant under scalings and orthogonal
similarity transformations. For the matrix \eqref{normalform} we have
$\mathrm{skew}(A(\rho,\varphi))=\frac{1}{2}(\rho + \rho^{-1})|\sin(\varphi)|$.

\begin{proposition} \label{lem2.3}
  For a matrix $A \in \mathrm{GL}(\R^2)$ all first
    angular values $\theta_1(A)$ with $\theta_1 \in \{
  \underaccent{\hat}\theta_{[1]}, \underaccent{\hat}\theta_{1},\hat
\theta_1,\hat\theta_{[1]},\underaccent{\bar}\theta_{[1]},
\underaccent{\bar}\theta_{1},\bar \theta_1,\bar\theta_{[1]} \}$
coincide. Moreover, the following holds:
  \begin{enumerate}
    \item[(a)]    If $A$ has only real eigenvalues then
        \begin{equation*} \label{eq2.10a}
          \theta_1(A)=
          \begin{cases}
      \frac{\pi}{2}, & \text{if} \; \sigma(A)=\{-\lambda,\lambda\}\subset
        \R, \lambda>0,\\
  0, & \text{otherwise.}
    \end{cases}
        \end{equation*}
      \item[(b)] If $\sigma(A)
        = \{\lambda,\bar{\lambda}\}, \mathrm{Im}(\lambda) \neq 0$ then
        \begin{equation} \label{complexgeneral}
         {\theta}_1(A)
          \le  \min(|\mathrm{arg}(\lambda)|,
          \pi-|\mathrm{arg}(\lambda)|).
        \end{equation}
        If additionally, $\mathrm{skew}(A) \le 1$ then we have equality, i.e.\
        \begin{equation} \label{skew}
         \theta_1(A)=  \min(|\mathrm{arg}(\lambda)|,
          \pi-|\mathrm{arg}(\lambda)|).
        \end{equation}
        \end{enumerate}
\end{proposition}

\begin{proof}
  By Proposition \ref{prop2}(i) we can assume $A$ to be in Schur form and scale
  $A$ such that the
  largest eigenvalue has (absolute) value $1$.
  Further we mention that $\bar{\theta}_1(A)=0$
    causes all other angular values to vanish by Corollary \ref{propauto}
    and Lemma \ref{lem2.2}.

  For $A=I$ the result is trivial and we are left with the cases
  \begin{equation} \label{Acases}
      A = \begin{cases} \begin{pmatrix} 1 & \eta \\ 0 & \lambda \end{pmatrix},
        & \begin{array}{ll} 0< |\lambda| < 1, \lambda\in \R, \eta \ge 0,
        & \text{case (i)}\\
     \lambda=1, \eta>0, & \text{case (ii)} \\
     \lambda=-1, \eta \ge 0, & \text{case (iii)}
     \end{array} \\
    A(\rho,\varphi), & \
    0<\rho \le 1, 0< \varphi <\pi, \hspace{0.97cm} \text{case (iv)}.
    \end{cases}
  \end{equation}
  It suffices to consider spaces $V=\mathrm{span}(v_0)$ where
  $v_0= \begin{pmatrix} \cos(\theta_0) \\ \sin(\theta_0) \end{pmatrix}$ and
  $|\theta_0| \le \frac{\pi}{2}$. We write the iterates in polar coordinates
  \begin{align} \label{polar}
    v_j = A^j v_0, \quad v_j =r_j
    \begin{pmatrix} \cos(\theta_j) \\ \sin(\theta_j)
  \end{pmatrix},
  \end{align}
    where $r_j = \|v_j\|$ and the angles $\theta_j\in \R$ will be determined
    appropriately. If $|\theta_j-\theta_{j-1}| \le \pi$ one finds
    that the angle between successive spaces is
    \begin{equation} \label{angvector}
      \ang(\mathrm{span}(v_{j-1}),\mathrm{span}(v_j)) = \chi(\theta_j-
      \theta_{j-1}), \quad \chi(x) := \min ( |x|, \pi-|x|).
    \end{equation}
    In the following we study the matrices from \eqref{Acases} case by case.
    \begin{itemize}
    \item[(i)]
      Since $|\lambda|<1$ the Blocking Lemma \ref{lemblock} below applies
      and reduces the formula to the one-dimensional case, i.e.\
      $\bar{\theta}_1(A)= \max(\bar{\theta}_1(1),\bar{\theta}_1(\lambda))=0$
      and similarly for $ \underaccent{\bar}{\theta}_1, \hat{\theta}_1,
  \underaccent{\hat}{\theta}_1$.
      Nevertheless, for later use and for the purpose of illustration
      we discuss the simple subcase $\eta=0<\lambda$ explicitly.
      In this case we obtain $|\theta_j| \le \frac{\pi}{2}$ for all $j \in \N$
      and the following formula
    \begin{equation} \label{angleiterate}
      \theta_j = \Psi_{\lambda}(\theta_{j-1}),
        \quad \Psi_{\lambda}(\theta)=
        \begin{cases}
           \arctan(\lambda \tan(\theta)), &
            |\theta| <  \frac{\pi}{2}, \\
                        \theta, & |\theta|= \frac{\pi}{2},
        \end{cases}
    \end{equation}
    cf.\ Figure \ref{PsiGamma}.
    For $\lambda>0$ we have $\Psi_{\lambda}'(\theta)>0$
    for all $|\theta|\le \frac{\pi}{2}$,
   $0< \Psi_{\lambda}(\theta)  <  \theta$
    for $\theta \in  (0,\frac{\pi}{2})$, and
    $0> \Psi_{\lambda}(\theta) > \theta$ for
    $\theta \in     (-\frac{\pi}{2},0)$.
    The values $\theta_j$ are monotone decreasing resp.\ increasing
    if $\theta_0>0$ resp.\ $\theta_0<0$, and therefore
    \begin{align} \label{esta1n}
     a_{1,n}= \sum_{j=1}^n
      \ang(\mathrm{span}(v_{j-1}),\mathrm{span}(v_j))= | \sum_{j=1}^n
      (\theta_{j-1}-\theta_j)|
      = |\theta_0 - \theta_n| \le \frac{\pi}{2}.
    \end{align}
    The assertion then follows from Proposition \ref{propequi} (a)
    with $\varphi=0$.
\begin{figure}[hbt]
    \begin{center}
      \includegraphics[width=0.75\textwidth]{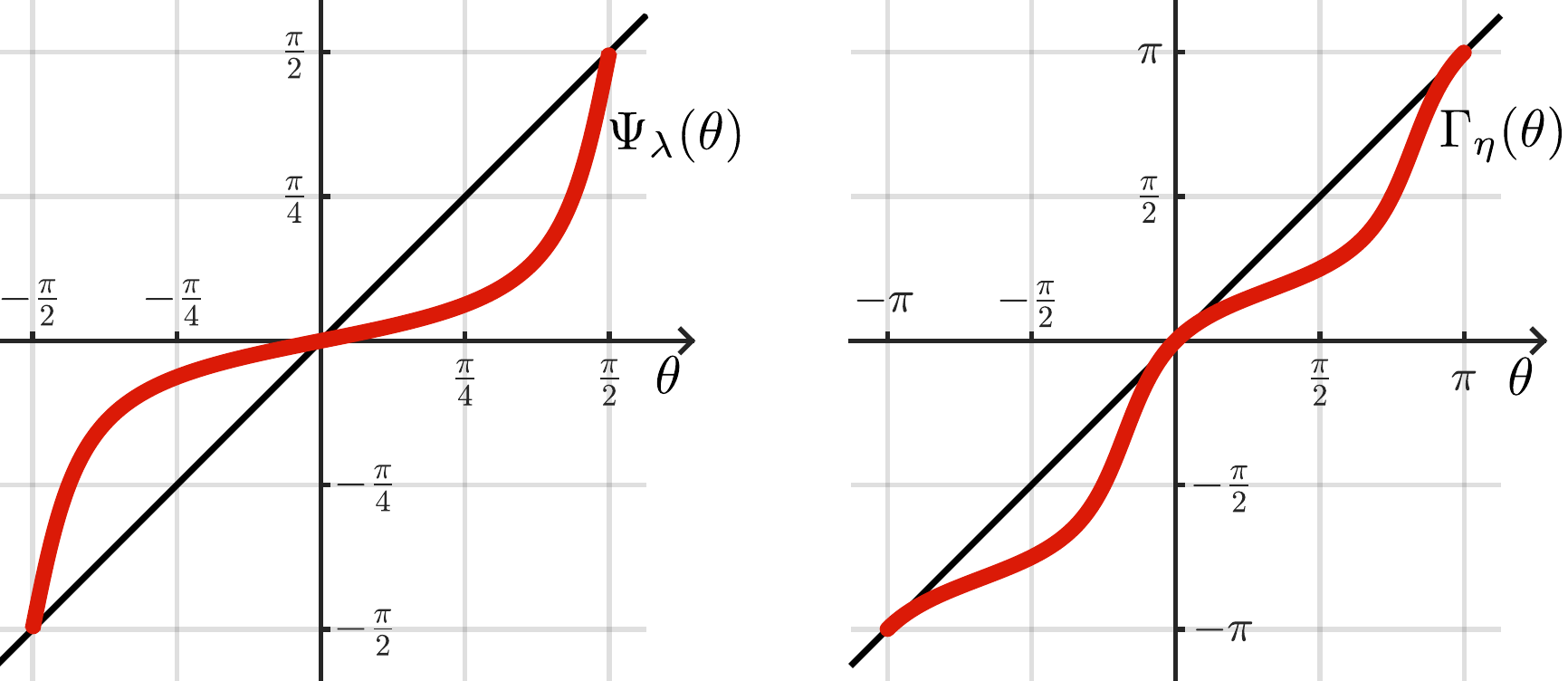}
    \end{center}
\caption{\label{PsiGamma} Graphs of $\Psi_\lambda$ for $\lambda = 0.2$
and of $\Gamma_\eta$ for $\eta = 1$.}
\end{figure}

\item[(ii)]
    For the matrix in \eqref{Acases} (ii) we obtain
    \begin{equation} \label{Jordaniterate}
      \theta_j = \Gamma_{\eta}(\theta_{j-1}),
        \quad \Gamma_{\eta}(\theta)=
        \begin{cases}
           \mathrm{arccot}(\eta+ \cot(\theta)), &
           0 < \theta < \pi, \\
           \mathrm{arccot}_{-1}(\eta+ \cot(\theta)), &-\pi< \theta < 0,\\
           \theta, & \theta=-\pi,0,\pi,
        \end{cases}
    \end{equation}
    where $\mathrm{arccot}_{-1}$ is the first negative branch of
    $\mathrm{arccot}$, see Figure \ref{PsiGamma}.
    The function $\Gamma_{\eta}$  is strictly monotone increasing and satisfies
    $\Gamma_{\eta}(\theta)< \theta$ for $0 < |\theta|< \pi$. Therefore,
    the sequence $\theta_j$
    is monotone decreasing and converges to $0$ if $0\le \theta_0< \pi$ and to
    $-\pi$ if $\theta_0<0$. Thus the minimum in
    \eqref{angvector} is achieved at $|\theta_j - \theta_{j-1}|$  and we
    obtain as in \eqref{esta1n}
    \begin{align*}
      a_{1,n}= \sum_{j=1}^n \ang(\mathrm{span}(v_{j-1}),\mathrm{span}(v_j))\le
      | \sum_{j=1}^n (\theta_{j-1}-\theta_j)|
      = |\theta_0 - \theta_n| \le \pi.
      \end{align*}
    \item[(iii)]
      The third case describes a reflection which satisfies $A^2=I$.
      Moreover, we find
        \begin{align*}
  v_0^{\top}Av_0= \cos(2 \theta_0)- \eta \sin(2 \theta_0)
        \end{align*}
        which vanishes for $\theta_0=\frac{\pi}{4}$ if $\eta=0$, and otherwise for
        \begin{align*}
          \theta_0 =\frac{1}{2} \arctan(\eta^{-1}) \in \big(0,\frac{\pi}{4}\big).
        \end{align*}
        Then we have $\ang(v_0,Av_0)=\frac{\pi}{2}=\ang(A^jv_0,A^{j-1}v_0)$
        for all $j \ge 1$. Since $\frac{\pi}{2}$ is the maximum possible
        angular value our assertion is proved.
    A reflection turns out to have the same angular value as a
    rotation by $\frac{\pi}{2}$.
  \item[(iv)]
    In \eqref{Acases} we can assume
    $\varphi \le \frac{\pi}{2}$ since $A(\rho,\varphi)$ is orthogonally
    similar to $- A(\rho,\pi - \varphi)$.
   For this rotational case we use ergodic theory
    and employ almost periodicity; see \cite[Ch.4.1, Remarks 1.3-1.7]{P89}.
  We extend the function $ \Psi_{\rho}$ defined in \eqref{angleiterate}
  from $[-\frac{\pi}{2},\frac{\pi}{2}]$ to $\R$ by setting
  \begin{align} \label{liftPsi}
    \Psi_{\rho}(\theta + n\pi) = \Psi_{\rho}(\theta)+ n \pi, \quad
    |\theta| \le \frac{\pi}{2}, n  \in \Z \setminus \{0\}.
  \end{align}
  For this extended function there exists a constant $C_{\rho}>0$ such that
  \begin{equation}\label{estPsi}
    |\Psi_{\rho}(x)- x |, |\Psi_{\rho}'(x)|, |\Psi_{\rho}''(x)| \le C_{\rho}
    \quad \text{for all} \quad x \in \R.
  \end{equation}
    The factorization
  \begin{align*}
    \begin{pmatrix} \cos(\varphi) & - \rho^{-1} \sin(\varphi)\\
      \rho \sin(\varphi) & \cos(\varphi) \end{pmatrix}=
\begin{pmatrix} 1& 0 \\ 0 & \rho \end{pmatrix}
    \begin{pmatrix} \cos(\varphi) & -  \sin(\varphi)\\
      \sin(\varphi) & \cos(\varphi) \end{pmatrix}
    \begin{pmatrix} 1 & 0 \\ 0 & \rho^{-1} \end{pmatrix}
  \end{align*}
    shows that the angles $\theta_j\in \R$ in \eqref{polar} are accumulated according to
  \begin{align} \label{itertheta}
    \theta_j =F_{\rho,\varphi}(\theta_{j-1}), j \in \N, \quad
    F_{\rho,\varphi}(\theta):=\Psi_{\rho}(\varphi + \Psi_{\rho^{-1}}(\theta)).
  \end{align}
  The new variables $\varphi_j = \varphi + \Psi_{\rho^{-1}}(\theta_j)$
  then satisfy the recursion
  \begin{align*}
    \varphi_j= \varphi + \Psi_{\rho^{-1}}(\Psi_{\rho}(\varphi +
    \Psi_{\rho^{-1}}(\theta_{j-1})))= \varphi+ \varphi_{j-1},
  \end{align*}
  hence $\varphi_j = \varphi_0+ j \varphi= (j+1)\varphi + \Psi_{\rho^{-1}}(\theta_0)$ and
  \begin{equation} \label{thetaformula}
    \theta_j = \Psi_{\rho}(j\varphi + \Psi_{\rho^{-1}}(\theta_0)).
  \end{equation}
  In particular, the values $\theta_j$ are monotone increasing.
  From \eqref{thetaformula} and \eqref{estPsi} we infer
\begin{equation} \label{ethetaless}
  \begin{aligned}
    \frac{1}{n} a_{1,n} & = \frac{1}{n} \sum_{j=1}^n \chi(\theta_{j-1}-\theta_j)
    \le \frac{1}{n} \sum_{j=1}^n (\theta_j-\theta_{j-1})
     = \frac{1}{n}(\theta_n- \theta_0) \\ & = \varphi +
    \frac{1}{n}\big( \Psi_{\rho}(\varphi_n- \varphi)-(\varphi_n- \varphi)
    + \Psi_{\rho^{-1}}(\theta_0)- \theta_0 \big)\\
    & \le \varphi + \frac{C_{\rho}+C_{\rho^{-1}}}{n}.
   \end{aligned}
 \end{equation}
This will lead to the estimate \eqref{complexgeneral} as $n \to
\infty$ provided we have
  shown the equality of all angular values. For this purpose we apply
  Proposition \ref{propequi}(b) where we identify $V \in
    \cG(1,2)$ with $\theta+ 2 \pi \Z \in S^{2 \pi}=\R/(2 \pi \Z)$ via $V=
\mathrm{span}\begin{pmatrix} \cos(\theta) \\
  \sin(\theta) \end{pmatrix}$.
The function $\Psi_{\rho}$
  is a lift of the circle map $\psi_{\rho}: S^{2 \pi}\to S^{2 \pi}$
  defined by $\psi_{\rho}(\theta + 2 \pi \Z)= \Psi_{\rho}(\theta) + 2 \pi \Z$.
  Further, the iteration \eqref{itertheta} may be written by means of  a circle
  map $T_{\rho,\varphi}:S^{2 \pi} \to S^{2\pi}$
  as follows
  \begin{align} \label{Tconj}
    \theta_j+2 \pi \Z =T_{\rho,\varphi}(\theta_{j-1}+ 2 \pi \Z), \quad
    T_{\rho,\varphi}= \psi_{\rho}\circ \tau_{\varphi} \circ
    \psi_{\rho}^{-1},
  \end{align}
  where the shift
  $\tau{_\varphi}:S^{2 \pi} \to S^{2\pi}$ is defined by $
  \tau_{\varphi}(\theta+2 \pi \Z)= \theta+\varphi+ 2 \pi \Z$. The map
  $F_{\rho,\varphi}$ in \eqref{itertheta}
  is then a lift of $T_{\rho,\varphi}$.
  It is well known (see  \cite[Ch.2.6.2]{Ba12}) that
  $\tau_{\varphi}$ is an ergodic isometry of $S^{2 \pi}$ with respect
  to
  Lebesgue measure $\mu_1$ and the standard metric
  \begin{align*}
    d_1(\theta_1+ 2 \pi \Z, \theta_2 + 2 \pi \Z)= \min_{z \in \Z}|\theta_1 -
    \theta_2 + 2 \pi z|
  \end{align*}
  if and only if $\frac{\varphi}{\pi} \notin \Q$. In this case the
  conjugacy \eqref{Tconj} implies that $T_{\rho,\varphi}$ is an ergodic isometry
  of $S^{2 \pi}$ with respect to the image measure $\mu_{\rho}=\mu_1 \circ
  \psi_{\rho}^{-1}$ and the image metric
  $ d_{\rho}(\cdot, \cdot) = d_1(\psi_{\rho}^{-1} \cdot,\psi_{\rho}^{-1} \cdot)$.
  We conclude from \cite[Remark 1.3]{P89} that the map $T_{\rho,\varphi}$ is uniformly
  almost periodic, i.e.\ for every $\varepsilon>0$ there exists a
  relatively dense set $\cP \subseteq \N_0$ such that
  $d_{\rho}(x,T_{\rho,\varphi}^px) \le \varepsilon$ for all $x \in
  S^{2 \pi}$, $p \in \cP$. Moreover, for
  any continuous function $g:S^{2 \pi}\to \R$ the sequence of functions
  $b_n(x)=g(T_{\rho,\varphi}^{n-1}x),
  x \in S^{2 \pi}$, $n \in \N$ is uniformly almost periodic in the sense
  of Definition \ref{unifap}. To see this, let $\eps_0 >0$ be given
  and take $\eps>0$
  such that $|g(x_1)-g(x_2)| \le \eps_0$ whenever $d_{\rho}(x_1,x_2) \le \eps$,
  $x_1,x_2 \in S^{2 \pi}$. For the relatively dense set $\cP \subset \N$
  belonging to $\eps$ we then find
  \begin{align*}
    |b_n(x)-b_{n+p}(x)|=|g(T_{\rho,\varphi}^{n-1}x)-
    g(T_{\rho,\varphi}^p(T_{\rho,\varphi}^{n-1}x))| \le \eps_0 \quad
    \forall
    n \in \N, p \in \cP, x \in S^{2 \pi}.
  \end{align*}
  In the case $\frac{\varphi}{\pi}\in \Q$ we have the same result since
  then every point $x \in S^{2 \pi}$ has the same period $q$ where
  $\frac{\varphi}{\pi}=\frac{2 p}{q}$.

  Let us apply this to the continuous function
  \begin{equation} \label{chooseg}
    g(x) =
    \min(d_1(x,T_{\rho,\varphi}x),d_1(\tau_{\pi}x,T_{\rho,\varphi}x)),
    \quad x \in S^{ 2\pi}.
  \end{equation}
  Setting $x=\theta_0+ 2 \pi \Z$ we obtain
  \begin{align*}
    T_{\rho,\varphi}^{j-1}x & = \theta_{j-1}+ 2 \pi \Z ,  \quad j \in \N.
  \end{align*}
  Using $\theta_{j-1} < \theta_j \le \theta_{j-1}+\pi$ and \eqref{angvector}
  for $j \in \N$ then  leads to
  \begin{equation} \label{bjexpress}
  \begin{aligned}
    b_j(x) & = g(T_{\rho,\varphi}^{j-1}x) = \min(\theta_j - \theta_{j-1}, \theta_{j-1}+ \pi -
    \theta_j)\\
    & = \chi(\theta_j - \theta_{j-1})= \ang(\mathrm{span}(v_{j-1}),\mathrm{span}(v_j)).
  \end{aligned}
  \end{equation}
  Therefore, all angular values agree by Proposition \ref{propequi} (b).

  Next we show
  that the assumption $\mathrm{skew}(A)= \frac{1}{2}(\rho + \rho^{-1})
  |\sin(\varphi)| \le 1 $ implies $\theta_j-\theta_{j-1} \le \frac{\pi}{2}$. Then
  the minimum in \eqref{angvector} is always achieved with the first term and
  the first inequality in \eqref{ethetaless} becomes an equality. Thus we find
  \begin{align*}
    |\frac{1}{n} a_{1,n} - \varphi| \le \frac{C_{\rho}+C_{\rho^{-1}}}{n},
  \end{align*}
  and Proposition \ref{propequi}(a) implies the assertion.
  It remains to analyze the inequality
  \begin{align} \label{lesspi/2}
   F_{\rho,\varphi}(\theta)= \Psi_{\rho}(\varphi+
    \Psi_{\rho^{-1}}(\theta)) \le \theta + \frac{\pi}{2},
    \quad \theta \in \R.
  \end{align}
  For later purposes we perform a rather explicit calculation.
  First note that it is enough to consider $0<|\theta| < \frac{\pi}{2}$
  since $F_{\rho,\varphi}(\theta+n \pi)=F_{\rho,\varphi}(\theta)+n \pi$ holds
  by \eqref{liftPsi}  and since \eqref{lesspi/2}
  is obvious for $\theta=0,\pm \frac{\pi}{2}$.
   By the monotonicity of $\Psi_{\rho^{-1}}$ and the
   sum formula\footnote{$\arctan(x)+\arctan(y)= \mathrm{sgn}(x)\pi -
     \arctan(\frac{x+y}{xy-1}$) for $x\neq 0,xy>1$.} for $\arctan$ we obtain
   that $F_{\rho,\varphi}(\theta) \le \theta + \frac{\pi}{2} $ holds
   for $0< |\theta| < \frac{\pi}{2}$ if and only if
   \begin{equation*}
     \begin{aligned}
        \varphi & \le \Psi_{\rho^{-1}}(\theta+ \frac{\pi}{2}) -
        \Psi_{\rho^{-1}}(\theta) = \begin{cases} r(\theta,\rho), & 0<\theta < \frac{\pi}{2}, \\
          \pi + r(\theta,\rho), & -\frac{\pi}{2}< \theta <0,
        \end{cases} \\
        r(\theta,\rho) & := \arctan\Big(\frac{\tan(\theta)+
          \frac{1}{\tan(\theta)}}{\rho^{-1}-\rho}\Big) =
        \arctan\Big(\frac{2}{(\rho^{-1}- \rho) \sin(2 \theta)} \Big).
     \end{aligned}
   \end{equation*}
   In the case $\theta <0$ this inequality always holds since $\varphi \le
   \frac{\pi}{2}$, while for $\theta>0$ it is equivalent to
   \begin{equation} \label{equivangle}
     \sin(2 \theta) \le \frac{2}{\tan(\varphi)(\rho^{-1}-\rho)}=: \beta(\rho,\varphi).
     \end{equation}
   Expressing $\tan(\varphi)$ in terms of $\sin(\varphi)=\frac{2\,
     \mathrm{skew}(A)}{\rho^{-1}+ \rho}$ leads to
   \begin{equation} \label{defbeta}
     \beta(\rho,\varphi) = \Big( 1 + \frac{4(1
       -\mathrm{skew}(A)^2)}{(\rho^{-1}-\rho)^2} \Big)^{1/2}.
   \end{equation}
   Hence condition \eqref{equivangle} holds for all $\theta \in \R$ if
   $\mathrm{skew}(A) \le 1$.
    \end{itemize}

\end{proof}
\begin{remark} \label{rotnumber} Let us relate the result of
  Proposition \ref{lem2.3} to the theory of rotation numbers; see
  \cite[Ch.11]{KH95}. First, note that this theory uses
    $[0,1)$ instead of $[0, 2\pi)$ as the interval of
    periodicity. Every matrix $A\in \mathrm{GL}(\R^2)$ induces
  a homeomorphism $f:S^1 \to S^1$ of  $S^1=\R / \Z$ via
  the relation (one step of the iteration \eqref{polar})
  \begin{align} \label{polarf}
    v = \| v \| \begin{pmatrix} \cos(2 \pi x) \\ \sin(2 \pi x) \end{pmatrix}
    \mapsto  Av = \|Av\| \begin{pmatrix} \cos(2 \pi f(x)) \\
      \sin(2 \pi f(x))
    \end{pmatrix}, \quad x \in S^1.
  \end{align}
  The homeomorphism is orientation-preserving if and only if
  $\det(A)>0$. For such a homeomorphism the rotation number
  $\tau(f)\in[0,1)$ is
      well defined. Iterating \eqref{polarf} and comparing with \eqref{polar}
      then shows the equality $2 \pi \tau(f)= \hat{\theta}_1(A)$, provided
    no vector rotates by more than $\frac{\pi}{2}$.
    For the matrices in \eqref{Acases} these conditions hold in case (i)
    if $\lambda>0$, in case (ii), and in case (iv) if
      $\mathrm{skew}(A) \le 1$ (see
      \eqref{itertheta}). The corresponding $f$-maps are
        $2 \pi f(x)=\Psi_{\lambda}(2 \pi x)$ (see case (i),
        $\lambda >0$, $\eta=0$, equation \eqref{angleiterate}), $2 \pi
        f(x)=\Gamma_{\eta}(2 \pi x)$ (case (ii), equation
        \eqref{Jordaniterate}), and
        $2 \pi f(x)= \Psi_{\rho}(\varphi+ \Psi_{\rho^{-1}}(2 \pi x))$ (case (iv)).
           Determining the exact first angular value $\theta_1(A)$ in
           case $\mathrm{skew}(A)>1$ of (a) is more
             involved. In Theorem \ref{prop5.1} we will
             show that the inequality \eqref{complexgeneral} is generally strict except
      for some resonant values of $\varphi=\mathrm{arg}(\lambda)$.
\end{remark}

\subsection{Systems of higher dimension}
\label{sec5b}
As a first step we consider a matrix with a single eigenvalue which
generalizes the second case in \eqref{Acases}. Its proof is stated in
the Supplementary materials \ref{app3}.
\begin{proposition} \label{prop2.5}
  Assume that the spectrum of $A \in \R^{d,d}$ consists of one
  eigenvalue $\lambda \in \R$, $\lambda\neq0$.
  Then all first angular values vanish, i.e.\
  \begin{equation*}
    \theta_1(A)=0 \quad \text{for} \quad \theta_1 \in \{
  \underaccent{\hat}\theta_{[1]}, \underaccent{\hat}\theta_{1}, \hat
\theta_1, \hat\theta_{[1]}, \underaccent{\bar}\theta_{[1]},
\underaccent{\bar}\theta_{1}, \bar \theta_1, \bar\theta_{[1]}\}.
\end{equation*}
 \end{proposition}

To proceed further, we require the following lemma.
\begin{lemma} (Blocking Lemma)
  \label{lemblock}
  Let $\R^d = X_s \oplus X_u$ be a decomposition into
  invariant subspaces of $A\in \R^{d,d}$ such that $A_s=A_{|X_s}$ and $A_u=A_{|X_u}$ satisfy
  \begin{equation} \label{eq:split}
    |\sigma(A_s)| < | \sigma(A_u)|.
  \end{equation}
  Then  the following holds for all types of angular values
  $\theta_1\in \{ \bar{\theta}_1, \underaccent{\bar}{\theta}_1, \hat{\theta}_1
  ,\underaccent{\hat}{\theta}_1 \}$
  \begin{equation} \label{eqreduce}
    \theta_1(A) = \max (\theta_1(A_s),\theta_1(A_u)).
  \end{equation}
\end{lemma}
We refer to the Supplementary materials \ref{app4} for a proof of the
Blocking Lemma.

\begin{remark}\label{remblockuniform}
  By Corollary \ref{propauto} it is clear that formula \eqref{eqreduce}
  also holds for the uniform first angular value $\bar{\theta}_{[1]}(A)$.
  We did not succeed in proving this for the remaining three uniform
  first angular values. However, we will be able to treat these three
  values in the subsequent main Theorem \ref{cor2.4}  under a special assumption.
  \end{remark}

The following Theorem combines the results of Propositions \ref{lem2.3},
\ref{prop2.5} and Lemma \ref{lemblock}.

\begin{theorem} \label{cor2.4}
  Let the spectrum of  $A\in \mathrm{GL}(\R^d)$
  satisfy
    \begin{equation} \label{eqeigencond}
    \begin{aligned}
     \lambda \in \sigma(A), \lambda \notin \R \Longrightarrow \lambda
     \text{ is simple and }
    |\eta| \neq |\lambda| \quad \forall \eta \in \sigma(A)\setminus
    \{\lambda,\bar{\lambda}\}.
    \end{aligned}
    \end{equation}
    Then all  $8$ types of angular values
    $\theta_1(A)$ with $\theta_1\in \{
  \underaccent{\hat}\theta_{[1]}, \underaccent{\hat}\theta_{1}, \hat
\theta_1, \hat\theta_{[1]}, \underaccent{\bar}\theta_{[1]},
\underaccent{\bar}\theta_{1}, \bar \theta_1, \bar\theta_{[1]}\}$
     coincide.
  Let $\R^d = \bigoplus_{i=1}^k \range(Q_i)$, $Q_i \in \R^{d,d_i}$,
  $Q_i^{\top}Q_i = I_{d_i}$  be a
  decomposition of $\R^d$ into invariant subspaces of $A$ corresponding to
  eigenvalues of equal modulus, i.e.\
       \begin{equation} \label{orderAi}
          \begin{aligned}
            A Q_i & = Q_i A_i, \quad
            A_i \in   \R^{d_i,d_i}, \quad
            |\sigma(A_1)|, \ldots, |\sigma(A_k)| \text{ pairwise different.}
          \end{aligned}
       \end{equation}
    Then the following equality holds
    \begin{equation} \label{angcomplex}
    \theta_1(A)= \max_{i=1,\ldots,k}\theta_1(A_i).
    \end{equation}
    If there exist two real
    eigenvalues of opposite sign in $\sigma(A)$ then
    $\theta_1(A)=\frac{\pi}{2}$.
    Otherwise, the following estimate holds
             \begin{equation} \label{angestgeneral}
               \theta_1(A) \le \max_{\lambda \in \sigma(A)}
               \min(|\mathrm{arg}(\lambda)|, \pi -
               |\mathrm{arg}(\lambda)|).
             \end{equation}
  Equality holds  if the maximum on the right-hand
  side of \eqref{angestgeneral} is zero or if it is achieved for an eigenvalue
  $\lambda_{i_0}\in \sigma(A_{i_0})$,
  $i_0\in \{1,\ldots,k\}$  with  $\mathrm{Im}(\lambda_{i_0})\neq 0$
  and  $\mathrm{skew}(A_{i_0}) \le 1$; \ i.e.\
               \begin{equation} \label{anglemax}
          \theta_1(A)=  \min(|\mathrm{arg}(\lambda_{i_0})|,
          \pi-|\mathrm{arg}(\lambda_{i_0})|).
        \end{equation}
\end{theorem}
\begin{remark} \label{remopposite}
 For the formulas \eqref{angcomplex} and \eqref{anglemax} it is
 essential to choose orthonormal
  bases for the invariant subspaces. Other bases will preserve the spectra of
  the matrices $A_i$ but neither the values $\mathrm{skew}(A_i)$ nor the
  angular values $\theta_1(A_i)$, see Proposition \ref{lem2.3} (b).
  Except for the first block $\Lambda_1$, the angular values of
  $\Lambda_i$ in the Schur form \eqref{eq2.8} generally do not agree with
  $\theta_1(A_i)$, see Algorithm \ref{algo} and the example in Section
  \ref{4Dillu}.
\end{remark}

\begin{proof}
  Let us first prove \eqref{angcomplex} for all $4$ nonuniform types of angular value.
  Note that a decomposition $\R^d = \bigoplus_{i=1}^k \range(Q_i)$ of
  the desired type  always
  exists since we can decompose $\sigma(A)$ into subsets of equal modulus
  and then select an  orthogonal basis for each of the corresponding invariant
  subspaces. In this way we transform
  $A$ into block-diagonal form in a specific way (see \eqref{orderAi}),
  \begin{align} \label{specialblock}
    A \begin{pmatrix} Q_1 & \cdots & Q_k \end{pmatrix}=
    \begin{pmatrix} Q_1 & \cdots & Q_k \end{pmatrix} \mathrm{diag}(A_1, \ldots, A_k).
  \end{align}
  If one does not insist on orthonormal bases for the subspaces then
  one can keep the diagonal blocks $\Lambda_i$ from the Schur form; see
  \cite[Thm 7.1.6]{GvL2013}.
    From $Q_i^{\top}Q_i=I_{d_i}$ we obtain for every $v_i  \in
    \R^{d_i}, v_i \neq 0$, $i=1,\ldots,k$, $j \in \N$,
  \begin{align*}
    \ang(A^{j-1}Q_i v_i, A^j Q_i v_i) = \ang(Q_iA_i^{j-1}v_i,Q_i A_i^j v_i) =
    \ang(A_i^{j-1}v_i, A_i^j v_i),
  \end{align*}
  so that all first angular values of $A_i$ and of the restriction $A_{|\range(Q_i)}$
  coincide. Hence Lemma \ref{lemblock} shows \eqref{angcomplex}.
  Note that \eqref{angcomplex} also holds for the $\bar{\theta}_{[1]}$-values
  by Corollary \ref{propauto}.

  Now assume that there exist two real eigenvalues $\lambda, -\lambda
  \in \sigma(A)$ and  w.l.o.g. assume $\lambda,-\lambda \in
  \sigma(A_1)$.
  Then there exists an orthogonal $S \in \R^{d_1,d_1}$ and some $\eta \in \R$ such that
  \begin{align*}
    A_1 S = SM, \quad M= \begin{pmatrix} M_{11} & M_{12} \\ 0 & M_{22} \end{pmatrix},
    M_{11}= \begin{pmatrix} \lambda & \eta \\0 & -\lambda \end{pmatrix}.
  \end{align*}
  The first two columns of $S$ form an orthonormal basis of the span
  of eigenvectors which belong to $\lambda$ and $- \lambda$. Choosing
  initial vectors
  $v_1=(v_0,0,\ldots,0)^{\top}\in \R^{d_1}, v_0 \in \R^2$ we find for all
  $j \in \N$
  \begin{align*}
    \ang(A_1^{j-1}Sv_1,A_1^j Sv_1) = \ang(SM^{j-1}v_1,SM^jv_1)=
    \ang(M^{j-1}v_1,M^jv_1)
      =\ang(M_{11}^{j-1}v_0,M_{11}^j v_0).
  \end{align*}
  The proof of Proposition \ref{lem2.3}(a) (see case (iii) in \eqref{Acases})
  shows that there exists  $v_0 \in \R^2$, $v_0 \neq0$ such that for all $j \in \N$
  \begin{align*}
    \frac{\pi}{2}=\ang(M_{11}^{j-1}v_0,M_{11}^j v_0)=\ang(A_1^{j-1}Sv_1,A_1^j Sv_1)=
    \ang(A^{j-1}Q_1S v_1, A^j Q_1 S v_1).
  \end{align*}
  Since $\frac{\pi}{2}$ is the maximum of all angular values, Definition
  \ref{defangularvalues} implies that all $8$ types of angular
  values are equal to $\frac{\pi}{2}$.

  If such a pair of real eigenvalues does not exist then
  assumption \eqref{eqeigencond} shows that the matrices $A_i$ either are
  two-dimensional as in Proposition \ref{lem2.3} (see case (iv) in
  \eqref{Acases}) or have
  a single real eigenvalue as in Proposition \ref{prop2.5}.
  In both cases the propositions guarantee all first angular values
  of the matrices $A_i$ to coincide. Thus  the four nonuniform
  angular values of the given matrix $A$ are equal by \eqref{angcomplex}.
  For $\bar{\theta}_{[1]}(A)$ the result then
  follows from Corollary \ref{propauto}. Moreover,
  Lemma \ref{lem2.2} yields formula \eqref{angcomplex} and the coincidence
  of the
  $\hat{\theta}_{[1]}$-values:
  \begin{align*}
    \bar{\theta}_{1}(B)=
     \hat{\theta}_{1}(B)
    \le \hat{\theta}_{[1]}(B) \le \bar{\theta}_{[1]}(B)=
    \bar{\theta}_{1}(B), \quad B \in \{A,A_i(i=1,\ldots,k)\}.
  \end{align*}
  Next we show $ \underaccent{\hat}\theta_{[1]}(A)= \underaccent{\hat}\theta_{1}(A)$.
  From \eqref{angcomplex} we  find an index $\ell \in \{1,\ldots,k\}$
  for which $\underaccent{\hat}\theta_{1}(A)
  =\underaccent{\hat}\theta_{1}(A_{\ell})$ holds. Then we use Lemma
  \ref{lem2.2} and the equality of angular values
  from Propositions \ref{lem2.3} and \ref{prop2.5},
   \begin{align*}
     \underaccent{\hat}\theta_{1}(A)&=\underaccent{\hat}\theta_{[1]}(A_{\ell})
     = \sup_{V_{\ell} \in \mathcal{G}(1,d_{\ell})}\liminf_{n \to \infty} \frac{1}{n}
     \inf_{k \in \N_0}\sum_{j=k+1}^{k+n} \ang(Q_{\ell}A_{\ell}^{j-1}V_{\ell},Q_{\ell}A_{\ell}^jV_{\ell})\\
   & = \sup_{V_{\ell} \in \mathcal{G}(1,d_{\ell})}\liminf_{n \to \infty} \frac{1}{n}
     \inf_{k \in \N_0}\sum_{j=k+1}^{k+n} \ang(A^{j-1}Q_{\ell}V_{\ell},A^jQ_{\ell}V_{\ell})
\\ &\le \sup_{V \in \mathcal{G}(1,d)}\liminf_{n \to \infty} \frac{1}{n}
\inf_{k \in \N_0}\sum_{j=k+1}^{k+n} \ang(A^{j-1}V,A^jV) =
 \underaccent{\hat}\theta_{[1]}(A) \le \underaccent{\hat}\theta_{1}(A).
   \end{align*}
   Using Lemma \ref{lem2.2} we obtain the result for the last angular value
    $\underaccent{\bar}\theta_{[1]}(A)$:
\begin{align*}
    \underaccent{\hat}\theta_{1}(B)=
     \underaccent{\hat}\theta_{[1]}(B)
     \le \underaccent{\bar}\theta_{[1]}(B) \le
  \underaccent{\bar}\theta_{1}(B)=
   \underaccent{\hat}\theta_{1}(B), \quad B \in \{A,A_i(i=1,\ldots,k)\}.
  \end{align*}

  Finally, the estimate
  \eqref{angestgeneral} follows from
  \eqref{complexgeneral} and Proposition \ref{prop2.5}.
  If the maximum value on the right of \eqref{angestgeneral} is zero then the
  assertion \eqref{anglemax} is obvious. Otherwise, it follows from
  \eqref{angcomplex} and \eqref{skew}
  in Proposition \ref{lem2.3} when applied to the $2 \times 2$
  matrix $A_{i_0}$.
    \end{proof}
Note that condition \eqref{eqeigencond} excludes a complex eigenvalue of
multiplicity $\ge 2$ and another eigenvalue of the same modulus.
Let us consider such an exceptional case, namely a block diagonal matrix
with two rotations
  \begin{equation} \label{blockdiag}
    A= \begin{pmatrix} T_{\varphi_1} & 0 \\ 0 & T_{\varphi_2} \end{pmatrix},
    \quad 0 \le \varphi_1,\varphi_2 \le \frac{\pi}{2}.
     \end{equation}
  We claim that every type of angular value is given by
  \begin{equation*} \label{angblock}
    \theta_1(A) = \max(\varphi_1,\varphi_2).
  \end{equation*}
    Let $v=\begin{pmatrix} v_1 & v_2 \end{pmatrix}^{\top}$, $v_1, v_2 \in \R^2$
  and $|v_1|^2 + |v_2|^2=1$. Since $A$ is orthogonal we obtain
  \begin{align*}
    \cos(\ang(Av,v)) & =  |\langle Av,v \rangle| =
    | \langle T_{\varphi_1}v_1,v_1 \rangle +
    \langle  T_{\varphi_2}v_2,v_2 \rangle|
    \\
    & =\cos(\varphi_1)|v_1|^2 + \cos(\varphi_2)|v_2|^2=
    |v_1|^2(\cos(\varphi_1)-\cos(\varphi_2)) + \cos(\varphi_2).
  \end{align*}
    By the orthogonal invariance of the angle (see Proposition \ref{prop2})
 this leads to
  \begin{align} \label{expressum}
    \lim_{n \to \infty} \frac{1}{n} \sum_{j=1}^n \ang(A^{j-1}v,A^jv)
    =\arccos\big(|v_1|^2(\cos(\varphi_1)-\cos(\varphi_2)) + \cos(\varphi_2)\big).\end{align}
  Suppose w.l.o.g.\ that $\varphi_2 \ge \varphi_1$ so that
  $\cos(\varphi_1)-\cos(\varphi_2)\ge 0$. Then the maximum w.r.t.\ $v$ in \eqref{expressum}
   occurs for $|v_1|=0$, hence
    $\theta_1(A) = \varphi_2$.
  The same argument applies to a block diagonal matrix with $k$ blocks
  $T_{\varphi_i},i=1,\ldots,k$  on the diagonal, leading to
  $\theta_1(A) = \max_{i=1,\ldots,k} \varphi_i$. However, we did not find
  a formula for $\theta_1(A)$ in cases which violate \eqref{eqeigencond}
  but which are more general than \eqref{blockdiag}.

\section{Numerical algorithms and results}
\label{sec6}
In this section, our main goal is to discuss algorithms for the computation of the first
outer angular value $\hat{\theta}_1(A)$ of an autonomous system
generated by a matrix
$A\in\R^{d,d}$. First we investigate the
two-dimensional case where our focus is on matrices with
$\mathrm{skew}(A)>1$.
We extend the theory underlying Proposition \ref{lem2.3} and compare with
numerical computations.

Then we use the results from Lemma \ref{lemblock} and
Theorem \ref{cor2.4} to develop an algorithm for
matrices of arbitrary dimension.
Let us emphasize that the whole  calculation aims at first outer
angular values. In the autonomous case we know the coincidence with
inner angular values by Theorem \ref{cor2.4}. However, for general
nonautonomous systems the computation of inner angular values  turns
out to be quite involved since one has to solve an optimization
problem in every time step.

Let us also note that simple algorithms based on
  subspace iterations tend to fail. The forward iteration of a generic one-dimensional
subspace converges to the most unstable direction. However, we must
consider also non-generic directions, i.e.\ all invariant subspaces, in
order to compute $\hat\theta_1(A)$.

\subsection{Two dimensional autonomous examples}
\label{sec5.2.2}
Consider the normal form \eqref{normalform} with increasing skewness
and recall from Proposition \ref{lem2.3} that all angular values coincide,
\begin{align}\label{rhomatrix}
     A(\rho,\varphi) =\begin{pmatrix} \cos(\varphi) & - \rho^{-1} \sin(\varphi)\\
    \rho \sin(\varphi) & \cos(\varphi) \end{pmatrix}, \quad
     0<\rho \le 1, \quad 0 <  \varphi \le \frac \pi 2.
\end{align}
In the following Table \ref{2Dnumb} we compare for three cases the
value of $\varphi$ from its normal form with the numerical value
$\hat{\theta}_{1,\mathrm{num}}$
obtained by solving the
optimization problem
\begin{equation}\label{opt2solve}
\hat \theta_{1,\mathrm{num}} =
  \max_{v\in\R^2,\|v\|=1} \frac 1N\sum_{j=1}^N
  \ang(A^{j-1}v,A^j v), \quad N=1000
\end{equation}
with the \textsc{MATLAB}-routine \texttt{fminbnd}. Note that in this case
computations using forward iteration are not spoilt by a dominating direction
since $A$ has two eigenvalues of equal modulus.
\begin{table}[hbt]
\begin{center}
\begin{tabular}{c||c|c|c|c|c}
$A$ & $\mathrm{skew}(A)$ & eigenvalues & $\varphi$ & $\hat \theta_{1,\mathrm{num}}$
  & $\varphi - \hat \theta_{1,\mathrm{num}}$ \\\hline &&&& \\[-4mm]
$\left(\begin{smallmatrix} 2 & 1\\ -1 & 3 \end{smallmatrix}\right)$
& $\frac 1{\sqrt 7}<1$ & $\tfrac 52\pm \tfrac{\sqrt 3}2 i$ & $\arctan(\tfrac{\sqrt 3}5)$ &
$0.33347$ & $6\cdot 10^{-17}$\\[1mm]\hline &&&& \\[-4mm]
$\left(\begin{smallmatrix} 1 & 1\\ -1 & 1 \end{smallmatrix}\right)$
& $\frac 1{\sqrt 2}<1$ & $1\pm i$ & $\frac \pi 4$ & $0.78540$ &
$1\cdot 10^{-16}$\\[1mm]\hline\hline &&&& \\[-4mm]
$\left(\begin{smallmatrix} 2 & 1\\ -49 & 3 \end{smallmatrix}\right)$
& $\frac {5\sqrt 5}{\sqrt{11}}>1$ & $\tfrac 52\pm \tfrac{\sqrt{195}}2
                                    i$ & $\arctan(\sqrt{\tfrac{39}5})$
  &
$0.52709$ & $0.6999$
\end{tabular}
\end{center}
\caption{First angular values for autonomous examples with increasing skewness.
\label{2Dnumb}}
\end{table}

The first and second example in Table \ref{2Dnumb} have skewness $\le1$.
Then the numerical  angular value $\hat{\theta}_{1,\mathrm{num}}$
agrees with $\min(|\arg(\lambda)|,\pi-|\arg(\lambda)|)$ to machine accuracy,
as predicted by Proposition \ref{lem2.3}.
However, the third example belongs to the values
$\varphi = \arctan(\sqrt{\tfrac{39}5})$, $\rho = \frac{ 10\sqrt 5
  - \sqrt{461}}{\sqrt{39}}\approx \frac 17$ and $\mathrm{skew}(A)\approx 3.37$,
so that Proposition \ref{lem2.3} provides no explicit expression for the first
angular value. The solution
 of \eqref{opt2solve} yields
a substantially smaller value $\hat{\theta}_{1,\mathrm{num}}< \varphi$
in this case. Indeed, the first angular value exhibits a rather subtle
dependence on the matrix entries for $\mathrm{skew}(A)>1$.
Figure \ref{reso} (left panel) shows the result of an extensive
computation of the angular
value for the matrix \eqref{rhomatrix}  with $\rho = \frac 17$ and for
$25210$ equidistant points $\varphi \in [0,\frac \pi 2]$. The vertical
red line on
the left marks the critical value $\varphi_{c}=\arcsin(\frac{2}{\rho+
  \rho^{-1}})$ below which we have
$\mathrm{skew}(A(\rho,\varphi))\le 1$ and Proposition \ref{lem2.3} guarantees
$\varphi$ as the first angular value.
The value $\varphi_c$ seems to be sharp, and
for values $\varphi>\varphi_c$ we observe resonances
occurring at rational multiples of $\pi$.

The following theorem gives an explicit formula for irrational multiples
of $\pi$ and reduces  the computation of the angular value to a finite
optimization problem for rational multiples of $\pi$. For comparison
we show in Figure \ref{reso}(right panel) the diagram of angular values
when evaluated directly  from the result of  Theorem \ref{prop5.1}.
Continuing this evaluation for several values of
  $\rho$ yields the
three-dimensional diagram in Figure \ref{reso3d}.
\begin{figure}[hbt]
 \begin{center}
   \includegraphics[width=0.99\textwidth]{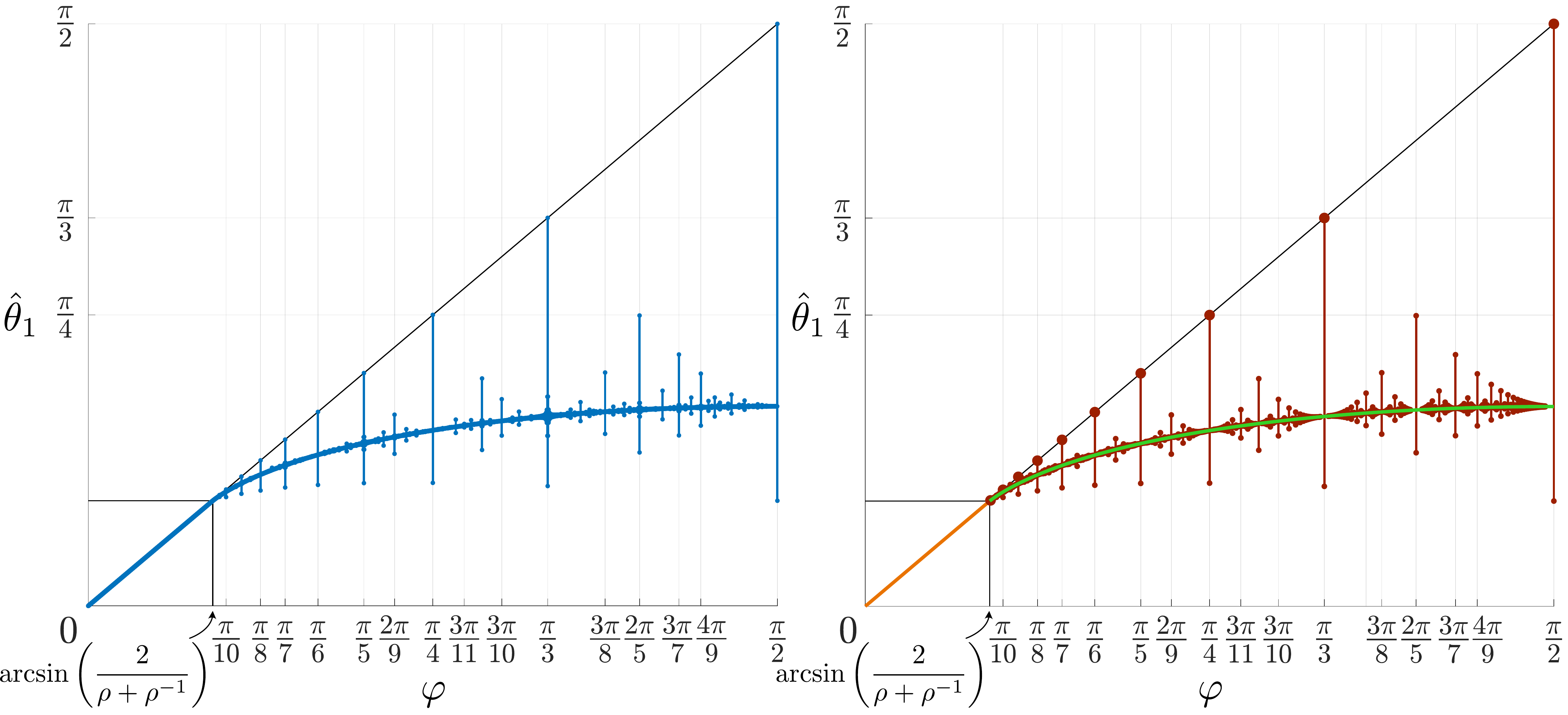}
 \end{center}
 \caption{\label{reso}Angular value $\hat \theta_1$ for \eqref{rhomatrix} with
   $\rho = \frac 17$. Left panel: For
   $25210$ equidistant points $\varphi \in [0,\frac \pi 2]$
   the minimal and maximal first angular value are computed by solving an
   optimization
   problem; minima and maxima are connected with lines. Right panel:
   Computation of first angular value via
 formula \eqref{finalform} in Theorem \ref{prop5.1}. Results for
 case 1  (orange), case 2 (green), case 3 (big points on the diagonal),
 case 4 (small points above the green curve). In case 4, minima are
 also shown (small points below the green curve)
 and connected  with corresponding maxima.}
\end{figure}

\begin{figure}[hbt]
 \begin{center}
   \includegraphics[width=0.9\textwidth]{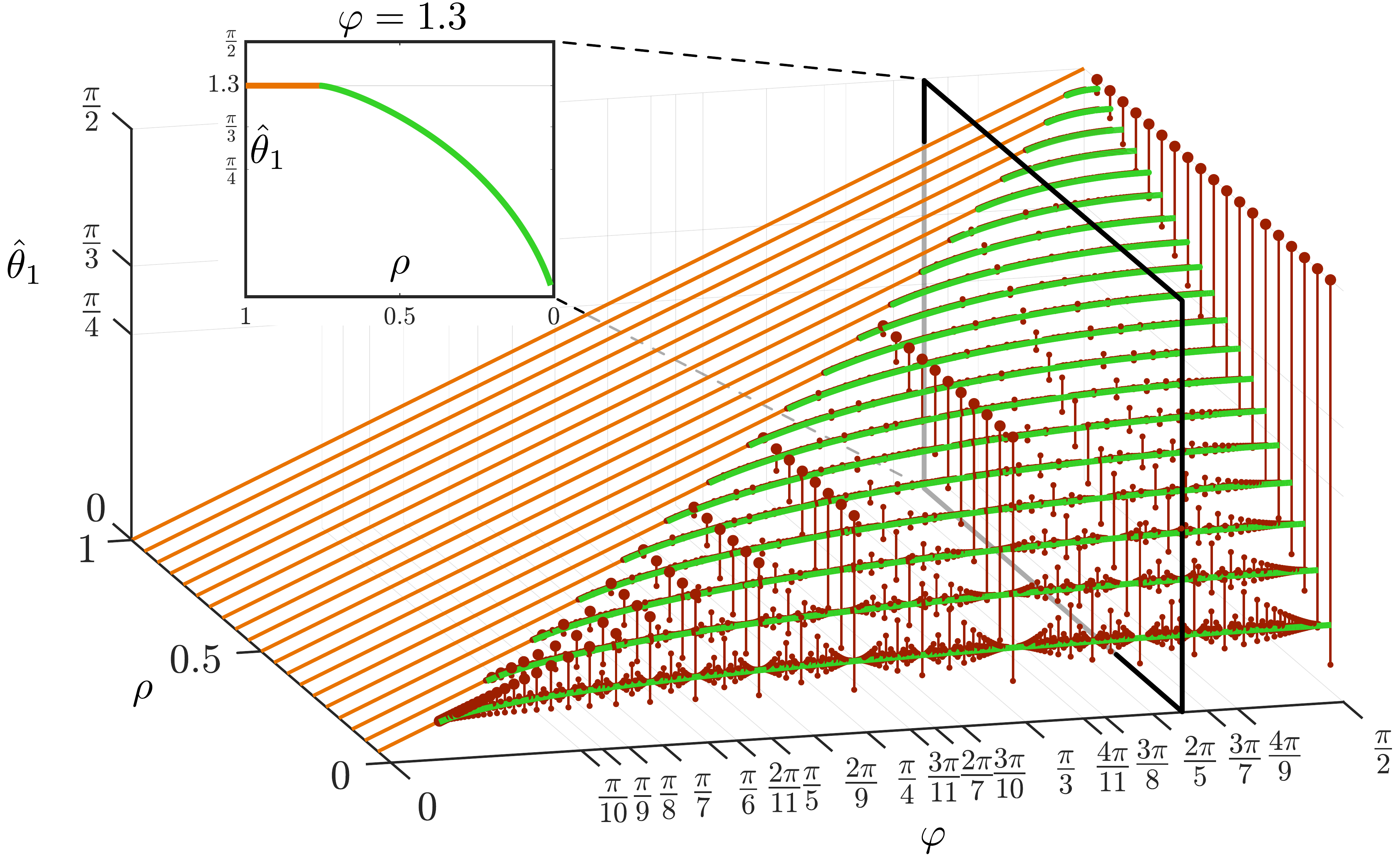}
 \end{center}
 \caption{\label{reso3d}
 Angular value $\hat \theta_1$ for \eqref{rhomatrix} with
   $\varphi\in[0,\frac \pi 2]$ and $\rho = 0.05,
   0.1,\dots,1$. Computation via formula \eqref{finalform} in Theorem
   \ref{prop5.1}. }
\end{figure}

\begin{theorem} \label{prop5.1} For $0< \rho \le 1$ and $0<\varphi \le
  \frac{\pi}{2}$ the first angular value $\hat\theta_1(A(\rho,\varphi))$
  of the  matrix from \eqref{rhomatrix} is given by
  \begin{equation} \label{finalform}
    \hat\theta_1(A(\rho,\varphi)) =
    \begin{cases} \begin{array}{lll}
      \varphi, &
      \mathrm{skew}(A(\rho,\varphi))\le 1, & \mathrm{case (i)}
         \\ \displaystyle
      \varphi+ \frac{1}{\pi} \int_{\{\delta<0\}}\delta(\theta)
      \mathrm{d}\theta, &
        \mathrm{skew}(A(\rho,\varphi))>1, \frac{\varphi}{\pi} \notin \Q,&
        \mathrm{case(ii)}  \\
      \varphi, &  \mathrm{skew}(A(\rho,\varphi))>1, \frac{\varphi}{\pi}= \frac{1}{q},
      q \ge 2,& \mathrm{case(iii)}  \\ \displaystyle
      \frac{1}{q} \max_{0 \le \theta \le \frac{\pi}{2}}\sum_{j=1}^{q}g_j(\theta),
      &  \mathrm{skew}(A(\rho,\varphi))>1, \frac{\varphi}{\pi}=\frac{p}{q}, q \notin p \N, & \mathrm{case(iv).} \end{array}
    \end{cases}
  \end{equation}
  Here the functions $g_j,\delta:[0,\pi] \to \R$, $j \in \N$ are defined as follows:
  \begin{equation*} \label{defdeltag}
  \begin{aligned}
    \delta(\theta)&= 2 \Psi_{\rho}(\theta)-2 \Psi_{\rho}(\theta+ \varphi)
    + \pi, \quad \text{with } \Psi_{\rho} \text{ from } \eqref{angleiterate},  \\
    g_j(\theta)& = \min( \theta_{j}- \theta_{j-1},
     \theta_{j-1} +\pi-\theta_{j} ), \quad
     \theta_{j-1} = \Psi_{\rho}((j-1) \varphi + \Psi_{\rho^{-1}}(\theta)).
  \end{aligned}
  \end{equation*}
  If $\frac{\pi}{2} < \varphi < \pi$ then $\hat\theta_1(A(\rho,\varphi)) =
  \hat\theta_1(A(\rho,\pi-\varphi))$.
  \end{theorem}
\begin{proof}  The proof is done sequentially for cases (i), (ii),
  (iv), and (iii).
    \begin{itemize}
    \item[(i)] This case follows from \eqref{skew} in Proposition \ref{lem2.3}
      since $A(\rho,\varphi)$ has eigenvalues $e^{\pm i \varphi}$.
  \item[(ii)]
    For nonresonant values $\frac{\varphi}{\pi}\notin \Q$ we return
    to the proof of  \eqref{Acases} case (iv) in Proposition \ref{lem2.3}. Let us
    apply Birkhoff's
  ergodic theorem to the ergodic isometry $T_{\rho,\varphi}$  of
  $(S^{2 \pi},d_{\rho},\mu_{\rho})$ (see \eqref{Tconj}) and to the continuous
  map $g$ from \eqref{chooseg},
  \begin{align} \label{thetaint1}
   \hat\theta_1(A(\rho,\varphi))&= \lim_{n\to \infty}\frac{1}{n}\sum_{j=1}^n
                           g(T_{\rho,\varphi}^{j-1}\xi) =
                           \frac{1}{2\pi} \int_{S^{2 \pi}} g(y)
                           \mathrm{d}\mu_{\rho}(y)
    = \frac{1}{2 \pi} \int_{S^{2\pi}} g(\psi_{\rho}(x)) \mathrm{d}\mu_1(x).
  \end{align}
  The last equality is due to the transformation formula. Also  note that
  the convergence is uniform in $\xi \in S^{2 \pi}$, see \cite[Ch.4.1,
  Remark 1.5]{P89}. We evaluate the integrand for $x=\theta + 2 \pi
  \Z$, $\theta \in [0,2\pi)$,
 \begin{equation*} \label{evalgpsi}
      \begin{aligned}
      g(\psi_{\rho}(x))& = \min ( d_1(\psi_{\rho}(x),T_{\rho,\varphi}\circ\psi_{\rho}(x)),
      d_1(\tau_{\pi}\circ\psi_{\rho}(x),T_{\rho,\varphi}\circ\psi_{\rho}(x))) \\ &=
      \min(d_1(\psi_{\rho}(x), \psi_{\rho}\circ \tau_{\varphi}(x)),
      d_1(\psi_{\rho}\circ\tau_{\pi}(x), \psi_{\rho}\circ \tau_{\varphi}(x)))\\
      & = \min( \Psi_{\rho}(\theta+ \varphi)- \Psi_{\rho}(\theta),
      \Psi_{\rho}(\theta+ \pi)- \Psi_{\rho}(\theta +\varphi)),
      \end{aligned}
 \end{equation*}
 where the last equality follows from
 $     \Psi_{\rho}(\theta) \le  \Psi_{\rho}(\theta + \varphi) <
 \Psi_{\rho}(\theta+\pi) = \Psi_{\rho}(\theta)+ \pi$.
Combining this with \eqref{thetaint1} and using \eqref{liftPsi} leads to
\begin{equation*} \label{thetaint2}
      \begin{aligned}
      \hat\theta_1(A(\rho,\varphi)) & = \frac{1}{2\pi} \int_0^{2 \pi}
      \min( \Psi_{\rho}(\theta+ \varphi)- \Psi_{\rho}(\theta),
      \Psi_{\rho}(\theta +\pi)-\Psi_{\rho}(\theta+ \varphi) ) \mathrm{d}\theta
      \\
      &= \frac{1}{\pi} \int_0^{\pi}
      \min( \Psi_{\rho}(\theta+ \varphi)- \Psi_{\rho}(\theta),
      \Psi_{\rho}(\theta +\pi)-\Psi_{\rho}(\theta+ \varphi) ) \mathrm{d}\theta.
      \end{aligned}
    \end{equation*}
We investigate the minimum by looking at the sign of the difference
 \begin{equation*} \label{investmin}
         \Psi_{\rho}(\theta +\pi)-\Psi_{\rho}(\theta+ \varphi)-(
         \Psi_{\rho}(\theta+ \varphi)- \Psi_{\rho}(\theta))
            = \delta(\theta).
 \end{equation*}
   For $\mathrm{skew}(A(\rho,\varphi))>1$ the equivalence of \eqref{lesspi/2} and
  \eqref{equivangle} yields
  \begin{equation} \label{deltapos}
    \delta(\theta)
   \begin{cases}
   \ \ge 0, &  \theta \in [0,\theta_-] \cup [\theta_+,\pi],\\
   \ < 0, & \theta \in (\theta_-,\theta_+),
\end{cases}
   \end{equation}
  where the values $\theta_{\pm}$ are given as follows
  \begin{equation*} \label{defthetapm}
    \begin{aligned}
       \theta_{\pm} & = \Psi_{\rho^{-1}}(\theta'_{\pm}), \\
        \theta'_{-}& = \frac{1}{2}\arcsin\left(
          \frac{2}{\tan(\varphi)(\rho^{-1}-\rho)}\right) \in
        \big(0,\frac{\pi}{4}\big),\; \theta'_+=\frac{\pi}{2}-\theta'_- \in
        \big(\frac{\pi}{4},\frac{\pi}{2}\big).
      \end{aligned}
  \end{equation*}
   Using \eqref{liftPsi} the following computation completes the proof
  of assertion (ii)
  \begin{equation*}
   \begin{aligned}
  \hat\theta_1(A(\rho,\varphi)) & = \frac{1}{\pi} \Big\{
  \big(\int_0^{\theta_-}+ \int_{\theta_+}^{\pi}\big)
  \Psi_{\rho}(\theta+\varphi)- \Psi_{\rho}(\theta) \mathrm{d}\theta
  + \int_{\theta_-}^{\theta_+} \Psi_{\rho}(\theta+\pi) -
  \Psi_{\rho}(\theta+ \varphi)\mathrm{d}\theta \Big\}\\
  & = \frac{1}{\pi} \Big\{\int_0^{\pi}
  \Psi_{\rho}(\theta+\varphi)- \Psi_{\rho}(\theta) \mathrm{d}\theta
  +  \int_{\theta_-}^{\theta_+} \delta(\theta) \mathrm{d}\theta \Big\} \\
  & = \frac{1}{\pi}\left\{\int_{\varphi}^{\pi+\varphi} - \int_0^{\pi} \right\}
   \Psi_{\rho}(\theta) \mathrm{d}\theta+  \frac{1}{\pi} \int_{\theta_-}^{\theta_+}
  \delta(\theta) \mathrm{d} \theta\\
   & = \frac{1}{\pi} \int_0^{\varphi} \Psi_{\rho}(\theta) + \pi -
   \Psi_{\rho}(\theta) \mathrm{d}\theta + \frac{1}{\pi} \int_{\theta_-}^{\theta_+}
   \delta(\theta) \mathrm{d} \theta\\
   &= \varphi +\frac{1}{\pi}
  \int_{\{\delta <0\}} \delta(\theta) \mathrm{d} \theta.
    \end{aligned}
  \end{equation*}
 Recall that the values $\theta_{\pm}=\theta_{\pm}(\rho,\varphi)$,
  depend on $\rho$, $\varphi$ and
  satisfy
  \begin{equation} \label{support}
    0 < \theta_-(\rho,\varphi) < \theta_+(\rho,\varphi) < \tfrac{\pi}{2},
    \quad \text{if} \; \mathrm{skew}(A(\rho,\varphi))>1
  \end{equation}
by the strict monotonicity of $\Psi_{\rho^{-1}}$.
  Therefore, \eqref{deltapos} implies $\hat\theta_1(A(\rho,\varphi)) < \varphi$.
  \item[(iv)]
  Next we consider $\frac{\varphi}{\pi}= \frac{p}{q}$ for some natural numbers
  $0 < p \le q$. From the definition \eqref{Tconj} of $T_{\rho,\varphi}$ we obtain
  \begin{align*}
    T_{\rho,\varphi}^q = \psi_{\rho} \circ \tau_{\varphi}^q \circ \psi_{\rho}^{-1}=
    \psi_{\rho} \circ \tau_{p \pi} \circ \psi_{\rho}^{-1}= \tau_{p \pi},
  \end{align*}
  where we used $ \psi_{\rho}\circ \tau_{p \pi} = \tau_{p \pi} \circ \psi_{\rho}$
  due to \eqref{liftPsi}.
  Moreover, translation invariance of the metric $d_1$ yields that the
  function $g$ in \eqref{chooseg} is $\pi$-periodic:
  \begin{align*}
     g(\tau_{\pi}x)&
     = \min(d_1(\tau_{\pi}x,T_{\rho,\varphi}(\tau_{\pi}x)),
            d_1(\tau_{2\pi}x,T_{\rho,\varphi}(\tau_{\pi}x)))\\
    & = \min(d_1(\tau_{\pi}x,\tau_{\pi}(T_{\rho,\varphi}x)),d_1(x,\tau_{\pi}(T_{\rho,\varphi}x)))\\
     & = \min(d_1(x,T_{\rho,\varphi}x),d_1(\tau_{\pi}x,\tau_{2\pi}(T_{\rho,\varphi}x)))\\
    & =\min(d_1(x,T_{\rho,\varphi}x),d_1(\tau_{\pi}x,T_{\rho,\varphi}x))= g(x).
  \end{align*}
  Therefore, decomposing $n=kq+r$ with $k \ge 0, 1 \le r \le q$ leads to
  \begin{align*}
    \frac{1}{n} \sum_{j=1}^n g(T_{\rho,\varphi}^{j-1}x) & = \frac{1}{n}\Big(
    \sum_{\nu=0}^{k-1} \sum_{\ell=1}^q g(T_{\rho,\varphi}^{\nu q + \ell -1}x)
    + \sum_{\ell=1}^r g(T_{\rho,\varphi}^{kq+\ell-1}x)\Big) \\
    &= \frac{1}{n} \Big(
    \sum_{\nu=0}^{k-1} \sum_{\ell=1}^q g(\tau_{\nu p \pi}(T_{\rho,\varphi}^{\ell -1}x ) )
    + \sum_{\ell=1}^r g(\tau_{kp \pi}(T_{\rho,\varphi}^{\ell-1}x))\Big)\\
    & = \frac{k}{kq+r} \sum_{\ell=1}^q g(T_{\rho,\varphi}^{\ell-1}x) + \frac{1}{n}
    \sum_{\ell=1}^r g(T_{\rho,\varphi}^{\ell-1}x)
    \stackrel{n\to\infty}{\xrightarrow{\hspace*{1cm}}} \frac{1}{q} \sum_{\ell=1}^{q}
    g(T_{\rho,\varphi}^{\ell-1}x).
    \end{align*}
  Maximizing over $x = \theta_0+ 2 \pi \Z$ and using \eqref{bjexpress} then
  proves case (iv) of formula \eqref{finalform}.
  \item[(iii)]
  It remains to show that the maximum in case $p=1$ is given by $\varphi=\frac{\pi}{q}$.
  In this case we have $T_{\rho,\varphi}^q=\tau_{\pi}$ and thus equation \eqref{bjexpress}
  yields  for all $\theta_0 \in  [0,\frac{\pi}{2}]$
   \begin{align} \label{estsummin}
  \frac{1}{q} \sum_{\ell=1}^{q}
  g(T_{\rho,\varphi}^{\ell-1}x)  = \frac{1}{q}
     \sum_{\ell=1}^{q}\chi(\theta_j - \theta_{j-1}) \le \frac{1}{q}
     \sum_{j=1}^q (\theta_j -\theta_{j-1})=
  \frac{1}{q}(\theta_q - \theta_0)= \frac{\pi}{q} = \varphi.
  \end{align}
  We set $\theta_0= \theta_-$ and recall that
  $\theta_-$ has been chosen such that $\theta_1- \theta_0=\frac{\pi}{2}$.
  Since $\theta_j - \theta_{j-1}\ge 0$ sum up to $\pi$ there is no
  index $j>1$ with $\theta_j - \theta_{j-1}> \frac{\pi}{2}$, hence equality
  holds in \eqref{estsummin}.
  \end{itemize}

\end{proof}
Theorem \ref{prop5.1} and Figures \ref{reso}, \ref{reso3d} show
that angular values
can be quite sensitive to parametric perturbations. For example,
approximate a rational multiple $\varphi_0=\frac{\pi}{q},q \ge 2$ by irrational
multiples $\varphi$ of $\pi$ for some value $\rho$ with
$\mathrm{skew}(A(\rho,\varphi_0)) > 1$.
Then the formula \eqref{finalform} and the relations \eqref{deltapos},
\eqref{support} imply
\begin{equation} \label{lowersemi}
\liminf_{\varphi\to\varphi_0} \hat\theta_1(A(\rho,\varphi))
= \varphi_0+ \int_{\theta_-(\rho,\varphi_0)}^{\theta_+(\rho,\varphi_0)} \delta(\theta)
d \theta < \varphi_0=
\hat\theta_1(A(\rho,\varphi_0)).
\end{equation}
Hence the angular value $\hat \theta_1$ is not lower semi-continuous.
However, angular values may still be upper semi-continuous.
\subsection{An algorithm for computing first angular values}\label{algorithm}
Based on Theorem \ref{prop5.1} and on the results from Section \ref{sec5},
we propose the following numerical scheme for autonomous systems; see
Algorithm \ref{algo}.
In case $A\in\R^{d,d}$ is invertible and satisfies the assumption
\eqref{eqeigencond}
our numerical approach is justified by Theorem \ref{cor2.4}.

\begin{algorithm}[hbt]
\caption{Computation of $\hat\theta_1(A)$ \label{algo}}
\begin{itemize}
\item [(1)] Compute a real Schur decomposition of $A$
$$
A = QSQ^\top,\quad S=
\begin{pmatrix}
  \Lambda_1 & & \star\\
  & \ddots &\\
  0&& \Lambda_k
  \end{pmatrix},\quad Q\in\R^{d,d} \text{ orthogonal, cf.\ \eqref{eq2.8},}
$$
 such that the diagonal blocks
  $\Lambda_1,\dots,\Lambda_\ell$ are two-dimen\-sional and
  $\Lambda_{\ell+1},\dots \Lambda_k$ are reals (such a Schur decomposition always
  exists, see \cite[Theorem 2.3.4]{HJ2013}). Let $\lambda_i,
  \bar{\lambda}_i$ be the eigenvalues of
  $\Lambda_i$, $i=1,\ldots \ell$ and let $A_i=\lambda_i=\Lambda_i$ for
  $i=\ell+1,\dots,k$.
\item [(2)] Compute $\hat \theta_1(A)$ as follows
\begin{algorithmic}

\If{$ \exists i\neq j \in \{\ell+1,\dots,k\}: \lambda_i = -\lambda_j$}
\State  $\hat \theta_1(A) = \frac \pi 2$
\Else
\For {$i=1,\dots,\ell$}
\If {$i=1$}
\State $A_1 = \Lambda_1$
\Else
\State Compute a reordered  Schur decomposition of $A$ using
\texttt{ordschur},
\State {such that the upper left $2 \times 2$-block has the
eigenvalue $\lambda_i$.}
\State {Denote this upper left $2\times 2$-block by $A_i$.}
\EndIf
\State Determine $\varphi_i$, $\rho_i$ such that $A_i = |\lambda_i|
   A(\rho_i,\varphi_i)$.
\State Compute $\hat \theta_1(A_i) = \theta_1(A(\rho_i,\varphi_i))$
using Theorem \ref{prop5.1}.
\EndFor
\State $\hat \theta_1(A) = \max\{0, \hat \theta_1(A_i), i =1,\dots,\ell\}$.
\EndIf
\end{algorithmic}
\end{itemize}
\end{algorithm}
As explained in Remark \ref{remopposite}, the Schur decomposition
of $A$ is reordered several times to obtain the diagonal blocks
$A_i$, $i=2,\ldots,\ell$. We apply the \textsc{MATLAB}
command \texttt{ordschur} for this task.\footnote{We are not aware of
  a \textsc{MATLAB} procedure that computes the block decomposition
  \eqref{specialblock} directly.}
The value
$\theta_1(A(\rho_i,\varphi_i))$ is calculated for $i=1,\dots,\ell$ with
Theorem \ref{prop5.1}. Note that the fourth case in \eqref{finalform}
results in a one-dimensional optimization problem which we solve
with a derivative-free method implemented
in the \textsc{MATLAB}-routine \texttt{fminbnd}.

\subsection{Numerical experiments}
\label{sec5.2}
Let us apply Algorithm
\ref{algo} to autonomous models with increasing complexity and
dimension. The example in Section \ref{4Dillu}
particularly illustrates the need for reordering the Schur
decomposition in Algorithm \ref{algo}.

\subsubsection{Block-diagonal examples}
\label{sec5.2.1}
We begin with autonomous examples which have a block-diagonal structure.
Due to the invariance of corresponding coordinate spaces, one can read
off first angular values without the need for reordering Schur
decompositions.
Furthermore, the
numerical calculation can be done with high
accuracy and even exactly when these expressions are evaluated
symbolically.
Therefore, approximation errors are not
discussed in Table \ref{skewsmall}.

\begin{table}[hbt]
\begin{center}
\begin{tabular}{c||c|c||c|c||c}
$A$ & $A_1$ & $\hat\theta_1(A_1)$ & $A_2$ &
$\hat\theta_1(A_2)$ & $\hat\theta_1(A)$\\[1mm]\hline &&&&& \\[-4mm]
 $\left(\begin{smallmatrix} 2 & 0\\ 0 & 3 \end{smallmatrix}\right)$
& $2$ & $0$ & $3$ & $0$ & $0$ \\[1mm]\hline &&&&& \\[-4mm]
$\left(\begin{smallmatrix} 2 & 0\\ 0 & -2 \end{smallmatrix}\right)$
& $2$
& $0$ & $-2$ & $0$ &$\frac \pi 2$ \\[1mm]\hline &&&&& \\[-4mm]
$\left(\begin{smallmatrix} c & s\\ s & -c \end{smallmatrix}\right)$
& $1$ & $0$ & $-1$ & $0$ & $\frac \pi 2$ \\[1mm]\hline &&&&& \\[-4mm]
$\left(\begin{smallmatrix} c & -s & 0\\ s & c & 0\\ 0 & 0 & -2 \end{smallmatrix}\right)$
& $\left(\begin{smallmatrix} c & -s\\ s & c \end{smallmatrix}\right)$
& $\varphi$ & $-2$ & $0$ &$\varphi$
\end{tabular}
\end{center}
\caption{Angular values of block-diagonal examples.
  We abbreviate $c= \cos(\varphi)$, $s=\sin(\varphi)$, $0<\varphi <
  \frac \pi 2$. \label{skewsmall}}
 \end{table}

\subsubsection{An illustrative four-dimensional example}\label{4Dillu}
Using the normal form \eqref{rhomatrix} we consider a $4 \times 4$-matrix which
has already Schur form
\begin{equation*}\label{schur1}
A =
\begin{pmatrix}
A(1,\frac 12) & I_2\\ 0 & \eta A(\frac 12, 1.4)
\end{pmatrix} =
\begin{pmatrix} \Lambda_1 & I_2 \\ 0 & \Lambda_2 \end{pmatrix}
\end{equation*}
with $\eta = 1.2$.
For this matrix we have $\theta_1(\Lambda_1)= \frac{1}{2}$ and
$\theta_1(\Lambda_2) = 1.128$.
The algorithm sets $A_1 = \Lambda_1$  and reorders the
Schur form so that the eigenvalues of $\Lambda_2$ appear in the first
$2 \times 2$-block:
$$
Q^T A Q =
\begin{pmatrix}
\eta A(0.7493, 1.4) & \star\\
0 & A(0.6142, \frac 12)
\end{pmatrix},\quad \text{whereby } A_2 = \eta A(0.7493, 1.4).
$$
From Theorem \ref{prop5.1} the algorithm then finds $\theta_1(A_2) = 1.355$
and thus we have
$$
\max(\theta_1(\Lambda_1),\theta_1(\Lambda_2))= 1.128  < 1.355=
\max(\theta_1 (A_1), \theta_1(A_2)) = \theta_1(A).
$$
This example illustrates that first angular values can
generally not be computed from the diagonal blocks of a single Schur
decomposition.

\subsubsection{High dimensional examples}
We illustrate the performance of our algorithm for three
matrices of dimension $10^2$, $10^3$ and $10^4$.
Their entries are uniformly distributed in $(0,1)$
and generated by the \textsc{MATLAB} random number generator
 initialized with \texttt{rng(1)}.
 Table \ref{trand} documents our numerical results. We measure the time for the
 initial Schur
decomposition and the maximal time for one reordering with
\texttt{ordschur}.
It turns out that the computing time for one reordering step grows
linearly with the position $i$ of the block $\Lambda_i$ in the Schur form.
The numerical experiments are carried out on an
Intel Xeon W-2140B CPU with \textsc{MATLAB 2020a}.

\begin{table}
\begin{tabular}{c||c||c||c|c}
$\dim(A)$ & $\hat \theta_1(A)$ & number of $2\times 2$ blocks &
initial Schur & max reordering\\\hline
$10^2$ & $1.5370$ & $45$ & $0.0045$ sec & $0.0001$ sec\\\hline
$10^3$ & $1.5643$ & $488$ & $0.43$ sec & $0.013$ sec\\\hline
$10^4$ & $1.5705$ & $4958$ & $105$ sec & $1.18$ sec
\end{tabular}
\caption{First angular values of three random matrices:
number of $2\times 2$ blocks, analyzed by Algorithm \ref{algo};
computing time for initial Schur
decomposition; maximal time for one reordering Schur step. \label{trand}}
\end{table}

Applying Algorithm \ref{algo} to the $10^4$-dimensional random matrix
yields $\ell = 4958$ two-dimensional blocks for which we calculate
the first angular value, using Theorem \ref{prop5.1}.
Then there are $84$ real eigenvalues of different modulus
leading to a vanishing angular value. Summing up we obtain
$k=5042$ one- resp.\ two-dimensional blocks $A_i$.
For the presentation in Figure \ref{hist}, these blocks are rearranged, such that
\begin{equation}\label{anordnung}
\theta_1(A_i) \le \theta_1(A_{i+1}) \text{ for all } i = 1,\dots,k-1.
\end{equation}
The left panel shows a plot of the pairs $(i,\theta_1(A_i))_{i=1,\dots,k}$. Except
for an  initial ramp due to the $84$ real eigenvalues, the plot
suggests an almost uniform distribution of angular values. This is
also confirmed by
the corresponding  histogram shown in the right panel.
As expected, further experiments show no correlation between the modulus
$|\lambda_i|$  of the eigenvalue and the angular value $\theta_1(A_i)$
of the corresponding $2 \times 2$ matrix $A_i$.

\begin{figure}[H]
 \begin{center}
   \includegraphics[width=0.95\textwidth]{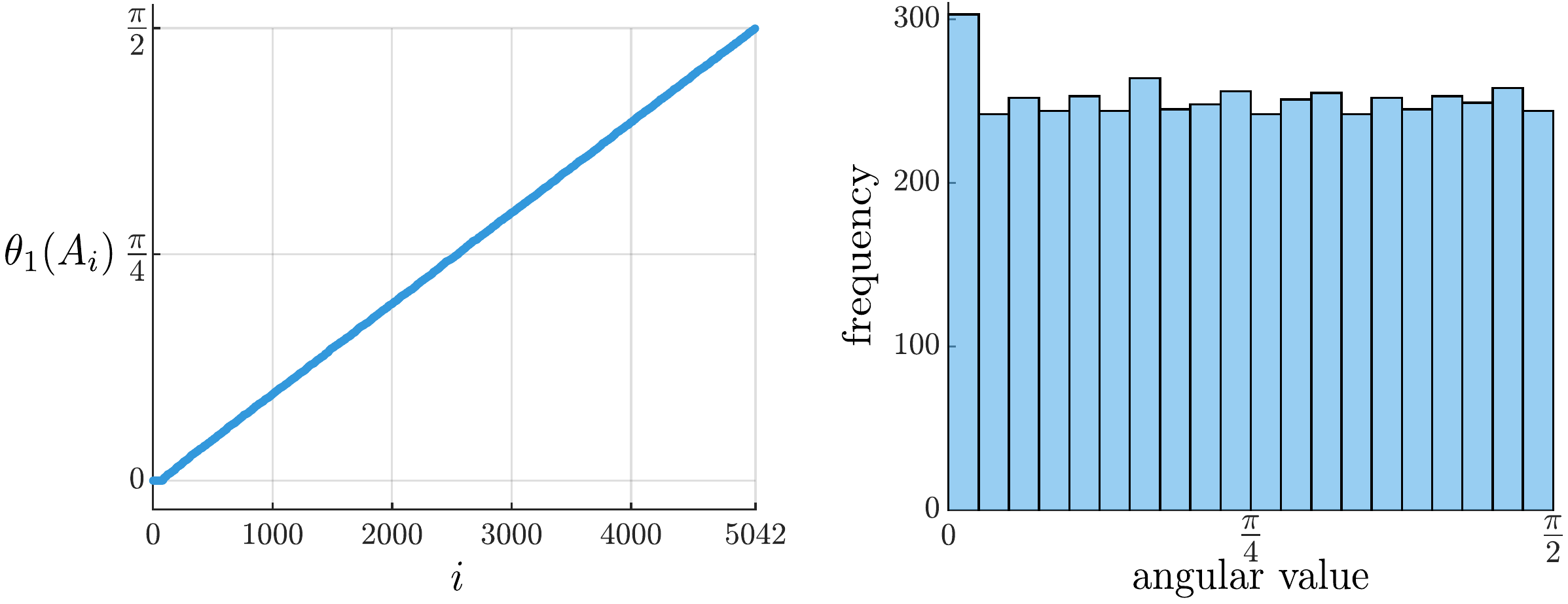}
 \end{center}
 \caption{\label{hist} Left: sorted angular values $\theta_1(A_i)$, see
   \eqref{anordnung}, of a $10^4$-dimensional random matrix; right: histogram
   of angular values.}
\end{figure}

\section*{Outlook}

The approach of this article lends itself to several extensions and
further problems which we discuss in  the following.
\subsection*{Numerics for nonlinear systems}
The content of Sections \ref{sec2} and \ref{sec3} applies to linear difference equations arising from variational equations of nonlinear dynamical systems of the form
 \begin{equation}\label{var2}
u_{n+1} = DF(\xi_n)u_n,\quad n\in\N_0,
 \end{equation}
 where $\xi_{n+1}=F(\xi_n),n\in \N_0$ is a bounded trajectory of a
 nonlinear diffeomorphism $F:\R^d \to \R^d$.
To numerically estimate the angular values of these variational equations, one has to extend the numerical algorithm for the autonomous case (Section \ref{sec6}) to general nonautonomous systems \eqref{diffeq} in dimension $d \ge 2$ and to angular values of arbitrary type $s\ge1$.
This is the topic of the forthcoming work \cite{BeHu20X}, which uses a reduction procedure to generalize
Theorem \ref{cor2.4} from eigenvalues and eigenspaces to the dichotomy spectrum and stable and unstable fibers (see \cite{ss78}, \cite{ak01}, \cite{p16}).
The algorithm is applied to variational equations.
 In fact, Figure \ref{hen0} shows
 for the well-known H\'enon map (\cite{h76}) the succession of those subspaces
 which lead to the outer angular value for the linearized equation \eqref{var2}.

\subsection*{Continuous-time systems}
   It is natural to set up a theory of angular values for continuous-time systems.  Such an extension requires one to
   handle derivatives of principal angles between moving subspaces
      both theoretically and numerically.  By Proposition \ref{prop1}
   principal angles can be computed from singular
   values of matrices, which employ orthogonal bases of
   subspaces---obtained by a QR-decomposition, for example.
   Thus one is led to the well-known problem
   of computing smooth singular value and QR-decompositions
   which has been studied extensively in the literature; see \cite{bbmn91},
   \cite{de99}, \cite{de06}. One approach is to solve suitable
   differential equations for smooth decompositions (\cite{de99}), and this
   has turned out to be efficient with numerical methods for
   Lyapunov exponents; see \cite[Section 4]{de06}.
      A corresponding analysis of the angle function from Section \ref{sec2}
   and a resulting algorithm are currently under investigation.

\subsection*{Perturbation theory}
      As noted after Theorem \ref{prop5.1} (see \eqref{lowersemi} and
   Figures \ref{reso}, \ref{reso3d})  angular values can be quite
   sensitive to parametric
   perturbations. In particular, without further assumptions they are
   not lower semi-continuous. It is an open question whether they are
   still upper
   semi-continuous in general. More specifically, it will be desirable to
   have criteria
   which ensure continuity of angular values. For
    Lyapunov exponents in continuous time such criteria are well known; see \cite[Ch.IV,V]{Ad1995},
   \cite[Section 2]{de06}. 

\subsection*{Regularity theory}
   On the one hand, the examples in Section \ref{sec3.2}
   demonstrate that $\liminf$ and $\limsup$ generally do not coincide
   for outer angular values. 
   On the other hand, the $\liminf$ and $\limsup$ do coincide for the inner
   angular values and the uniform angular values in random dynamical systems
   and in all cases for autonomous dynamical systems; see Sections \ref{sec3} and \ref{sec5}.
   It will be of interest to
   identify a larger class of systems for
   which the corresponding limits exist. 
   This will provide a weak analogy to the class of regular continuous-time
   dynamical systems that have sharp Lyapunov exponents;
   see \cite[Theorem 3.9.1]{Ad1995}.

\section*{Acknowledgments}
WJB thanks
the CRC 701 `Spectral Structures and Topological Methods in Mathematics' at
Bielefeld University  and the School of
Mathematics and Statistics at the University of New South Wales
 for supporting his visit to GF in 2017, where this research
 was initiated. Further support by the CRC 1283 `Taming uncertainty
 and profiting from randomness and low regularity in analysis,
 stochastics and their
 applications' is gratefully acknowledged.
 GF is partially supported by an ARC Discovery Project.
TH thanks the Research Centre for Mathematical
Modelling ($\text{RCM}^2$) and the Faculty of Mathematics at Bielefeld
University for supporting this research project.
 The authors thank the referees
for pointing them to various related topics and references from the
literature.


\bibliographystyle{abbrv}

\newpage
\setcounter{section}{0}
\renewcommand{\thesection}{\Roman{section}}
\section*{Supplementary materials}
\label{app}
\section{Variational characterization of maximum principal angle --
Proof of\\Proposition \ref{Lemma2}}
\label{app1}
\begin{proof}
  We use the following elementary fact
  \begin{equation} \label{eq1.1}
    \max_{ y \in \R^j, \|y\|=1} v^{\top}y = \|v\| \quad \forall \; v \in \R^j, v \neq 0,
  \end{equation}
  with the maximum achieved at $y = \frac{1}{\|v\|} v$ for $v \neq 0$.
  Consider $w \in W$ with $\|w\|=1$ and
  $w^{\top}w_{\ell}=0$, $\ell=j+1,\ldots,s$.  By Proposition
  \ref{prop1} there exists $b \in \R^s$ such that
  \begin{align} \label{wrepresent}
    w = Qb = QZZ^{\top}b= \begin{pmatrix}  w_1 & \cdots & w_{s} \end{pmatrix}
    Z^{\top} b.
  \end{align}
  Since $\|w\|=1$ and $w^{\top}w_\ell=0$ for $\ell=j+1,\ldots,s$ we obtain
  the  partitioning
    \begin{align*}
      Z^{\top}b = \begin{pmatrix} b^{I} \\ 0 \end{pmatrix}, b^{I} \in
      \R^j, \|b^{I}\|=1, \quad
      Z = \begin{pmatrix} Z^{I} & Z^{II} \end{pmatrix}, Z^{I} \in \R^{d,j}.
    \end{align*}
    By \eqref{eq1.1} this implies for all $v \in \R^d$, $ v\neq 0$
    \begin{align} \label{eq1.2}
      \max_{\substack{w\in W, \|w\|=1 \\ w^{\top} w_{\ell} =0,
          \ell=j+1,\ldots,s}}  v^{\top} w= \max_{b^{I} \in
      \R^j,\|b^{I}\|=1}v^{\top}Q Z^{I} b^{I} = \|Z^{I \top}Q^{\top}v
      \|.
      \end{align}
    In  a similar way, for $v \in \R^d$ with $\|v\|=1$ and $v^{\top}v_{\ell}=$ for
    $\ell=j+1,\ldots,d$ we find vectors $a \in \R^d$, $a^{I} \in \R^j$ such
    that
    \begin{align*}
      v= Pa = P Y Y^{\top}a =\begin{pmatrix}  v_1 & \cdots & v_{s} \end{pmatrix}
      Y^{\top}a, \quad Y^{\top}a =\begin{pmatrix} a^{I} \\ 0 \end{pmatrix},\quad
      \|a^I\|=1.
    \end{align*}
    Using this and \eqref{eq1.b} in \eqref{eq1.2} and setting
    $\Sigma^I = \mathrm{diag}(\sigma_1, \ldots, \sigma_j)$ leads to
    \begin{align*}
     & \min_{\substack{v\in V, \|v\|=1 \\  v^{\top}v_{\ell} =0,
      \ell=j+1,\ldots,s}}\
       \max_{\substack{w\in W, \|w\|=1 \\ w^{\top} w_{\ell} =0,
          \ell=j+1,\ldots,s}}  v^{\top} w\\
       = &
       \min_{a^I \in \R^j , \|a^I\|=1} \|Z^{I \top} Q^{\top}PY
       \begin{pmatrix} a^I \\ 0 \end{pmatrix}\|
       = \min_{a^I \in \R^j , \|a^I\|=1}  \|Z^{I \top} Z \Sigma \begin{pmatrix}
         a^{I} \\ 0 \end{pmatrix}\| \\
       = &\min_{a^I \in \R^j , \|a^I\|=1} \|Z^{I \top}
       Z^I \Sigma^I a^I \| = \min_{a^I \in \R^j , \|a^I\|=1}\| \Sigma^I a^I \|.
    \end{align*}
    Since $\sigma_1 \ge \ldots \ge \sigma_j$ the last minimum is
    $\sigma_j$ and it is achieved at the $j$-th unit vector
    $a^I=e^I_j\in \R^j$. With Proposition \ref{prop1}
    this yields the minimizer $v= PYe_j =v_j$, where
    $e_j = \begin{pmatrix} e_j^I \\ 0 \end{pmatrix} \in \R^d$. Returning to
    \eqref{eq1.2} we obtain the maximizer
    $b^I=\frac{1}{\sigma_j} Z^{I \top}Q^{\top} v_j$  where $\sigma_j =
    \|Z^{I \top}Q^{\top} v_j\|$ is the maximum value. By \eqref{wrepresent}
    and \eqref{eq1.d} this
    leads to the maximizer of the original problem
\begin{align*}
  w = &\, \frac{1}{\sigma_j}QZ \begin{pmatrix} Z^{I \top} Q^{\top} v_j
    \\0 \end{pmatrix}= \frac{1}{\sigma_j}QZ \begin{pmatrix} Z^{I \top}
    Q^{\top} PY e_j \\0
  \end{pmatrix} \\
  = & \, \frac{1}{\sigma_j}QZ \begin{pmatrix} Z^{I \top} Z \Sigma e_j \\0
  \end{pmatrix} = \frac{1}{\sigma_j}QZ \begin{pmatrix}  \sigma_j e_j^I \\
    0 \end{pmatrix} = w_j.
    \end{align*}
Finally, note that taking $\arccos$ reverses $\min$ and $\max$ in \eqref{A1}.
\end{proof}

\section{Uniform almost periodicity -- Proof of Lemma \ref{estap}}\label{app2}
\begin{proof}
Let $\eps >0$. By the  uniform almost periodicity there exists a
$P\in\N$ such that for every $V\in\cV$ and
each $k\in\N_0$ we find a $p_k\in\{k,\dots,P+k\}$ (which may depend on $V$) with
\begin{equation}\label{apbasic}
\|b_n(V)-b_{n+p_k}(V)\| \le \frac \eps 8 \quad \forall n\in\N.
\end{equation}
Let $b_{\infty}=\sup_{n,V}\|b_n(V)\|$ and $L = \lceil \frac {16}\eps P
b_\infty\rceil$. It follows for each $k\in\N_0$ that
\begin{equation}\label{finally}
\begin{aligned}
\Big\| \sum_{j=1}^L b_j(V) - \sum_{j=1}^L b_{j+k}(V)\Big\|
& \le \sum_{j=1}^L \|b_j(V)-b_{j+p_k}(V)\| + \Big\|\sum_{j=1}^L
      b_{j+p_k}(V)-\sum_{j=1}^L b_{j+k}(V)\Big\| \\
&\le \sum_{j=1}^L \frac \eps 8 + 2 P b_\infty \le  L \frac \eps 4.
\end{aligned}
\end{equation}

Let $N= \lceil \frac 8\eps L b_\infty\rceil$ and decompose $n\ge m\ge N$
modulo $L$, i.e.\
\begin{equation*}
m = \ell_m L + r_m,\ 0\le r_m<L,\quad  n = \ell_n L + r_n,\ 0\le r_n<L.
\end{equation*}
For $c(V):= \sum_{j=1}^L b_j(V)$ we obtain from \eqref{apbasic} and
\eqref{finally} for each $k\in \N_0$
the estimates
\begin{align*}
\Big\|\sum_{j=1}^{L \ell_n} b_j(V) - \ell_n c(V)\Big\|
 &  \le \sum_{i=1}^{\ell_n} \Big\| \sum_{j=1}^L b_{j+(i-1)L}(V) -
     c(V)\Big\|
\le  \ell_n L \frac \eps 4,\\
\Big\| \frac{\ell_n}{n} c(V) - \frac{\ell_m}{m} c(V)\Big\|
&= \Big|\frac{\ell_n r_m - \ell_m r_n}{nm}\Big| \|c(V)\|\le
\frac {\ell_n L}{nm} \|c(V)\|
\le \frac 1m \|c(V)\| \le \frac{L}{N} b_{\infty}  \le \frac \eps 4.
\end{align*}
 Combining these estimates, we find for every $k\in\N_0$
\begin{align*}
&\Big\|\frac 1n \sum_{j=1}^nb_j(V) - \frac 1m \sum_{j=1}^m
                 b_{j+k}(V)\Big\|\\
& \le \Big\|\frac 1n \sum_{j=1}^{L\ell_n} b_j(V) - \frac 1m
   \sum_{j=1}^{L\ell_m}b_{j+k}(V)\Big\|
+ \Big\|\frac 1n \sum_{j=L\ell_n+1}^{L\ell_n+r_n} b_j(V) - \frac 1m
   \sum_{j=L\ell_m+1}^{L\ell_m+r_m}b_{j+k}(V)\Big\|\\
&\le \Big\|\frac 1n \sum_{j=1}^{L\ell_n} b_j(V) - \frac{\ell_n}{n}
                                                       c(V)\Big\|
+ \Big\|\frac{\ell_n}{n} c(V) - \frac{\ell_m}{m} c(V)\Big\|\\
&+ \Big\| \frac{\ell_m}{m} c(V) -  \frac 1m
   \sum_{j=1}^{L\ell_m}b_{j+k}(V)\Big\|
+ \Big\|\frac 1n \sum_{j=L\ell_n+1}^{L\ell_n+r_n} b_j(V) - \frac 1m
   \sum_{j=L\ell_m+1}^{L\ell_m+r_m}b_{j+k}(V)\Big\|\\
&\le \frac{\ell_n L}{n} \frac \eps 4 + \frac \eps 4 + \frac{\ell_m L}{m}
\frac \eps 4 + \frac{2 L}{N} b_\infty \le \eps.
\end{align*}
\end{proof}

\section{A matrix with a single eigenvalue -- Proof of Proposition
  \ref{prop2.5}}\label{app3}
\begin{proof} By Lemma \ref{lem2.2} and Corollary \ref{propauto} it suffices
  to show that $\bar{\theta}_1(A)=0$. Further, by Proposition \ref{prop2}
  we can assume $\lambda=1$ and $A$ to be in (real) Jordan normal form
  \begin{equation} \label{jordanform}
  \begin{aligned}
    A& =\mathrm{diag}(\Lambda_1,\ldots,\Lambda_k), \quad \Lambda_{\ell}=
    I_{d_{\ell}}+ E_{\ell}\in \R^{d_{\ell},d_{\ell}}, \\
    (E_{\ell})_{ij}& = \delta_{i+1,j},
    1\le i,j \le d_{\ell}, \ell=1,\ldots,k.
  \end{aligned}
  \end{equation}
  Consider first the case $k=1$ and drop the index $\ell$. For a
  vector $v \in \R^d, v \neq 0$ let $d_{\star}+1 = \max\{j\in
  \{1,\ldots,d\}: v_j \neq 0\}$ and
  assume w.l.o.g.\ $v_{d_{\star}+1}=1$. Further, we  define vectors $v^j, j \in \N_0$ and
  polynomials $q_i$ of degree $d_{\star}+1-i$ for $i=1,\ldots,d_{\star}+1$ by
  \begin{equation} \label{defpoly}
    v^j := A^j v = \sum_{\nu=0}^{d_{\star}} {j\choose \nu} E^{\nu} v, \quad
    q_i(j) = (v^j)_i = \sum_{\nu=0}^{d_{\star}+1-i} { j \choose \nu} v_{i+\nu}.
  \end{equation}
  If $d_{\star}=0$ then we have $v^j=v$ for all $j\in \N_0$, hence all angles
  $\ang(v^{j},v^{j+1})=0$ vanish and do not
  contribute to the supremum in \eqref{dinner}. Therefore, we can
  assume $d _{\star}\ge1$. Let $z_{\nu} \in \C$, $\nu=1,\ldots,d_{\star}$ denote the roots of $q_1$
  (repeated according to multiplicity) and set $x_{\nu}= \mathrm{Re}(z_{\nu})$.
  Our goal is to show that there exists a constant $C_{\star}>0$ independent of
  $v$ such that for all $j \in \N_0$
  \begin{equation} \label{estsing}
    \ang(v^{j},v^{j+1}) \le
    \frac{C_{\star}}{\min_{\nu=1,\ldots,d_{\star}}|j - x_{\nu}|},
    \quad \text{if } \min_{\nu=1,\ldots,d_{\star}}|j - x_{\nu}| \ge
    1.
  \end{equation}
  Suppose this has been shown, then the set
  $M=\bigcup_{\nu=1,\ldots,d_{\star}}(x_{\nu}-1,x_{\nu}+1)$ contains at most
    $2d_{\star}$  natural numbers and \eqref{estsing} leads to the estimate
  \begin{align*}
    \frac{1}{n}\sum_{j=0}^{n-1} \ang(v^{j},v^{j+1}) & \le \frac{d_{\star} \pi}{n}
    + \frac{1}{n} \sum_{j\in \{0,\ldots,n-1\}\setminus M}\frac{C_{\star}}{
      \min_{\nu=1,\ldots,d_{\star}}|j - x_{\nu}|} \\
    & \le \frac{d \pi}{n} + \frac{C_{\star}}{n}
    \sum_{j\in \{0,\ldots,n-1\}\setminus M}\sum_{\nu=1}^{d_{\star}} \frac{1}{|j-x_{\nu}|}\\
    & \le \frac{d \pi}{n} + \frac{C_{\star}}{n}2 d (\log(n) +1).
        \end{align*}
  In the last step we used the standard estimate of the harmonic sum.
  The right-hand side is independent of $v$, taking the supremum over $v$
  and letting $n \to \infty$ shows $\bar{\theta}_1(A)=0$.

  For the proof of \eqref{estsing} let us first notice the relation
  $v^{j+1}-v^j = (A-I_d)v^j = E v^j$. By \eqref{defpoly} this leads to the
  recursion (setting $q_{d_{\star}+2}\equiv 0$)
  \begin{align} \label{polyrec}
    q_i(j+1)-q_i(j) = q_{i+1}(j), \quad j \in \N_0,\; i=1,\ldots,d_{\star}+1
  \end{align}
  and to the expression
  \begin{equation} \label{sumdiffsquare}
    \|v^{j+1}-v^j\|^2  = \sum_{i=1}^{d_{\star}+1} q_{i+1}(j)^2 = \|v^j\|^2 - q_1(j)^2
    \le \|v^j\|^2.
  \end{equation}
  If $q_1(j)\neq 0$ then Lemma \eqref{lem0} (i) applies and yields
  \begin{equation} \label{estang2}
    \begin{aligned}
      \ang(v^{j},v^{j+1}) & \le \tan \ang(v^{j},v^{j+1}) \le
      \left[\frac{\|v^j\|^2 -q_1(j)^2}{q_1(j)^2}\right]^{1/2}\\
      &= \left[\sum_{i=2}^{d_{\star}+1}
      \frac{q_i(j)^2}{q_1(j)^2} \right]^{1/2} \le \sqrt{d_{\star}}
    \max_{i=2,\ldots,d_{\star}+1}\frac{|q_i(j)|}{|q_1(j)|}.
    \end{aligned}
  \end{equation}
    In view of the recursion \eqref{polyrec} and \eqref{estang2} it is
  sufficient to prove for some constant $C_2$, independent of $v$, and
 for all $\tau=0,\ldots,d_{\star}$ the estimate
  \begin{equation} \label{estq2q1}
    \left|\frac{q_2(j+\tau)}{q_1(j)}\right| \le
    \frac{C_2}{\min_{\nu=1,\ldots,d_{\star}}|j - x_{\nu}|}, \quad
    \text{if } \min_{\nu=1,\ldots,d_{\star}}|j - x_{\nu}| \ge 1.
  \end{equation}

  From $q_1(j) =\prod_{\nu=1}^{d_{\star}}(j- z_{\nu}) $ and \eqref{polyrec} we obtain
  by expanding products
  \begin{align*}
    \left|\frac{q_2(j+\tau)}{q_1(j)}\right| &= \prod_{\nu=1}^{d_{\star}}|j- z_{\nu}|^{-1}
    \Big| \prod_{\nu=1}^{d_{\star}}(j- z_{\nu}+\tau+1) -
    \prod_{\nu=1}^{d_{\star}}(j- z_{\nu}+\tau) \Big| \\
    &= \prod_{\nu=1}^{d_{\star}}|j- z_{\nu}|^{-1} \Big| \sum_{\substack{J
      \subset \{1,\ldots,d_{\star}\}\\ |J| < d_{\star}}} \prod_{\nu
      \in J}(j-z_{\nu}) \big[ (\tau+1)^{d_{\star}-|J|}-\tau^{d_{\star}-|J|} \big] \Big|\\
    & \le  \sum_{\substack{J \subset \{1,\ldots,d_{\star}\} \\ |J| \ge 1}}
    \big[ (\tau+1)^{|J|}-  \tau^{|J|} \big]\prod_{\nu \in J}|j-z_{\nu}|^{-1}.
  \end{align*}
  Because of $|j-z_{\nu}| \ge |j-x_{\nu}|$ and $|J|\ge 1$ we have
  \begin{align*}
    \prod_{\nu \in J}|j-z_{\nu}|^{-1}& \le
                                       \frac{1}{\min_{\nu=1,\ldots,d_{\star}}|j
                                       - x_{\nu}|}, \quad \text{if }
                                       \min_{\nu=1,\ldots,d_{\star}}|j
                                       - x_{\nu}| \ge 1,
  \end{align*}
  which proves \eqref{estq2q1}.

  The proof is easily adapted to the general Jordan form \eqref{jordanform}.
  Assertion \eqref{estsing} remains the same,
  but now we have block vectors $v^j=(v^j_1,\ldots, v^j_k)^{\top}$ and
  polynomials $q_{i,\ell}$,$i=1,\ldots,d_{\ell}$, $\ell=1,\ldots,k$.
  The formula \eqref{sumdiffsquare} turns into
  \begin{align*}
    \|v^{j+1}-v^{j} \|^2 & = \|v^j\|^2 - \sum_{\ell =1}^k q_{1,\ell}(j)^2
    =\sum_{\ell=1}^k( \|v^j_{\ell}\|^2 - q_{1,\ell}(j)^2),
  \end{align*}
  and the estimate \eqref{estang2} is modified by using
  \begin{align*}
    \frac{\|v^j\|^2 - \sum_{\ell =1}^k q_{1,\ell}^2(j)}{\sum_{\ell
    =1}^k q_{1,\ell}^2(j)} & \le \sum_{\ell=1}^k
                             \frac{\|v^j_{\ell}\|^2-q_{1,\ell}(j)^2}{q_{1,\ell}(j)^2}.
  \end{align*}
  The subsequent arguments remain unchanged.
    \end{proof}

\section{Proof of the Blocking Lemma \ref{lemblock}}\label{app4}
\begin{proof}
    By scaling $A$  and \eqref{eq:split} we can arrange that
  $ |\sigma(A_s)|, |\sigma(A_u^{-1})| < q <1$. Then there exists a constant
  $C_{\star}$ such that
  \begin{align} \label{AuAs}
    \|A_u^{-j}v_u\| \le C_{\star} q^j \|v_u\|, \quad
    \|A_s^{j}v_s\| \le C_{\star} q^j \|v_s\|, \quad \forall \  v_u \in X_u, v_s \in X_s.
  \end{align}
  Let us first consider outer angular values and
  decompose $v\in \R^d$ as $v=v_s+v_u$, $v_s \in X_s,v_u \in X_u$.

  The cases $v_s=0$ resp.\ $v_u=0$ immediately show that
  $\theta_1(A) \ge \max (\theta_1(A_s),\theta_1(A_u))$ holds for
  $ \theta_1\in \{ \hat{\theta}_1,\underaccent{\hat}{\theta}_1 \}$.
  To prove the converse, we assume $v_u \neq 0$ and obtain
  from the triangle inequality
\begin{equation}\label{estAtou}
  \begin{aligned}
    &\big| \frac{1}{n} \sum_{j=1}^n \ang(A^{j-1}v,A^{j}v)-
    \frac{1}{n} \sum_{j=1}^n \ang(A_u^{j-1}v_u,A_u^{j}v_u)\big| \le
     \frac{2}{n} \sum_{j=0}^n \ang(A^jv,A_u^jv_u).
  \end{aligned}
\end{equation}
  We show that the right-hand side converges to zero as $n \to \infty$
  for every $v$. Then the $\liminf$ and the $\limsup$ of the first two
  sums in \eqref{estAtou} agree and our assertion follows by
  taking the supremum over $v$.
  With $C_{\star}$, $q$ from \eqref{AuAs} there exist an index $j_{\star}=j_{\star}(v)$
  such that
  \begin{equation} \label{condvuvs}
    2 C_{\star}^2 q^{2j} \|v_s \| \le \sqrt{3} \|v_u\|, \quad \text{for all } j\ge j_{\star}.
  \end{equation}
  The estimate \eqref{eq1.6} in Lemma \ref{lem0} then shows for $j \ge j_{\star}$
\begin{equation} \label{estAvAvu}
  \begin{aligned}
    \ang(A^jv,A_u^j v_u)&  \le \tan\ang(A^jv,A_u^j v_u)
    \le \frac{\|A_s^jv_s\|}{\left(\|A_u^jv_u\|^2 - \|A_s^jv_s\|^2
      \right)^{1/2} } \\
    & \le \frac{C_{\star} q^j
      \|v_s\|}{\left(C_{\star}^{-2}q^{-2j}\|v_u\|^2 - C_{\star}^2
        q^{2j} \|v_s\|^2
      \right)^{1/2}}
   \le 2 C_{\star}^2 q^{2j} \frac{\|v_s\|}{\|v_u\|}.
  \end{aligned}
  \end{equation}
  Since the right-hand side is summable our conclusion follows.

  Next we analyze the inner angular values. By Corollary \ref{propauto} it
  suffices to consider $\theta_1=\bar{\theta}_1$.
   As above, Definition \ref{defangularvalues} implies the estimate
   $\theta_1(A) \ge \max (\theta_1(A_s),\theta_1(A_u))$,
  and it remains to prove the converse.
    From \eqref{condvuvs} and \eqref{estAvAvu} we infer that for each
  $v=v_s+v_u$ with $v_u \neq 0$ the following index exists
  \begin{equation*} \label{defkstar}
    k_{\star} =k_{\star}(v)=  \min \{ j \in \N: \|A^j v_s\| \le \|A^jv_u\|\}.
  \end{equation*}
  Further choose $j_{\star}$ such that $2 C_{\star}^2 q^{2j_{\star}} \le \sqrt{3}$.
  Then \eqref{condvuvs} holds for $A_s^{k_{\star}}v_s,A_u^{k_{\star}}v_u$ instead
  of $v_s,v_u$ and the estimate \eqref{estAvAvu} yields
  \begin{align*}
    \ang(A^jv,A^jv_u) \le 2 C_{\star}^2 q^{2(j-k_{\star})} \quad \text{for }
    j - k_{\star} \ge j_{\star}.
  \end{align*}
  We use $\|A_u^{k_{\star}-1}v_u\| \le \|A_s^{k_{\star}-1}v_s\|$, \eqref{AuAs} and
  Lemma \ref{lem0}  to derive a corresponding estimate
   of angles to the stable part for $j \le k_{\star}-j_{\star}-1$:
  \begin{equation*}
    \begin{aligned}
      \ang(A^jv,A^jv_s) & \le \tan \ang(A^jv,A^jv_s) \\
      & \le \frac{\|A_u^{j-k_{\star}+1}(A_u^{k_{\star}-1}v_u) \|}{\left(
        \|A_s^{j-k_{\star}+1}A_s^{k_{\star}-1}v_s \|^2 -
        \|A_u^{j-k_{\star}+1}(A_u^{k_{\star}-1}v_u) \|^2 \right)^{1/2}} \\
      & \le \frac{C_{\star}q^{k_{\star}-j-1}\|A_u^{k_{\star}-1}v_u\|}
      {\left(C_{\star}^{-2}q^{-2(k_{\star}-j-1)}\|A_s^{k_{\star}-1}v_s \|^2 -
        C_{\star}^2q^{2(k_{\star}-j-1)}\|A_s^{k_{\star}-1}v_s\|^2 \right)^{1/2}}\\
      & \le \frac{C_{\star}^2 q^{2(k_{\star}-j-1)} \| A_u^{k_{\star}-1}v_u\|}
      {\left( 1 - C_{\star}^4 q^{4(k_{\star}-j-1)}\right)^{1/2}
        \|A_s^{k_{\star}-1}v_s \|} \le 2 C_{\star}^2 q^{2(k_{\star}-j-1)}.
    \end{aligned}
  \end{equation*}
  With these preparations the triangle inequality leads to
  (recall $\sum_m^n = 0$ if $m>n$)
  \begin{align*}
    \frac{1}{n} \sum_{j=1}^n \ang(A^{j-1}v,A^{j}v) & \le \frac{1}{n}
    \Big[ \big(\sum_{j=1}^{k_{\star}-1-j_{\star}} + \sum_{j=k_{\star}+1+j_{\star}}^n\big)
      \ang(A^{j-1}v,A^{j}v) + (2  j_{\star}+1) \frac{\pi}{2} \Big]\\
    & \le \frac{1}{n} \Big[ 2C_{\star}^2\big( \sum_{j=1}^{k_{\star}-1-j_{\star}}
      q^{2(k_{\star}-j-1)} + \sum_{j=k_{\star}+1+j_{\star}}^nq^{2(j-k_{\star})}\big)
      + (j_{\star}+1)\pi \\
      & + \sum_{j=1}^{\min(k_{\star},n)} \ang(A_s^{j-1}v_s,A_s^{j}v_s)
      + \sum_{j=k_{\star}+1}^n \ang(A_u^{j-1}v_u,A_u^{j}v_u)\Big].
  \end{align*}
  For any given $\varepsilon >0$, there exists $n_0 \in \N$ such that for
  all $k\ge n_0$, $v_s \in  X_s$, $v_s \neq 0$, $v_u \in X_u$, $v_u \neq 0$
  the following holds
  \begin{align*}
    \sum_{j=1}^k\ang(A_s^{j-1}v_s,A_s^{j}v_s) \le
    k(\bar{\theta}_1(A_s)+ \varepsilon), \quad
    \sum_{j=1}^k\ang(A_u^{j-1}v_u,A_u^{j}v_u) \le k(\bar{\theta}_1(A_u)+ \varepsilon).
  \end{align*}
  Thus we have for $n \ge n_0$
  \begin{align*}
    \sum_{j=1}^{\min(k_{\star},n)} \ang(A_s^{j-1}v_s,A_s^{j}v_s) \le
    \begin{cases} \min(k_{\star},n) (\bar{\theta}_1(A_s)+ \varepsilon), &
      k_{\star} \ge n_0, \\
      n_0 \frac{\pi}{2}, & k_{\star} \le n_0.
    \end{cases}
  \end{align*}
      With a similar estimate for $\sum_{j=k_{\star}+1}^n\ang(A_u^{j-1}v_u,A_u^{j}v_u)$
    we obtain for $n \ge n_0$ and  some constant $C$ independent of $v$ and $n$
   \begin{align*}
     \frac{1}{n} \sum_{j=1}^n \ang(A^{j-1}v,A^{j}v) &\le \frac{1}{n} \big[
       C +  n_0 \frac{\pi}{2} + \min(k_{\star},n)(\bar{\theta}_1(A_s)+ \varepsilon) \\
      &+  n_0 \frac{\pi}{2} + (n - \min(k_{\star},n))(\bar{\theta}_1(A_u)+ \varepsilon)  \big] \\
     & \le \max(\bar{\theta}_1(A_s),\bar{\theta}_1(A_u)) + \varepsilon
     + \frac{1}{n}(C+ n_0 \pi).
   \end{align*}
   Now take the supremum over $v\in \R^d, v_u \neq 0$ and then
   $n$ large so that the last summand is less than $\varepsilon$.
   This finishes the proof of \eqref{eqreduce}.
\end{proof}

\end{document}